\documentclass{lms}

\usepackage{amssymb}
\usepackage{amsmath}
\usepackage{arydshln}

\newtheorem{theorem}{Theorem}[section] 
\newtheorem{lemma}[theorem]{Lemma}     
\newtheorem{corollary}[theorem]{Corollary}
\newtheorem{proposition}[theorem]{Proposition}

\newnumbered{assertion}{Assertion}    
\newnumbered{conjecture}{Conjecture}  
\newnumbered{definition}{Definition}
\newnumbered{hypothesis}{Hypothesis}
\newnumbered{remark}{Remark}
\newnumbered{note}{Note}
\newnumbered{observation}{Observation}
\newnumbered{problem}{Problem}
\newnumbered{question}{Question}
\newnumbered{algorithm}{Algorithm}
\newnumbered{example}{Example}
\newunnumbered{notation}{Notation} 

\title[Quaternion Derivatives: The GHR Calculus]
 {Quaternion Derivatives: The GHR Calculus} 

\author{Dongpo Xu, Cyrus Jahanchahi, Clive C. Took and Danilo P. Mandic}

\classno{46G05, 46T20 (primary), 58C20, 58C25 (secondary)}

\extraline{This work was completed on 19 Aug 2013 and was submitted to PLMS on 08 Jan 2014.}

\begin{document}
\maketitle

\begin{abstract}
Quaternion derivatives in the mathematical literature are typically defined only
for analytic (regular) functions. However, in engineering problems, functions of interest are often real-valued and thus not analytic, such as the standard cost function. The HR calculus is a convenient way to calculate formal derivatives of both analytic and non-analytic functions of quaternion variables, however, both the HR and other functional calculus in quaternion analysis have encountered an essential technical obstacle, that is, the traditional product rule is invalid due to the non-commutativity of the quaternion algebra. To address this issue, a generalized form of the HR derivative is proposed based on a general orthogonal system. The so introduced generalization, called the generalized HR (GHR) calculus, encompasses not just the left- and right-hand versions of quaternion derivative, but also enables solutions to some long standing problems, such as the novel product rule, the chain rule, the mean-valued theorem and Taylor's theorem. At the core of the proposed approach is the quaternion rotation, which can naturally be applied to other functional calculi in non-commutative settings. Examples on using the GHR calculus in adaptive signal processing support the analysis.
\end{abstract}


\section{Introduction.}
Quaternions have become a standard in physics \cite{Girard}, computer graphics \cite{Hanson}, and have also been successfully applied to
many signal processing and communications problems  \cite{Took09,Took10a,Took10b,Khong13,Bihan,Bihan06,Bihan14,Buchholz,Sangwine12,Sangwine11,Miron}.
One attractive property is that quaternion algebra \cite{Hamilton} reduces the number of parameters, and offers improvements in terms of computational complexity \cite{Neto} and functional simplicity \cite{Jahanchahi13}. Often, the task is to find the values of quaternion parameters which optimize a chosen objective function. To solve this kind of optimization problems, a common approach is to build adaptive optimization algorithms based on the gradient of the objective function, as in the quaternion least mean square (QLMS) adaptive filter \cite{Took09}. However, a confusing aspect of QLMS adaptive filtering, and other gradient-based optimization procedures,
is that the objective functions of interest are often real-valued and thus not analytic according to the analyticity condition in quaternion analysis \cite{Sudbery,Deavours,Khalek,Fueter34,Fueter35,Leo}. An alternative way to find the derivative of real functions of quaternion variables is therefore needed. Following the idea of the CR calculus in the complex domain \cite{Hjorungnes07,Brandwood,Kreutz,Wirtinger}, two alternative ways can be used to find the derivative of a real function $f$ with respect to the unknown quaternion variable $q$. The first way, called the pseudo-derivative, rewrites $f$ as a function of the four real components $q_a,q_b,q_c$ and $q_d$ of the quaternion variable $q$, and then find the real derivatives
of the so rewritten function with respect to the independent real variables, $q_a,q_b,q_c$ and $q_d$, separately. In this way, we can treat the real-valued function $f$ as a real differentiable mapping between $\mathbb{R}^4$ and $\mathbb{R}$. The second way, called the HR calculus, is more elegant \cite{Jahanchahi10,Mandic11}, and aims to find the formal derivatives of $f$ with respect to the quaternion variables $q,q^i,q^j,q^k$ and their conjugates. The differentials of these quaternion variables are independent, and shown in Lemma \ref{lm:dpindp} and Lemma \ref{lm:dpconjindp}.

In this paper,  motivated by the CR calculus \cite{Hjorungnes11,Mandic09,Kreutz}, we revisit the theory of HR calculus introduced by three of the authors \cite{Mandic11}, and further extend this theory by developing the product rule and chain rule for the HR calculus. However, as shown in Section \ref{sec:limithr}, the traditional product rule is not suitable for
the HR calculus due to the non-commutativity of quaternion algebra. Other functional calculi \cite{Sudbery,Deavours,Gentili06,Khalek,Leo} in quaternion analysis are similarly suffering from this barrier.

To this end, we firstly generalize the HR calculus based on a general orthogonal system. The generalized HR (GHR) calculus encompasses not just the left- and right-hand versions of quaternion derivative, we also show that for the two versions of the HR derivative, their results are identical for real-valued functions. One major result is therefore that using the GHR calculus, it is no longer important which version of the HR derivative is used. Also, within the GHR framework, we introduce a novel product rule to facilitate the calculation of the HR derivatives of general functions of quaternion variables, and show that
if one function of the product is  real-valued, this novel product rule degenerates into the traditional product rule shown in Corollary in \ref{cor:leftprodrl}. The core of the novel product rule is the quaternion rotation, this idea can be also naturally applied to other functional calculus in non-commutative settings. In the process of refining the theory of HR calculus, we revisit two important and fundamental theorems: the mean value theorem and Taylor's theorem. Taylor's theorem is presented in a compact and familiar form involving the HR derivatives. The GHR calculus poses an answer to an long-standing mathematical problem \cite{Schwartz}, while illustrative examples show how it can be applied as an important tool for solving problems
in signal processing and communications.

\subsection{Quaternion Algebra.}
Quaternions are an associative but not commutative algebra
over $\mathbb{R}$, defined as
\begin{equation}
\mathbb{H}=\textrm{span}\{1,i,j,k\}\triangleq \{q_a+iq_b+jq_c+kq_d\;|\;q_a,q_b,q_c,q_d\in\mathbb{R}\}
\end{equation}
where $\{1,i,j,k\}$ is a basis of $\mathbb{H}$, and the imaginary units $i,j$ and $k$ satisfy $i^2=j^2=k^2=ijk=-1$, which implies $ij=k=-ji$, $jk=i=-kj$, $ki=j=-ik$. For any quaternion
\begin{equation}\label{eq:qqabcd}
q=q_a+iq_b+jq_c+kq_d=Sq+Vq
\end{equation}
the scalar (real) part is denoted by $q_a=Sq=\mathfrak{R}(q)$, whereas the vector part $Vq=\mathfrak{I}(q)=iq_b+jq_c+kq_d$ comprises the three imaginary parts. The quaternion
product for $p,q \in \mathbb{H}$ is given by
\begin{equation}
pq=SpSq-Vp \cdot Vq+SpVq+SqVp+Vp\times Vq
\end{equation}
where the symbols $'\cdot'$ and $'\times'$ denote the usual inner product and vector product, respectively.
The presence of the vector product causes the quaternion product to be noncommutative, i.e., for $p,q\in\mathbb{H}$, $pq\neq qp$ in general. The conjugate of a quaternion $q$ is defined as $q^*=Sq-Vq$, while the conjugate of the product satisfies $(pq)^*=q^*p^*$. The modulus of a quaternion is defined as $|q|=\sqrt{qq^*}=\sqrt{q_a^2+q_b^2+q_c^2+q_d^2}$, and it is
easy to check that $|pq|=|p||q|$. The inverse of a quaternion $q\neq 0$ is $q^{-1}=q^*/|q|^2$ which yields an important consequence
\begin{equation}
(pq)^{-1}=\frac{(pq)^*}{|pq|^2}=\frac{q^*p^*}{|q|^2|p|^2}=\frac{q^*}{|q|^2}\frac{p^*}{|p|^2}=q^{-1}p^{-1}
\end{equation}
(note the change in order).

If $|q| = 1$, we call $q$ a \textit{unit} quaternion. A quaternion $q$ is said to be \textit{pure} if $\mathfrak{R}(q)=0$. Then
$q^*=-q$ and $q^2=-|q|^2$. Thus, a \textit{pure unit} quaternion is a square root of -1, such as imaginary units $i,j$ and $k$.

\subsection{Analytic Functions in $\mathbb{H}$.} A function that is analytic is also called regular, or monogenic. Due to the non-commutativity of quaternion products, there are two ways to write the quotient in the definition of quaternion derivative, as shown below.
\begin{proposition}[\cite{Gentili09}]\label{pr:tradefideri}
Let $D\subseteq \mathbb{H}$ be a simply-connected domain of definition of the function $f:D\rightarrow \mathbb{H}$. If for any $q\in D$
\begin{equation}\label{eq:leftlimderiv}
\lim_{h\rightarrow 0}[\left(f(q+h)-f(q)\right)h^{-1}]
\end{equation}
exists in $\mathbb{H}$, then necessarily $f(q)=\omega q +\lambda$ for some $\omega,\lambda \in \mathbb{H}$. If for any $q\in D$
\begin{equation}\label{eq:rtlimderiv}
\lim_{h\rightarrow 0}[h^{-1}\left(f(q+h)-f(q)\right)]
\end{equation}
exists in $\mathbb{H}$, then necessarily $f(q)=q \nu +\lambda$ for some $\nu,\lambda \in \mathbb{H}$.
\end{proposition}

Proposition \ref{pr:tradefideri} is discussed in detail in \cite{Sudbery,Luna} and indicates that the traditional definitions of derivative in \eqref{eq:leftlimderiv} and \eqref{eq:rtlimderiv} are too restrictive. One attempt to relax this constraint is due to Feuter \cite{Fueter34,Fueter35}, summarized in \cite{Sudbery,Deavours}. Feuter's analyticity condition is termed the Cauchy-Riemann-Fueter (CRF) equation, given by

\begin{equation}\label{eq:crf}
\begin{split}
{\rm Left\;CRF:}\quad \frac{\partial f}{\partial q_a}+\frac{\partial f}{\partial q_b}i+\frac{\partial f}{\partial q_c}j+\frac{\partial f}{\partial q_d}k=0\\
{\rm Right\;CRF:}\quad \frac{\partial f}{\partial q_a}+i\frac{\partial f}{\partial q_b}+j\frac{\partial f}{\partial q_c}+k\frac{\partial f}{\partial q_d}=0
\end{split}
\end{equation}
The limitations of the CRF condition were pointed out by Gentili and Struppa in \cite{Gentili06,Gentili07}, illuminating that the polynomial functions (even the identity $f(q)=q$) satisfy neither the left CRF nor the right CRF. To further relax the analyticity condition, a \textit{local}  analyticity condition (LAC) was proposed in \cite{Leo}, by using the polar form of a quaternion to give
\begin{equation}\label{eq:lac}
\begin{split}
{\rm Left\;LAC:}\quad \frac{\partial f}{\partial q_a}+\left(q_b\frac{\partial f}{\partial q_a}+q_c\frac{\partial f}{\partial q_c}+q_d\frac{\partial f}{\partial q_d}\right)\frac{V_q}{|V_q|^2}=0 \\
{\rm Right\;LAC:}\quad \frac{\partial f}{\partial q_a}+\frac{V_q}{|V_q|^2}\left(q_b\frac{\partial f}{\partial q_a}+q_c\frac{\partial f}{\partial q_c}+q_d\frac{\partial f}{\partial q_d}\right)=0
\end{split}
\end{equation}
where $q =q_a+\hat{q}\alpha$, $\hat{q}=\frac{V_q}{|V_q|}$ and $V_q=iq_b+jq_c+kq_d$. This theory of local analyticity is now very well developed, and in many different directions, and we refer the reader to \cite{Colombo09,Gentili06,Gentili07} for
the slice regular functions. More recent work in this area includes \cite{Colombo,Gentili09,Gentili}, and references therein.
The advantage of the local analyticity condition is that both the polynomial functions of $q$ and some elementary functions in Section \ref{sec:elemfuns} satisfy the left LAC or the right LAC.

\begin{remark}
Note that the product and composition of two LAC functions $f$ and $g$ generally no longer meet the local analytic condition.
For example, if $f(q)=q$ and $g(q)=\omega q$, $\omega\in\mathbb{H}$, then $f$ and $g$ satisfy the left LAC,
but the product $fg=q\omega q$ does not satisfy the left LAC. It is the same situation for the right LAC,  only we need to write the function $g$ as $g(q)= q \omega$.
\end{remark}

The quaternion derivative in quaternion analysis is defined only
for analytic functions. However, in engineering problems, objective functions of interest are often real-valued to minimize or maximize them
and thus not analytic, such as
\begin{equation}
f(q)=|q|^2=qq^*
\end{equation}
Notice that if the definition of the analytic (regular) function given in \cite{Sudbery,Deavours,Colombo09,Gentili06,Gentili07,Khalek,Leo} is used, then the
function $f$ is not analytic.  In order to take the derivative of these functions (but not limited to only such functions), the HR calculus extends the classical idea of complex CR calculus \cite{Hjorungnes07,Brandwood,Kreutz,Wirtinger} to the quaternion field \cite{Mandic11}. This generalization is not trivial, and we show that many rules of the CR and HR calculus are different. The details are given in Section \ref{sec:hr}.

\begin{remark}
It is important to note that the left (right) terminology in \eqref{eq:crf}, \eqref{eq:lac} and below differ from those in \cite{Sudbery,Deavours,Colombo09,Gentili06,Khalek,Leo}. In this paper, the standard of left (right) is based on the position of $\frac{\partial f}{\partial q_a}$, $\frac{\partial f}{\partial q_b}$, $\frac{\partial f}{\partial q_c}$ and $\frac{\partial f}{\partial q_d}$, rather than on the positions of imaginary units $i,j,k$. Although
this is only a notational difference, we later show that
the left derivatives (named based on this standard) in Definition \ref{def:lefthr} and \ref{def:leftghr} result in a left constant rule \eqref{pr:lefthrpr2} and \eqref{pr:leftghrpr1}, that is, the \textit{left} constant can come out from the \textit{left} derivative of product, and the left derivatives stand on the left side of the quaternion differential in \eqref{eq:case1dfdqcmp} and \eqref{eq:case1bdfdqcmp}. This allows for a consistent use of terminology.
\end{remark}

\subsection{Quaternion Rotation.}
Every quaternion can be written in the \textit{polar} form
\begin{equation}
q=Sq+Vq =|q|\left(\frac{Sq}{|q|}+\frac{Vq}{|Vq|}\frac{|Vq|}{|q|}\right)=|q|(\cos \theta + \hat{q} \sin \theta )
\end{equation}
where $\hat{q}=Vq/|Vq|$ is a pure unit quaternion and $\theta=\arccos(S_q/|q|)$ is the angle (or argument)
of the quaternion. We now introduce the rotation and involution operations.
\begin{definition}[(Quaternion Rotation {\cite{Ward}})]\label{def:qrot}
For any quaternion number q, consider the transformation
\begin{equation*}
\phi_{\mu}(q)\equiv q^{\mu}\triangleq \mu q \mu^{-1}
\end{equation*}
where $\mu=|\mu|(\cos \theta + \hat{\mu} \sin \theta )\in \mathbb{H}$. This transformation geometrically describes a 3-dimensional rotation of the vector part of q by the angle $2\theta$ about to the vector part of $\mu$.
\end{definition}

The properties of the operator in Definition \ref{def:qrot} (see \cite{Bihan06,Buchholz}) are:
\begin{equation}\label{pr:qrotpr1}
\mu q^{(\mu^*)}=\mu \mu^* q (\mu^*)^{-1}=|\mu|^2q \frac{\mu}{|\mu|^2}=q\mu,\quad   \forall \mu \in\mathbb{H}
\end{equation}
\begin{equation}
q^{(\mu^{-1})}=\mu^{-1} q \mu=\frac{\mu^*q \mu}{|\mu|^2} =\mu^* q (\mu^*)^{-1}=q^{(\mu^*)},\quad    \forall \mu \in\mathbb{H}
\end{equation}
\begin{equation}\label{pr:pqmu}
(pq)^{\mu}=\mu pq \mu^{-1}=\mu p \mu^{-1}\mu q \mu^{-1}=p^{\mu}q^{\mu},\quad   \forall p, q\in\mathbb{H}
\end{equation}
\begin{equation}\label{pr:def1qmunu}
q^{\mu \nu}=({\mu}{\nu})q({\mu}{\nu})^{-1}= {\mu}({\nu}q{\nu}^{-1} ) {\mu}^{-1}=(q^{\nu} )^{\mu},\quad
 \forall \nu, \mu \in\mathbb{H}
\end{equation}
\begin{equation}
q^{\mu*} \equiv (q^*)^{\mu}=\mu q^* \mu^{-1}=(\mu q \mu^{-1})^*=(q^{\mu})^*\equiv q^{*\mu},\quad   \forall q,\mu\in\mathbb{H}
\end{equation}
(note the change in order and $q^{\mu*}\neq q^{(\mu^*)}$).

\begin{definition}[(Quaternion Involution {\cite{Ell}})]\label{def:qinv}
The involution of a
quaternion over a pure unit quaternion $\eta$ is given by
\begin{equation*}
\phi_{\eta}(q)\equiv q^{\eta}= \eta q \eta^{-1}=\eta q \eta^*=-\eta q \eta
\end{equation*}
and represents a rotation by $\pi$ about $\eta$.
\end{definition}

The properties of quaternion involution are:
\begin{equation}
(q^{\eta})^{\eta}=q^{\eta\eta}=q^{(-1)}=q,\quad   \forall q\in\mathbb{H}
\end{equation}
\begin{equation}
q^{(\eta^{*})}=q^{-\eta}=(-\eta)q(-\eta)^{-1}=q^{\eta},\quad   \forall q\in\mathbb{H}
\end{equation}

\begin{definition}[(Quaternion Reflection {\cite{Ward}})]
The reflection of a
quaternion about a pure unit quaternion $\eta$ is given by
\begin{equation*}
\psi_{\eta}(q)\equiv\phi_{\eta}(-q)=\eta q \eta
\end{equation*}
and the reflection is performed in the plane whose normal is $\eta$.
\end{definition}

We should point out that the real representation in (1) can be
easily extended to other orthogonal bases. In particular, for any non-zero quaternion $\mu=\mu_a+i\mu_b+j\mu_c+k\mu_d$,
we consider an orthogonal system $\{1,i^{\mu},j^{\mu},k^{\mu} \}$ given by \cite{Ward}
\begin{equation}
\begin{split}
&i^{\mu}=\frac{1}{|\mu|^2}\left((\mu_a^2+\mu_b^2-\mu_c^2-\mu_d^2 )i+(2\mu_a \mu_d+2\mu_b \mu_c )j+(2\mu_b \mu_d -2\mu_a \mu_c)k\right)\\
&j^{\mu}=\frac{1}{|\mu|^2}\left((2\mu_b \mu_c -2\mu_a \mu_d)i+(\mu_a^2+\mu_c^2-\mu_b^2-\mu_d^2 )j+(2\mu_c \mu_d+2\mu_a \mu_b )k\right)\\
&k^{\mu}=\frac{1}{|\mu|^2}\left((2\mu_a \mu_c+2\mu_b \mu_d )i+(2\mu_c \mu_d -2\mu_a \mu_b)j+(\mu_a^2+\mu_d^2-\mu_b^2-\mu_c^2 )k\right)
\end{split}
\end{equation}
so that the matrix representation of the map $(\cdot)^{\mu}$ becomes
\begin{equation}
{\bf M}=\frac{1}{|\mu|^2}\left(\begin{array}{ccc}
\mu_a^2+\mu_b^2-\mu_c^2-\mu_d^2  & 2\mu_a \mu_d+2\mu_b \mu_c  & 2\mu_b \mu_d -2\mu_a \mu_c \\
2\mu_b \mu_c -2\mu_a \mu_d & \mu_a^2+\mu_c^2-\mu_b^2-\mu_d^2  & 2\mu_c \mu_d+2\mu_a \mu_b\\
2\mu_a \mu_c+2\mu_b \mu_d & 2\mu_c \mu_d -2\mu_a \mu_b  & \mu_a^2+\mu_d^2-\mu_b^2-\mu_c^2
\end{array}\right)
\end{equation}
It is easily shown that ${\bf M}$ is orthogonal: ${\bf M}{\bf M}^T={\bf I}_3$ and ${\rm det}({\bf M})=1$, so that the linear map  $q^{\mu}=\mu q \mu^{-1}$ represents a rotation in $\mathbb{R}^3$,  which implies
\begin{equation}\label{pr:prodijkmurl}
i^{\mu}i^{\mu}=j^{\mu}j^{\mu}=k^{\mu}k^{\mu}=i^{\mu}j^{\mu}k^{\mu}=-1
\end{equation}
Thus, any quaternion $q$ can be alternatively represented in the $\mu$\textit{-basis} as
\begin{equation}
q=q_a+q_bi+q_cj+q_dk=q_a+q_b'i^{\mu}+q_c'j^{\mu}+q_d'k^{\mu}
\end{equation}
where $(q_b',q_c',q_d')=(q_b,q_c,q_d){\bf M}^T$.

\subsection{The Equivalence Relations and Involutions}\label{sec:equivrelation}

Given a complex number $z=z_a+iz_b$, its real and imaginary part
can be extracted as $z_a=\frac{1}{2}(z+z^*)$ and $z_b=\frac{1}{2i}(z-z^*)$ \cite{Moreno08}. Such convenient manipulation offers a number of advantages, but is not possible to achieve in the quaternion domain.
To deal with this issue, we employ the quaternion involutions (self-inverse mappings), given by \cite{Ell}
\begin{equation}\label{eq:qijkrot}
\begin{split}
&q^i=-iqi=q_a+iq_b-jq_c-kq_d,\quad q=q_a+iq_b+jq_c+kq_d\\
&q^j=-jqj=q_a-iq_b+jq_c-kq_d,\quad q^k=-kqk=q_a-iq_b-jq_c+kq_d
\end{split}
\end{equation}
and their conjugate involutions given by
\begin{align}\label{eq:qijkconjrot}
\begin{split}
&q^{i*}=q_a-iq_b+jq_c+kq_d,\quad q^*=q_a-iq_b-jq_c-kq_d\\
&q^{j*}=q_a+iq_b-jq_c+kq_d,\quad q^{k*}=q_a+iq_b+jq_c-kq_d
\end{split}
\end{align}
In this way, the four real components of the quaternion $q$ can
now be computed based on \eqref{eq:qijkrot} or \eqref{eq:qijkconjrot} as \cite{Sudbery,Took11,Mandic11,Gentili}
\begin{equation}\label{eq:z1qlink}
\begin{split}
&q_a=\frac{1}{4}(q+q^i+q^j+q^k),\quad  q_b=-\frac{i}{4}(q+q^i-q^j-q^k)\\
&q_c=-\frac{j}{4}(q-q^i+q^j-q^k),\quad  q_d=-\frac{k}{4}(q-q^i-q^j+q^k)
\end{split}
\end{equation}
\begin{equation}\label{eq:z1conjqlink}
\begin{split}
&q_a=\frac{1}{4}(q^*+q^{i*}+q^{j*}+q^{k*}),\quad  q_b=\frac{i}{4}(q^*+q^{i*}-q^{j*}-q^{k*})\\
&q_c=\frac{j}{4}(q^*-q^{i*}+q^{j*}-q^{k*}),\quad  q_d=\frac{k}{4}(q^*-q^{i*}-q^{j*}+q^{k*})
\end{split}
\end{equation}
This allows for any quaternion function of the four real variables $q_a,q_b,q_c,q_d$ to be
expressed as a function of the quaternion variables $\{q,q^{i},q^{j},q^{k}\}$
or $\{q^{*},q^{i*},q^{j*},q^{k*}\}$, whereby, the relationship between the involutions in \eqref{eq:qijkrot} and conjugate involutions in \eqref{eq:qijkconjrot} is given by
\begin{equation}\label{eq:linkqconjq}
\begin{split}
&q^*=\frac{1}{2}\left(-q+q^i+q^j+q^k\right),\quad q^{i*}=\frac{1}{2}\left(q-q^i+q^j+q^k\right)\\
&q^{j*}=\frac{1}{2}\left(q+q^i-q^j+q^k\right),\quad q^{k*}=\frac{1}{2}\left(q+q^i+q^j-q^k\right)
\end{split}
\end{equation}
\begin{equation}\label{eq:linkq2qconj}
\begin{split}
&q=\frac{1}{2}\left(-q^*+q^{i*}+q^{j*}+q^{k*}\right),\quad q^{i}=\frac{1}{2}\left(q^*-q^{i*}+q^{j*}+q^{k*}\right)\\
&q^{j}=\frac{1}{2}\left(q^*+q^{i*}-q^{j*}+q^{k*}\right),\quad q^{k}=\frac{1}{2}\left(q^*+q^{i*}+q^{j*}-q^{k*}\right)
\end{split}
\end{equation}

\begin{remark}
Observe that $q$ and $q^*$ are not independent and thus $\frac{\partial q^*}{\partial q}\neq 0$ as shown in \eqref{pr:lefthrpr1}.
This is the main difference from the CR calculus, where the derivative $\frac{\partial z^*}{\partial z} = 0$.
\end{remark}

\subsection{Results Used to Introduce GHR Derivatives}
The quaternion components, that is, the real variables $q_a,q_b,q_c$ and $q_d$ are mutually independent and hence so
are their differentials. Although the quaternion variables $q,q^{i},q^{j}$ and $q^{k}$ are related, it is important to notice that their differentials are linearly independent, similar to the CR calculus \cite{Hjorungnes11}. This condition is very important for distinguishing the GHR derivatives
from the quaternion differential of the function under consideration.
\begin{lemma}\label{lm:dpindp}
Let $f_n:\mathbb{H}\rightarrow\mathbb{H}$, $(n=1,2,3,4)$ be any arbitrary quaternion-valued functions. If the left case
\begin{equation}\label{eq:leftcsindp}
f_1dq^{\mu}+f_2dq^{\mu i}+f_3dq^{\mu j}+f_4dq^{\mu k}=0
\end{equation}
or the right case
\begin{equation}
dq^{\mu}f_1+dq^{\mu i}f_2+dq^{\mu j}f_3+dq^{\mu k}f_4=0
\end{equation}
for $\forall \mu \in \mathbb{H},\mu\neq 0$, then $f_n=0$ for $n\in\{1,2,3,4\}$.
\end{lemma}

\begin{proof}
The Left Case: By applying the rotation transformation on  both sides of \eqref{eq:qqabcd} and \eqref{eq:qijkrot}, it follows that
\begin{align}\label{eq:qijkmurot}
\begin{split}
&q^{\mu}=q_a+i^{\mu}q_b+j^{\mu}q_c+k^{\mu}q_d,\quad q^{\mu i}=q_a+i^{\mu}q_b-j^{\mu}q_c-k^{\mu}q_d\\
&q^{\mu j}=q_a-i^{\mu}q_b+j^{\mu}q_c-k^{\mu}q_d,\quad q^{\mu k}=q_a-i^{\mu}q_b-j^{\mu}q_c+k^{\mu}q_d
\end{split}
\end{align}
By applying the differential operator to the above expressions, and substituting $dq^{\mu},dq^{\mu i},dq^{\mu j}$ and $dq^{\mu k}$  into \eqref{eq:leftcsindp}, we have
\begin{equation}
\begin{split}
&f_1(dq_a+i^{\mu}dq_b+j^{\mu}dq_c+k^{\mu}dq_d)+f_2\left(dq_a+i^{\mu}dq_b-j^{\mu}dq_c-k^{\mu}dq_d\right)\\
&+f_3\left(dq_a-i^{\mu}dq_b+j^{\mu}dq_c-k^{\mu}dq_d\right)+f_4\left(dq_a-i^{\mu}dq_b-j^{\mu}dq_c+k^{\mu}dq_d\right)=0
\end{split}
\end{equation}
This is equivalent to
\begin{equation}
\begin{split}
&(f_1+f_2+f_3+f_4)dq_a+\left(f_1+f_2-f_3-f_4\right)i^{\mu}dq_b\\
&+\left(f_1-f_2+f_3-f_4\right)j^{\mu}dq_c+\left(f_1-f_2-f_3+f_4\right)k^{\mu}dq_d=0
\end{split}
\end{equation}
Since the differentials $dq_a,dq_b,dq_c$ and $dq_d$ are independent, we have
\begin{equation}
\begin{split}
f_1+f_2+f_3+f_4=0,\quad f_1+f_2-f_3-f_4=0\\
f_1-f_2+f_3-f_4=0,\quad f_1-f_2-f_3+f_4=0
\end{split}
\end{equation}
Hence, it follows that $f_1=f_2=f_3=f_4=0$. The right case can be proved in a similar way.
\end{proof}

The next lemma enables us to identify the conjugate GHR derivatives, and its proof is essentially the same as that of Lemma \ref{lm:dpindp}, so it is omitted.
\begin{lemma}\label{lm:dpconjindp}
Let $f_n:\mathbb{H}\rightarrow\mathbb{H}$, $(n=1,2,3,4)$ be any arbitrary quaternion-valued function. If the left case
\begin{equation*}
f_1dq^{\mu*}+f_2dq^{\mu i*}+f_3dq^{\mu j*}+f_4dq^{\mu k*}=0
\end{equation*}
or the right case
\begin{equation*}
dq^{\mu*}f_1+dq^{\mu i*} f_2+dq^{\mu j*}f_3+dq^{\mu k*}f_4=0
\end{equation*}
for $\forall \mu \in \mathbb{H},\mu\neq 0$, then $f_n=0$ for $n\in\{1,2,3,4\}$.
\end{lemma}

\section{The HR calculus}\label{sec:hr}

The optimization problems in quaternion variables frequently arise in engineering applications such as control theory, signal processing, and
electrical engineering. Solutions often requires a first- or second-order approximation
of the objective function to generate a new search direction. However, real functions of quaternion variables are essentially non-analytic.
The recently proposed HR calculus solves these issues by using the quaternion involution, and we now introduce two kinds of HR derivatives (the derivation of HR calculus is given in Appendix A).
\begin{definition}[(The Left HR Derivatives)]\label{def:lefthr}
Let $q=q_a+iq_b+jq_c+kq_d$, where $q_a,q_b,q_c,q_d\in\mathbb{R}$, then the formal left HR derivatives, with respect to $\{q,q^i,q^j,q^k\}$ and $\{q^*,q^{i*},q^{j*},q^{k*}\}$ of the function $f$, are defined as
\begin{align*}
{\large
\left(\begin{array}{cc}
\frac{\partial f}{\partial q},&\frac{\partial f}{\partial q^*}\\
\frac{\partial f}{\partial q^i},&\frac{\partial f}{\partial q^{i*}}\\
\frac{\partial f}{\partial q^j},&\frac{\partial f}{\partial q^{j*}}\\
\frac{\partial f}{\partial q^k},&\frac{\partial f}{\partial q^{k*}}
\end{array}\right)^T}=\frac{1}{4}
{\large
\left(\begin{array}{cc}
\frac{\partial f}{\partial q_a},&\frac{\partial f}{\partial q_a}\\
\frac{\partial f}{\partial q_b},&-\frac{\partial f}{\partial q_b}\\
\frac{\partial f}{\partial q_c},&-\frac{\partial f}{\partial q_c}\\
\frac{\partial f}{\partial q_d},&-\frac{\partial f}{\partial q_d}
\end{array}\right)^T}
\left(\begin{array}{cccc}
1 & 1  & 1 & 1 \\
-i & -i & i  & i\\
-j & j  & -j & j\\
-k & k  & k  & -k\\
\end{array}\right)
\end{align*}
where $\frac{\partial f}{\partial q_a}$, $\frac{\partial f}{\partial q_b}$, $\frac{\partial f}{\partial q_c}$ and $\frac{\partial f}{\partial q_d}$ are the partial derivatives of $f$ with respect to $q_a$, $q_b$, $q_c$ and $q_d$, respectively.
\end{definition}
The properties of Definition \ref{def:lefthr} are:
\begin{align}\label{pr:lefthrpr1}
\begin{split}
\frac{\partial  q}{\partial q}=\frac{\partial  q^*}{\partial q^*}=1,\quad &\frac{\partial  q}{\partial q^{\eta}}=\frac{\partial  q^*}{\partial q^{\eta*}}=0\\
\frac{\partial  q}{\partial q*}=\frac{\partial  q^*}{\partial q}=-\frac{1}{2},\quad &\frac{\partial  q}{\partial q^{\eta*}}=\frac{\partial  q^*}{\partial q^{\eta}}=\frac{1}{2}
\end{split}\quad \forall\eta \in \{i,j,k\}
\end{align}
\begin{align}\label{pr:lefthrpr2}
\begin{split}
\left(\frac{\partial f}{\partial q}\right)^{\eta}=\frac{\partial f^{\eta}}{\partial q^{\eta}},
\quad &\left(\frac{\partial f}{\partial q^{\eta}}\right)^{\eta}=\frac{\partial f^{\eta}}{\partial q}\\
\frac{\partial (\eta f)}{\partial q}=\eta\frac{\partial f}{\partial q}=\frac{\partial f^{\eta}}{\partial q^{\eta}}\eta,
\quad &\frac{\partial (f\eta)}{\partial q}=\frac{\partial f}{\partial q^{\eta}}\eta=\eta\frac{\partial f^{\eta}}{\partial q}
\end{split}\quad \forall\eta \in \{1,i,j,k\}
\end{align}
It is important to note that if a function $f$ is premultiplied by a constant $\eta$ in the second line of \eqref{pr:lefthrpr2}, then the derivative of the product is equal to the derivative of $f$ premultiplied by the constant, but not for postmultiplication. In other words, the left constant $\eta$ can come out from the derivative of the product, which is the reason we call Definition \ref{def:lefthr} the left HR derivative.

\begin{definition}[(The Right HR Derivatives \cite{Mandic11})]\label{def:righthr}
 Let $q=q_a+iq_b+jq_c+kq_d$, where $q_a,q_b,q_c,q_d\in\mathbb{R}$, then the formal right HR derivatives, with respect to $\{q,q^i,q^j,q^k\}$ and $\{q^*,q^{i*},q^{j*},q^{k*}\}$ of the function $f$, are defined as
\begin{align*}
{\large
\left(\begin{array}{cc}
\frac{\partial_r f}{\partial q},&\frac{\partial_r f}{\partial q^*}\\
\frac{\partial_r f}{\partial q^i},&\frac{\partial_r f}{\partial q^{i*}}\\
\frac{\partial_r f}{\partial q^j},&\frac{\partial_r f}{\partial q^{j*}}\\
\frac{\partial_r f}{\partial q^k},&\frac{\partial_r f}{\partial q^{k*}}
\end{array}\right)}=\frac{1}{4}
\left(\begin{array}{cccc}
1 & -i  & -j & -k \\
1 & -i & j  & k\\
1 & i  & -j & k\\
1 & i  & j  & -k
\end{array}\right)
{\large
\left(\begin{array}{cc}
\frac{\partial f}{\partial q_a},&\frac{\partial f}{\partial q_a}\\
\frac{\partial f}{\partial q_b},&-\frac{\partial f}{\partial q_b}\\
\frac{\partial f}{\partial q_c},&-\frac{\partial f}{\partial q_c}\\
\frac{\partial f}{\partial q_d},&-\frac{\partial f}{\partial q_d}
\end{array}\right)}
\end{align*}
where $\frac{\partial f}{\partial q_a}$, $\frac{\partial f}{\partial q_b}$, $\frac{\partial f}{\partial q_c}$ and $\frac{\partial f}{\partial q_d}$ are the partial derivatives of $f$ with respect to $q_a$, $q_b$, $q_c$ and $q_d$, respectively.
\end{definition}
The properties of Definition \ref{def:righthr} are:
\begin{align}\label{pr:righthrpr1}
\begin{split}
\frac{\partial_r  q}{\partial q}=\frac{\partial_r  q^*}{\partial q^*}=1,\quad &\frac{\partial_r  q}{\partial q^{\eta}}=\frac{\partial_r  q^*}{\partial q^{\eta*}}=0\\
\frac{\partial_r  q}{\partial q*}=\frac{\partial_r  q^*}{\partial q}=-\frac{1}{2},\quad &\frac{\partial_r  q}{\partial q^{\eta*}}=\frac{\partial_r  q^*}{\partial q^{\eta}}=\frac{1}{2}
\end{split}\quad \forall\eta \in \{i,j,k\}
\end{align}
\begin{align}\label{pr:righthrpr2}
\begin{split}
\left(\frac{\partial_r f}{\partial q}\right)^{\eta}=\frac{\partial_r f^{\eta}}{\partial q^{\eta}},
\quad &\left(\frac{\partial_r f}{\partial q^{\eta}}\right)^{\eta}=\frac{\partial_r f^{\eta}}{\partial q}\\
\frac{\partial_r (f\eta)}{\partial q}=\frac{\partial_r f}{\partial q}\eta=\eta\frac{\partial_r f^{\eta}}{\partial q^{\eta}}
\quad &\frac{\partial_r (\eta f)}{\partial q}=\eta\frac{\partial_r f}{\partial q^{\eta}}=\frac{\partial_r f^{\eta}}{\partial q}\eta
\end{split}\quad \forall\eta \in \{1,i,j,k\}
\end{align}
where the second line of \eqref{pr:righthrpr2} is just a mirrors image of \eqref{pr:lefthrpr2}. Thus, we call Definition \ref{def:righthr} the right HR derivative, denoted by the $\partial_r$ to distinguish from the left HR derivatives.
\begin{remark}
The only difference between the left HR derivatives and the right HR derivative is the position of the partial derivative $\frac{\partial f}{\partial q_a}$, $\frac{\partial f}{\partial q_b}$, $\frac{\partial f}{\partial q_c}$ and $\frac{\partial f}{\partial q_d}$. In the left HR derivative, $\frac{\partial f}{\partial q_a},\frac{\partial f}{\partial q_b}, \frac{\partial f}{\partial q_c}$ and $\frac{\partial f}{\partial q_d}$ stand on the left side and imaginary units $i,j,k$ on the right side. It is exactly the opposite case for the right HR derivative. Note that the $\frac{\partial f}{\partial q_a},\frac{\partial f}{\partial q_b}, \frac{\partial f}{\partial q_c}$ and $\frac{\partial f}{\partial q_d}$ cannot swap position with the imaginary units $i,j,k$ because of the noncommutative nature of quaternion product.
\end{remark}

\subsection{Relation Between the HR Derivatives}
By applying the Hermitian operator to both sides of the expression in Definition \ref{def:lefthr} and using $(AB)^H=B^HA^H$, we obtain
\begin{align}\label{eq:lefthrconj}
\left(\begin{array}{cc}
\frac{\partial f}{\partial q},&\frac{\partial f}{\partial q^*}\\
\frac{\partial f}{\partial q^i},&\frac{\partial f}{\partial q^{i*}}\\
\frac{\partial f}{\partial q^j},&\frac{\partial f}{\partial q^{j*}}\\
\frac{\partial f}{\partial q^k},&\frac{\partial f}{\partial q^{k*}}
\end{array}\right)^*=\frac{1}{4}
\left(\begin{array}{cccc}
1 & i  & j & k \\
1 & i & -j  & -k\\
1 & -i  & j & -k\\
1 & -i  & -j  & k\
\end{array}\right)
\left(\begin{array}{cc}
\frac{\partial f}{\partial q_a},&\frac{\partial f}{\partial q_a}\\
\frac{\partial f}{\partial q_b},&-\frac{\partial f}{\partial q_b}\\
\frac{\partial f}{\partial q_c},&-\frac{\partial f}{\partial q_c}\\
\frac{\partial f}{\partial q_d},&-\frac{\partial f}{\partial q_d}
\end{array}\right)^*
\end{align}
by replacing $f$ with its conjugate $f^*$ in \eqref{eq:lefthrconj} and using $\left(\frac{\partial f^*}{\partial \xi}\right)^*=\frac{\partial f}{\partial \xi}$, $\xi\in\{q_a,q_b,q_c,q_d\}$. Then, the pair of relationships between the left HR derivatives and the right HR derivatives becomes
\begin{align}\label{eq:leftrightlink}
\left(\begin{array}{cc}
\frac{\partial_r f}{\partial q},&\frac{\partial_r f}{\partial q^*}\\
\frac{\partial_r f}{\partial q^i},&\frac{\partial_r f}{\partial q^{i*}}\\
\frac{\partial_r f}{\partial q^j},&\frac{\partial_r f}{\partial q^{j*}}\\
\frac{\partial_r f}{\partial q^k},&\frac{\partial_r f}{\partial q^{k*}}
\end{array}\right)=
\left(\begin{array}{cc}
\frac{\partial f^*}{\partial q^*},&\frac{\partial f^*}{\partial q}\\
\frac{\partial f^*}{\partial q^{i*}},&\frac{\partial f^*}{\partial q^i}\\
\frac{\partial f^*}{\partial q^{j*}},&\frac{\partial f^*}{\partial q^j}\\
\frac{\partial f^*}{\partial q^{k*}},&\frac{\partial f^*}{\partial q^k}
\end{array}\right)^*
,\quad
\left(\begin{array}{cc}
\frac{\partial f}{\partial q^*},&\frac{\partial f}{\partial q}\\
\frac{\partial f}{\partial q^{i*}},&\frac{\partial f}{\partial q^i}\\
\frac{\partial f}{\partial q^{j*}},&\frac{\partial f}{\partial q^j}\\
\frac{\partial f}{\partial q^{k*}},&\frac{\partial f}{\partial q^k}
\end{array}\right)=
\left(\begin{array}{cc}
\frac{\partial_r f^*}{\partial q},&\frac{\partial_r f^*}{\partial q^*}\\
\frac{\partial_r f^*}{\partial q^i},&\frac{\partial_r f^*}{\partial q^{i*}}\\
\frac{\partial_r f^*}{\partial q^j},&\frac{\partial_r f^*}{\partial q^{j*}}\\
\frac{\partial_r f^*}{\partial q^k},&\frac{\partial_r f^*}{\partial q^{k*}}
\end{array}\right)^*
\end{align}

\begin{remark}
From the identity \eqref{eq:leftrightlink}, we can see that the left HR derivative is equal to the right HR derivative if the function $f$ is real-valued. This result is instrumental for practical applications of the HR calculus, where the objective function (or cost function) is often real-valued, such as the mean square error. Using the HR calculus, it is not important to choose the kind of HR derivative,  because the final results are exactly the same. In the sequel, we therefore mainly focus on the left HR derivatives.
\end{remark}

\subsection{Higher Order HR Derivatives}\label{sec:hgderiv}
Since a formal derivative of a function $f:\mathbb{H}\rightarrow\mathbb{H}$ is (wherever it exists) again a function from $\mathbb{H}$ to $\mathbb{H}$, it makes sense to take the HR derivative of HR derivative, i.e., a higher order HR derivative. We shall
consider second order left derivatives of the form
\begin{equation}
\begin{split}
\frac{\partial^2 f}{\partial q^{\mu} \partial q^{\nu}}=\frac{\partial }{\partial q^{\mu}}\left(\frac{\partial f}{\partial q^{\nu}}\right), \quad
\frac{\partial^2 f}{\partial q^{\mu*} \partial q^{\nu*}}=\frac{\partial }{\partial q^{\mu*} }\left(\frac{\partial f}{\partial q^{\nu*} }\right)\\
\frac{\partial^2 f}{\partial q^{\mu} \partial q^{\nu*}}=\frac{\partial }{\partial q^{\mu}}\left(\frac{\partial f}{\partial q^{\nu*}}\right), \quad
\frac{\partial^2 f}{\partial q^{\mu*} \partial q^{\nu}}=\frac{\partial }{\partial q^{\mu*} }\left(\frac{\partial f}{\partial q^{\nu} }\right)
\end{split}
\end{equation}
where $\mu,\nu\in\{1,i,j,k\}$. From \eqref{pr:lefthrpr1} and \eqref{pr:lefthrpr2}, we obtain
\begin{equation}
\begin{split}
&\left(\frac{\partial^2 f}{\partial q^{i} \partial q^{i}}\right)^i=\frac{\partial }{\partial q^{ii}}\left(\frac{\partial f}{\partial q^{i}}\right)^i=\frac{\partial }{\partial q}\left(\frac{\partial f^i}{\partial q^{ii}}\right)=\frac{\partial^2 f^i}{\partial q\partial q}\\
&\left(\frac{\partial^2 f}{\partial q^{i} \partial q^{j}}\right)^k=\frac{\partial }{\partial q^{ki}}\left(\frac{\partial f}{\partial q^{j}}\right)^k=\frac{\partial }{\partial q^{j}}\left(\frac{\partial f^k}{\partial q^{kj}}\right)=\frac{\partial^2 f^k}{\partial q^{j} \partial q^{i}}\\
\end{split}
\end{equation}
If $f$ is a real-valued function, the second formula in the above expression can be simplified as
\begin{equation}
\left(\frac{\partial^2 f}{\partial q^{i} \partial q^{j}}\right)^k=\frac{\partial^2 f}{\partial q^{j} \partial q^{i}}
\end{equation}
This clearly shows that the mixed second order left HR derivatives are in general not equal, that is
\begin{equation}
\frac{\partial^2 f}{\partial q^{\mu} \partial q^{\nu}}\neq\frac{\partial^2 f}{\partial q^{\nu} \partial q^{\mu}}
\end{equation}
where $\mu,\nu\in\{1,i,j,k\}$. The second order left HR derivatives have a commutative property \cite{Sudbery}
\begin{equation}
16\frac{\partial^2 f}{\partial q^{\mu} \partial q^{\mu*}}=16\frac{\partial^2 f}{\partial q^{\mu*} \partial q^{\mu}}
=\Delta f\triangleq \frac{\partial^2 f}{\partial q^2_a}+\frac{\partial^2 f}{\partial q^2_b}
+\frac{\partial^2 f}{\partial q^2_c}+\frac{\partial^2 f}{\partial q^2_d}
\end{equation}
If $f$ is a real-valued function, we obtain
\begin{equation}
\begin{split}
\left(\frac{\partial^2 f}{\partial q^{\mu} \partial q^{\nu}}\right)^*=\frac{\partial^2 f}{\partial q^{\nu*} \partial q^{\mu*}},\quad
\left(\frac{\partial^2 f}{\partial q^{\mu*} \partial q^{\nu*}}\right)^*=\frac{\partial^2 f}{\partial q^{\nu} \partial q^{\mu}}\\
\left(\frac{\partial^2 f}{\partial q^{\mu} \partial q^{\nu*}}\right)^*=\frac{\partial^2 f}{\partial q^{\nu} \partial q^{\mu*}},\quad
\left(\frac{\partial^2 f}{\partial q^{\mu*} \partial q^{\nu}}\right)^*=\frac{\partial^2 f}{\partial q^{\nu*} \partial q^{\mu}}\\
\end{split}
\end{equation}
In a similar manner, the second order right HR derivatives can be defined as
\begin{equation}
\begin{split}
\frac{\partial_r^2 f}{\partial q^{\mu} \partial q^{\nu}}=\frac{\partial_r }{\partial q^{\mu}}\left(\frac{\partial_r f}{\partial q^{\nu}}\right), \quad
\frac{\partial_r^2 f}{\partial q^{\mu*} \partial q^{\nu*}}=\frac{\partial_r }{\partial q^{\mu*} }\left(\frac{\partial_r f}{\partial q^{\nu*} }\right)\\
\frac{\partial_r^2 f}{\partial q^{\mu} \partial q^{\nu*}}=\frac{\partial_r }{\partial q^{\mu}}\left(\frac{\partial_r f}{\partial q^{\nu*}}\right), \quad
\frac{\partial_r^2 f}{\partial q^{\mu*} \partial q^{\nu}}=\frac{\partial_r }{\partial q^{\mu*} }\left(\frac{\partial_r f}{\partial q^{\nu} }\right)
\end{split}
\end{equation}
where $\mu,\nu\in\{1,i,j,k\}$. An important commutative property between second order left and right derivatives of a real valued function $f$ is given by
\begin{equation}
\begin{split}
\frac{\partial_r^2 f}{\partial q^{\mu} \partial q^{\nu}}=\frac{\partial^2 f}{\partial q^{\nu} \partial q^{\mu}}, \quad
\frac{\partial_r^2 f}{\partial q^{\mu*} \partial q^{\nu*}}=\frac{\partial^2 f}{\partial q^{\nu*} \partial q^{\mu*}}\\
\frac{\partial_r^2 f}{\partial q^{\mu} \partial q^{\nu*}}=\frac{\partial^2 f}{\partial q^{\nu*} \partial q^{\mu}}, \quad
\frac{\partial_r^2 f}{\partial q^{\mu*} \partial q^{\nu}}=\frac{\partial^2 f}{\partial q^{\nu} \partial q^{\mu*}}
\end{split}
\end{equation}

\subsection{The Validity of the Traditional Product Rule.}\label{sec:limithr}  Definitions \ref{def:lefthr} and \ref{def:righthr} give a method to calculate the HR derivatives, but are complicated and inefficient. For example, the power function $f(q)=q^n$, it is too complicated and inconvenient to compute using Definition \ref{def:lefthr} or \ref{def:righthr}. The greatest difficulty with the HR calculus is that it does not satisfy the traditional product rule, that is, for any quaternion functions $f(q)$ and $g(q)$, in general we have
\begin{equation}\label{rl:oldprdrule}
\frac{\partial (fg)}{\partial q}\neq f\frac{\partial g}{\partial q}+\frac{\partial f}{\partial q}g
\end{equation}
We shall illustrate this technical obstacle by two examples.
\begin{example}\label{exp:qq}
Find the HR derivative of the function $f:\mathbb{H}\rightarrow \mathbb{H}$ given by
\begin{equation}
f(q)=q^2=q_a^2-(q_b^2+q_c^2+q_d^2)+2q_a(iq_b+jq_c+kq_d)
\end{equation}
where $q=q_a+iq_b+jq_c+kq_d$, $q_a,q_b,q_c,q_d\in\mathbb{R}$.
\end{example}
\textbf{Solution}: By Definition \ref{def:lefthr}, the left side of \eqref{rl:oldprdrule} becomes
\begin{align}
\begin{split}
&\frac{\partial (q^2)}{\partial q}=\frac{1}{4}\left(\frac{\partial q^2}{\partial q_a}-\frac{\partial q^2}{\partial q_b}i-\frac{\partial q^2}{\partial q_c}j-\frac{\partial q^2}{\partial q_d}k\right)\\
&=\frac{1}{4}\Big(2q_a+2(iq_b+jq_c+kq_d)-(-2q_b+2q_a i)i-(-2q_c+2q_a j)j-(-2q_d+2q_a k)k\Big)\\
&=\frac{1}{4}\Big(8q_a+4q_b i+4q_c j+4q_d k\Big)=q+\mathfrak{R}(q)
\end{split}
\end{align}
Alternatively, using the property \eqref{pr:lefthrpr1}, the right side of \eqref{rl:oldprdrule} can be calculated as
\begin{equation}
q\frac{\partial q}{\partial q}+\frac{\partial q}{\partial q}q=2q
\end{equation}
This clearly shows that the left side is not equal to the right side of \eqref{rl:oldprdrule}, and thus the product rule is not valid.

\begin{example}\label{exp:qqconj}
Find the HR derivative of the function $f:\mathbb{H}\rightarrow \mathbb{H}$ given by
\begin{equation}
f(q)=|q|^2=qq^*
\end{equation}
\end{example}
\textbf{Solution}: By Definition \ref{def:lefthr}, we will first calculate the left side of \eqref{rl:oldprdrule} as
\begin{align}
\begin{split}
\frac{\partial (|q|^2)}{\partial q}&=\frac{1}{4}\left(\frac{\partial |q|^2}{\partial q_a}-\frac{\partial |q|^2}{\partial q_b}i-\frac{\partial |q|^2}{\partial q_c}j-\frac{\partial |q|^2}{\partial q_d}k\right)\\
&=\frac{1}{4}\Big(2q_a-2q_bi-2q_cj-2q_dk\Big)=\frac{1}{2}q^*
\end{split}
\end{align}
while, using the property \eqref{pr:lefthrpr1}, the right side of \eqref{rl:oldprdrule} can be calculated as
\begin{equation}
q\frac{\partial q^*}{\partial q}+\frac{\partial q}{\partial q}q^*=-\frac{q}{2}+q^*
\end{equation}
This clearly shows that the left side is not equal to the right side of \eqref{rl:oldprdrule}.

\begin{remark}
Examples \ref{exp:qq} and \ref{exp:qqconj} show that the traditional product rule is not applicable for the left HR derivative in Definition \ref{def:lefthr}. In a similar manner, the traditional product rule is not applicable for the right HR derivative in Definition \ref{def:righthr}.
\end{remark}

\section{The Generalization of HR Calculus}
In this section, we propose the GHR derivatives to solve the obstacle of the product rule within the HR calculus. We achieve this by changing the basis $\{1,i,j,k\}$ in Definition \ref{def:lefthr} and \ref{def:righthr} to a general orthogonal basis $\{1,i^{\mu},j^{\mu},k^{\mu}\}$, as shown in \eqref{pr:prodijkmurl}.
This allows us to give a similar derivation of the GHR calculus as that of HR calculus in Appendix A, but this is omitted to save the space.
\begin{definition}[(The Left GHR Derivatives)]\label{def:leftghr}
Let $q=q_a+iq_b+jq_c+kq_d$, where $q_a,q_b,q_c,q_d\in\mathbb{R}$, then the left GHR derivatives, with respect to $q^{\mu}$ and $q^{\mu*}$ $(\mu\neq 0, \mu \in \mathbb{H})$, of the function $f$ are defined as
\begin{equation*}
\begin{split}
\frac{\partial f}{\partial q^{\mu}}=\frac{1}{4}\left(\frac{\partial f}{\partial q_a}-\frac{\partial f}{\partial q_b}i^{\mu}-\frac{\partial f}{\partial q_c}j^{\mu}-\frac{\partial f}{\partial q_d}k^{\mu}\right)\\
\frac{\partial f}{\partial q^{\mu*}}=\frac{1}{4}\left(\frac{\partial f}{\partial q_a}+\frac{\partial f}{\partial q_b}i^{\mu}+\frac{\partial f}{\partial q_c}j^{\mu}+\frac{\partial f}{\partial q_d}k^{\mu}\right)
\end{split}
\end{equation*}
where $\frac{\partial f}{\partial q_a}$, $\frac{\partial f}{\partial q_b}$, $\frac{\partial f}{\partial q_c}$ and $\frac{\partial f}{\partial q_d}$ are the partial derivatives of $f$ with respect to $q_a$, $q_b$, $q_c$ and $q_d$, respectively, and the set $\{1,i^{\mu},j^{\mu},k^{\mu}\}$ is a general orthogonal basis of $\mathbb{H}$.
\end{definition}

The properties of Definition \ref{def:leftghr} are:
\begin{align}\label{pr:leftghrpr2}
\frac{\partial f}{\partial q^{\mu}}=
\left\{\begin{array}{cc}
\frac{\partial f}{\partial q}, \quad if\;\mu=1\\
\frac{\partial f}{\partial q^i}, \quad if\;\mu=i\\
\frac{\partial f}{\partial q^j}, \quad if\;\mu=j\\
\frac{\partial f}{\partial q^k}, \quad if\;\mu=k
\end{array}\right.
,\quad
\frac{\partial f}{\partial q^{\mu*}}=
\left\{\begin{array}{cc}
\frac{\partial f}{\partial q^*}, \quad if\;\mu=1\\
\frac{\partial f}{\partial q^{i*}}, \quad if\;\mu=i\\
\frac{\partial f}{\partial q^{j*}}, \quad if\;\mu=j\\
\frac{\partial f}{\partial q^{k*}}, \quad if\;\mu=k
\end{array}\right.
\end{align}
\begin{equation}\label{pr:leftghrpr1}
\begin{split}
\left(\frac{\partial f}{\partial q^{\mu}}\right)^{\nu}=\frac{\partial f^{\nu}}{\partial q^{\nu\mu}},
\quad &\left(\frac{\partial f}{\partial q^{\nu^*\mu}}\right)^{\nu}=\frac{\partial f^{\nu}}{\partial q^{\mu}}\\
\frac{\partial (\nu f)}{\partial q^{\mu}}=\nu\frac{\partial f}{\partial q^{\mu}}=\frac{\partial f^{\nu}}{\partial q^{\nu\mu}}\nu,
\quad &\frac{\partial (f\nu)}{\partial q^{\mu}}=\frac{\partial f}{\partial q^{\nu\mu}}\nu=\nu\frac{\partial f^{(\nu^*)}}{\partial q^{\mu}}
\end{split}\quad \forall\nu \in \mathbb{H}
\end{equation}
where the properties \eqref{pr:pqmu} and \eqref{pr:def1qmunu} are used in the first line of \eqref{pr:leftghrpr1}, and $\mu q=q^{\mu}\mu$ and \eqref{pr:def1qmunu} are used in the second line of \eqref{pr:leftghrpr1}. The detail is omitted because the proof is similar to \eqref{pr:lefthrpr2}. If $f$ is a real-valued function, the conjugate rule of the left GHR derivatives is given by
\begin{equation}
\begin{split}
\left(\frac{\partial f}{\partial q^{\mu}}\right)^{*}=\frac{\partial f}{\partial q^{\mu*}},
\quad &\left(\frac{\partial f}{\partial q^{\mu*}}\right)^{*}=\frac{\partial f}{\partial q^{\mu}}
\end{split}
\end{equation}

\begin{definition}[(The Right GHR Derivatives)]\label{def:rightghr}
Let $q=q_a+iq_b+jq_c+kq_d$, where $q_a,q_b,q_c,q_d\in\mathbb{R}$, then the right GHR derivatives with respect to $q^{\mu}$ and $q^{\mu*}$ $(\mu\neq 0, \mu \in \mathbb{H})$ of the function $f$, are defined as
\begin{equation*}
\begin{split}
\frac{\partial_r f}{\partial q^{\mu}}=\frac{1}{4}\left(\frac{\partial f}{\partial q_a}-i^{\mu}\frac{\partial f}{\partial q_b}-j^{\mu}\frac{\partial f}{\partial q_c}-k^{\mu}\frac{\partial f}{\partial q_d}\right)\\
\frac{\partial_r f}{\partial q^{\mu*}}=\frac{1}{4}\left(\frac{\partial f}{\partial q_a}+i^{\mu}\frac{\partial f}{\partial q_b}+j^{\mu}\frac{\partial f}{\partial q_c}+k^{\mu}\frac{\partial f}{\partial q_d}\right)
\end{split}
\end{equation*}
where $\frac{\partial f}{\partial q_a}$, $\frac{\partial f}{\partial q_b}$, $\frac{\partial f}{\partial q_c}$ and $\frac{\partial f}{\partial q_d}$ are the partial derivatives of $f$ with respect to $q_a$, $q_b$, $q_c$ and $q_d$, respectively, and the set $\{1,i^{\mu},j^{\mu},k^{\mu}\}$ is a general orthogonal basis of $\mathbb{H}$.
\end{definition}

The properties of Definition \ref{def:rightghr} are:
\begin{align}\label{pr:grightghrpr2}
\frac{\partial_r f}{\partial q^{\mu}}=
\left\{\begin{array}{cc}
\LARGE{\frac{\partial_r f}{\partial q}}, \quad if\;\mu=1\\
\frac{\partial_r f}{\partial q^i}, \quad if\;\mu=i\\
\frac{\partial_r f}{\partial q^j}, \quad if\;\mu=j\\
\frac{\partial_r f}{\partial q^k}, \quad if\;\mu=k
\end{array}\right.
,\quad
\frac{\partial_r f}{\partial q^{\mu*}}=
\left\{\begin{array}{cc}
\LARGE{\frac{\partial_r f}{\partial q^*}}, \quad if\;\mu=1\\
\frac{\partial_r f}{\partial q^{i*}}, \quad if\;\mu=i\\
\frac{\partial_r f}{\partial q^{j*}}, \quad if\;\mu=j\\
\frac{\partial_r f}{\partial q^{k*}}, \quad if\;\mu=k
\end{array}\right.
\end{align}
\begin{equation}\label{pr:grightghrpr1}
\begin{split}
\left(\frac{\partial_r f}{\partial q^{\mu}}\right)^{\nu}=\frac{\partial_r f^{\nu}}{\partial q^{\nu\mu}},
\quad &\left(\frac{\partial_r f}{\partial q^{\nu^*\mu}}\right)^{\nu}=\frac{\partial_r f^{\nu}}{\partial q^{\mu}}\\
\frac{\partial_r (f\nu)}{\partial q^{\mu}}=\frac{\partial_r f}{\partial q^{\mu}}\nu=\nu\frac{\partial_r f^{(\nu^*)}}{\partial q^{\nu^*\mu}},
\quad &\frac{\partial_r (\nu f)}{\partial q^{\mu}}=\nu\frac{\partial_r f}{\partial q^{\nu^*\mu}}=\frac{\partial_r f^{\nu}}{\partial q^{\mu}}\nu
\end{split}\quad \forall\nu \in \mathbb{H}
\end{equation}
Similar to the relation in \eqref{eq:leftrightlink}, the relation between the two kinds of the GHR derivatives can be found as
\begin{equation}\label{eq:gleftrightlink}
\begin{split}
\frac{\partial_r f}{\partial q^{\mu}}=\left(\frac{\partial f^*}{\partial q^{\mu*}}\right)^*,\quad
\frac{\partial_r f}{\partial q^{\mu*}}=\left(\frac{\partial f^*}{\partial q^{\mu}}\right)^*\\
\frac{\partial f}{\partial q^{\mu}}=\left(\frac{\partial_r f^*}{\partial q^{\mu*}}\right)^* ,\quad
\frac{\partial f}{\partial q^{\mu*}}=\left(\frac{\partial_r f^*}{\partial q^{\mu}}\right)^*
\end{split}\quad \forall\mu \in \mathbb{H}
\end{equation}

\begin{remark}
By comparing Definition \ref{def:lefthr} and Definition \ref{def:leftghr},
it is seen that the GHR derivative is more concise and easier to understand than the HR derivative,
and that the HR derivative is a special case of the GHR derivative. More importantly, as shown below, the GHR derivatives incorporate a novel product rule, which is very convenient for calculating the HR and GHR derivatives. In addition, the GHR derivative can be extended to other orthogonal systems, such as $\{1,\eta,\eta',\eta''\}$ in \cite{Jvia10,Jvia11}.
\end{remark}

\subsection{The Novel Product Rule.}
In Section \ref{sec:limithr}, we have explained that the traditional product rule is not feasible for the HR calculus. Now, we propose a novel product rule to solve this technical obstacle, and show that the traditional product rule is a special case of the novel product rule in Corollary \ref{cor:leftprodrl}.
\begin{theorem}[(Product Rule of Left GHR)]\label{thm:nwpleftpdrl}
If the functions $f, g:\mathbb{H}\rightarrow \mathbb{H}$ have the left GHR derivatives, then so too has their product $fg$, and
\begin{equation*}
\frac{\partial (fg)}{\partial q^{\mu}}=f\frac{\partial g}{\partial q^{\mu}}+\frac{\partial f}{\partial q^{g\mu}}g,\quad \frac{\partial (fg)}{\partial q^{\mu*}}=f\frac{\partial g}{\partial q^{\mu*}}+\frac{\partial f}{\partial q^{g\mu*}}g
\end{equation*}
where $\frac{\partial f}{\partial q^{g\mu}}$ and $\frac{\partial f}{\partial q^{g\mu*}}$ can be obtained by replacing $\mu$ with $g\mu$ in Definition \ref{def:leftghr}.
\end{theorem}
\begin{proof}
The proof of Theorem \ref{thm:nwpleftpdrl} is given in Appendix B.
\end{proof}

\begin{corollary}[(Product Rule of Left HR)]\label{cor:nwpleftpdrl}
If the functions $f,g:\mathbb{H}\rightarrow \mathbb{H}$ have the left HR derivatives, then so too has their product $fg$, and
\begin{equation}
\begin{split}
&\frac{\partial (fg)}{\partial q^{}}=f\frac{\partial g}{\partial q^{}}+\frac{\partial f}{\partial q^{g}}g,
\quad \frac{\partial (fg)}{\partial q^{i}}=f\frac{\partial g}{\partial q^{i}}+\frac{\partial f}{\partial q^{gi}}g\\
&\frac{\partial (fg)}{\partial q^{j}}=f\frac{\partial g}{\partial q^{j}}+\frac{\partial f}{\partial q^{gj}}g,
\quad \frac{\partial (fg)}{\partial q^{k}}=f\frac{\partial g}{\partial q^{k}}+\frac{\partial f}{\partial q^{gk}}g
\end{split}
\end{equation}
\begin{equation}
\begin{split}
&\frac{\partial (fg)}{\partial q^{*}}=f\frac{\partial g}{\partial q^{*}}+\frac{\partial f}{\partial q^{g*}}g,\quad
\frac{\partial (fg)}{\partial q^{i*}}=f\frac{\partial g}{\partial q^{i*}}+\frac{\partial f}{\partial q^{gi*}}g\\
&\frac{\partial (fg)}{\partial q^{j*}}=f\frac{\partial g}{\partial q^{j*}}+\frac{\partial f}{\partial q^{gj*}}g
\quad \frac{\partial (fg)}{\partial q^{k*}}=f\frac{\partial g}{\partial q^{k*}}+\frac{\partial f}{\partial q^{gk*}}g
\end{split}
\end{equation}
where $\frac{\partial f}{\partial q},\frac{\partial f}{\partial q^i},\frac{\partial f}{\partial q^j},\frac{\partial f}{\partial q^k}$ and so on are the left HR derivatives in Definition \ref{def:lefthr}, and $\frac{\partial f}{\partial q^{gi}},\frac{\partial f}{\partial q^{gj}},\frac{\partial f}{\partial q^{gk}}$ and so on
can be obtained by replacing $\mu$ with $gi,gj,jk$ in Definition \ref{def:leftghr}.
\end{corollary}
\begin{proof}
Set $\mu=1,i,j,k$ in Theorem \ref{thm:nwpleftpdrl}, then the corollary follows.
\end{proof}

Theorem \ref{thm:nwpleftpdrl} is also valid for the product of quaternion-valued function and real-valued function of quaternion variables, as stated below.
\begin{corollary}\label{cor:leftprodrl}
If the functions $f:\mathbb{H}\rightarrow \mathbb{H}$ and $g:\mathbb{H}\rightarrow \mathbb{R}$ have the left GHR derivatives, then
their product $fg$ satisfies the traditional product rule
\begin{equation*}
\frac{\partial (fg)}{\partial q^{\mu}}=f\frac{\partial g}{\partial q^{\mu}}+\frac{\partial f}{\partial q^{\mu}}g,\quad \frac{\partial (fg)}{\partial q^{\mu*}}=f\frac{\partial g}{\partial q^{\mu*}}+\frac{\partial f}{\partial q^{\mu*}}g
\end{equation*}
where $\frac{\partial f}{\partial q^{\mu}}$ and $\frac{\partial f}{\partial q^{\mu*}}$ are the left GHR derivatives in Definition \ref{def:leftghr}.
\end{corollary}
\begin{proof}
From $q^{g\mu}=q^{\mu}$ and $q^{g\mu*}=q^{\mu*}$ for the real function $g$, then the corollary follows.
\end{proof}

\begin{theorem}[(Product Rule of Right GHR)]\label{thm:nwrtpdrl}
If the functions $f, g:\mathbb{H}\rightarrow \mathbb{H}$ have the right GHR derivatives, then so too has their product $fg$, and
\begin{equation*}
\frac{\partial_r (fg)}{\partial q^{\mu}}=\frac{\partial_r f}{\partial q^{\mu}}g+f\frac{\partial_r g}{\partial q^{f^*\mu}},
\quad \frac{\partial_r (fg)}{\partial q^{\mu*}}=\frac{\partial_r f}{\partial q^{\mu*}}g+f\frac{\partial_r g}{\partial q^{f^*\mu*}}
\end{equation*}
where $\frac{\partial_r g}{\partial q^{\mu f}}$ and $\frac{\partial_r g}{\partial q^{\mu f*}}$ are obtained by replacing $\mu$ with $\mu f$ in Definition \ref{def:rightghr}.
\end{theorem}

\begin{corollary}[(Product Rule of Right HR)]\label{cor:nwprightpdrl}
If the functions $f, g:\mathbb{H}\rightarrow \mathbb{H}$ have the right HR derivatives, then so too has their product $fg$, and
\begin{equation}
\begin{split}
&\frac{\partial_r (fg)}{\partial q^{}}=\frac{\partial_r f}{\partial q^{}}g+f\frac{\partial_r g}{\partial q^{(f^*)}},
\quad \frac{\partial_r (fg)}{\partial q^{i}}=\frac{\partial_r f}{\partial q^{i}}g+f\frac{\partial_r g}{\partial q^{f^*i}}\\
&\frac{\partial_r (fg)}{\partial q^{j}}=\frac{\partial_r f}{\partial q^{j}}g+f\frac{\partial_r g}{\partial q^{f^*j}},
\quad \frac{\partial_r (fg)}{\partial q^{k}}=\frac{\partial_r f}{\partial q^{k}}g+f\frac{\partial_r g}{\partial q^{f^*k}}
\end{split}
\end{equation}
\begin{equation}
\begin{split}
&\frac{\partial_r (fg)}{\partial q^{*}}=\frac{\partial_r f}{\partial q^{*}}g+f\frac{\partial_r g}{\partial q^{(f^**)}},
\quad \frac{\partial_r (fg)}{\partial q^{i*}}=\frac{\partial_r f}{\partial q^{i*}}g+f\frac{\partial_r g}{\partial q^{f^*i*}}\\
&\frac{\partial_r (fg)}{\partial q^{j*}}=\frac{\partial_r f}{\partial q^{j*}}g+f\frac{\partial_r g}{\partial q^{f^*j*}}
\quad \frac{\partial_r (fg)}{\partial q^{k*}}=\frac{\partial_r f}{\partial q^{k*}}g+f\frac{\partial_r g}{\partial q^{f^*k*}}
\end{split}
\end{equation}
where $\frac{\partial_r f}{\partial q},\frac{\partial_r f}{\partial q^i},\frac{\partial_r f}{\partial q^j},\frac{\partial_r f}{\partial q^k}$ etc. are the right HR derivatives in Definition \ref{def:righthr}.
\end{corollary}

\begin{corollary}\label{cor:rightprodrl}
If the functions $f:\mathbb{H}\rightarrow \mathbb{R}$ and $g:\mathbb{H}\rightarrow \mathbb{H}$ have the right GHR derivatives, then
their product $fg$ satisfy the traditional product rule
\begin{equation*}
\frac{\partial_r (fg)}{\partial q^{\mu}}=\frac{\partial_r f}{\partial q^{\mu}}g+f\frac{\partial_r g}{\partial q^{\mu }},\quad \frac{\partial_r (fg)}{\partial q^{\mu*}}=\frac{\partial_r f}{\partial q^{\mu}}g+f\frac{\partial_r g}{\partial q^{\mu *}}
\end{equation*}
where $\frac{\partial_r f}{\partial q^{\mu}}$ and $\frac{\partial_r f}{\partial q^{\mu*}}$ are the right GHR derivatives in Definition \ref{def:rightghr}.
\end{corollary}

The proofs of Theorem \ref{thm:nwrtpdrl}, Corollary \ref{cor:nwprightpdrl} and Corollary \ref{cor:rightprodrl} are conformal with Theorem \ref{thm:nwpleftpdrl}, Corollary \ref{cor:nwpleftpdrl} and Corollary \ref{cor:leftprodrl}, and thus omitted.

\subsection{The Chain Rule.}
Another advantage of the GHR derivative is defined in Definition \ref{def:leftghr} and Definition \ref{def:rightghr} is that the chain rule can be obtained in a very simple form, and is formulated in the following theorem.
\begin{theorem}[(Chain Rule of Left GHR)]\label{thm:nwpleftchainrl}
 Let $S\subseteq \mathbb{H}$ and suppose $g:S\rightarrow \mathbb{H}$ has the left GHR derivative at an interior point $q$ of the set $S$. Let $T\subseteq \mathbb{H}$ be such that $g(q)\in T$ for all $q \in S$. Assume $f:T\rightarrow \mathbb{H}$ has left GHR derivatives at an inner point $g(q)\in T$, then the left GHR derivatives of the composite function $f(g(q))$ are as follows:
\begin{equation}
\begin{split}
&\frac{\partial f(g(q))}{\partial q^{\mu}}=\frac{\partial f}{\partial g^{\nu}}\frac{\partial g^{\nu}}{\partial q^{\mu}}+\frac{\partial f}{\partial g^{\nu i}}\frac{\partial g^{\nu i}}{\partial q^{\mu}}+\frac{\partial f}{\partial g^{\nu j}}\frac{\partial g^{\nu j}}{\partial q^{\mu}}+\frac{\partial f}{\partial g^{\nu k}}\frac{\partial g^{\nu k}}{\partial q^{\mu}}\\
&\frac{\partial f(g(q))}{\partial q^{\mu*}}=\frac{\partial f}{\partial g^{\nu}}\frac{\partial g^{\nu}}{\partial q^{\mu*}}+\frac{\partial f}{\partial g^{\nu i}}\frac{\partial g^{\nu i}}{\partial q^{\mu*}}+\frac{\partial f}{\partial g^{\nu j}}\frac{\partial g^{\nu j}}{\partial q^{\mu*}}+\frac{\partial f}{\partial g^{\nu k}}\frac{\partial g^{\nu k}}{\partial q^{\mu*}}
\end{split}
\end{equation}
\begin{equation}
\begin{split}
&\frac{\partial f(g(q))}{\partial q^{\mu}}=\frac{\partial f}{\partial g^{\nu*}}\frac{\partial g^{\nu*}}{\partial q^{\mu}}+\frac{\partial f}{\partial g^{\nu i*}}\frac{\partial g^{\nu i*}}{\partial q^{\mu}}+\frac{\partial f}{\partial g^{\nu j*}}\frac{\partial g^{\nu j*}}{\partial q^{\mu}}+\frac{\partial f}{\partial g^{\nu k*}}\frac{\partial g^{\nu k*}}{\partial q^{\mu}}\\
&\frac{\partial f(g(q))}{\partial q^{\mu*}}=\frac{\partial f}{\partial g^{\nu*}}\frac{\partial g^{\nu*}}{\partial q^{\mu*}}+\frac{\partial f}{\partial g^{\nu i*}}\frac{\partial g^{\nu i*}}{\partial q^{\mu*}}+\frac{\partial f}{\partial g^{\nu j*}}\frac{\partial g^{\nu j*}}{\partial q^{\mu*}}+\frac{\partial f}{\partial g^{\nu k*}}\frac{\partial g^{\nu k*}}{\partial q^{\mu*}}
\end{split}
\end{equation}
where $\mu, \nu\in\mathbb{H},\mu\nu \neq 0$.
\end{theorem}
\begin{proof}
The proof of Theorem \ref{thm:nwpleftchainrl} is given in Appendix C.
\end{proof}

\begin{corollary}[(Chain Rule of Left HR)]\label{cor:leftchainrl}
Let $S\subseteq \mathbb{H}$ and suppose $g:S\rightarrow \mathbb{H}$ has the left HR derivative at an interior point $q$ of the set $S$. Let $T\subseteq \mathbb{H}$ be such that $g(q)\in T$ for all $q \in S$. Assume $f:T\rightarrow \mathbb{H}$ has left HR derivatives at an inner point $g(q)\in T$, then the left HR derivatives of the composite function $f(g(q))$ are as follows:
\begin{equation}
\begin{split}
&\frac{\partial f(g(q))}{\partial q^{\mu}}=\frac{\partial f}{\partial g}\frac{\partial g}{\partial q^{\mu}}+\frac{\partial f}{\partial g^{i}}\frac{\partial g^{i}}{\partial q^{\mu}}+\frac{\partial f}{\partial g^{j}}\frac{\partial g^{j}}{\partial q^{\mu}}+\frac{\partial f}{\partial g^{k}}\frac{\partial g^{k}}{\partial q^{\mu}}\\
&\frac{\partial f(g(q))}{\partial q^{\mu*}}=\frac{\partial f}{\partial g}\frac{\partial g}{\partial q^{\mu*}}+\frac{\partial f}{\partial g^{i}}\frac{\partial g^{i}}{\partial q^{\mu*}}+\frac{\partial f}{\partial g^{j}}\frac{\partial g^{j}}{\partial q^{\mu*}}+\frac{\partial f}{\partial g^{k}}\frac{\partial g^{k}}{\partial q^{\mu*}}
\end{split}
\end{equation}
\begin{equation}
\begin{split}
&\frac{\partial f(g(q))}{\partial q^{\mu}}=\frac{\partial f}{\partial g^{*}}\frac{\partial g^{*}}{\partial q^{\mu}}+\frac{\partial f}{\partial g^{i*}}\frac{\partial g^{i*}}{\partial q^{\mu}}+\frac{\partial f}{\partial g^{j*}}\frac{\partial g^{j*}}{\partial q^{\mu}}+\frac{\partial f}{\partial g^{k*}}\frac{\partial g^{k*}}{\partial q^{\mu}}\\
&\frac{\partial f(g(q))}{\partial q^{\mu*}}=\frac{\partial f}{\partial g^{*}}\frac{\partial g^{*}}{\partial q^{\mu*}}+\frac{\partial f}{\partial g^{i*}}\frac{\partial g^{i*}}{\partial q^{\mu*}}+\frac{\partial f}{\partial g^{j*}}\frac{\partial g^{j*}}{\partial q^{\mu*}}+\frac{\partial f}{\partial g^{k*}}\frac{\partial g^{k*}}{\partial q^{\mu*}}
\end{split}
\end{equation}
where $\mu \in \{1,i,j,k\}$.
\end{corollary}
\begin{proof}
Set $\nu=1$ in Theorem \ref{thm:nwpleftchainrl}, then the corollary follows.
\end{proof}

Theorem \ref{thm:nwpleftchainrl} is also valid for complex-valued and real-valued composite functions of quaternion variables stated in the following two corollaries, whose proofs are the same as that of Theorem \ref{thm:nwpleftchainrl}, thus omitted.
\begin{corollary}\label{cor:leftcompchainrl}
Let $S\subseteq \mathbb{H}$ and suppose $g:S\rightarrow \mathbb{C}$ has the left GHR derivative at an interior point $q$ of the set $S$. Let $T\subseteq \mathbb{C}$ be such that $g(q)\in T$ for all $q \in S$. Assume $f:T\rightarrow \mathbb{C}$ has the CR derivatives at an inner point $g(q)\in T$, then the left GHR derivatives of the composite function $f(g(q))$ are as follows:
\begin{equation}
\begin{split}
&\frac{\partial f(g(q))}{\partial q^{\mu}}=\frac{\partial f}{\partial g}\frac{\partial g}{\partial q^{\mu}}+\frac{\partial f}{\partial g^*}\frac{\partial g^*}{\partial q^{\mu}},\quad
\frac{\partial f(g(q))}{\partial q^{\mu*}}=\frac{\partial f}{\partial g}\frac{\partial g}{\partial q^{\mu*}}+\frac{\partial f}{\partial g^*}\frac{\partial g^*}{\partial q^{\mu*}}
\end{split}
\end{equation}
where $\mu\in\mathbb{H},\mu \neq 0$, $\frac{\partial f}{\partial g}$ and $\frac{\partial f}{\partial g^*}$ are the CR derivatives in CR calculus.
\end{corollary}

\begin{corollary}\label{cor:leftrealchainrl}
Let $S\subseteq \mathbb{H}$ and suppose $g:S\rightarrow \mathbb{R}$ has the left GHR derivative at an interior point $q$ of the set $S$. Let $T\subseteq \mathbb{R}$ be such that $g(q)\in T$ for all $q \in S$. Assume $f:T\rightarrow \mathbb{R}$ has real derivatives at an inner point $f(q)\in T$, then the left GHR derivatives of the composite function $f(g(q))$ are as follows:
\begin{equation}
\begin{split}
&\frac{\partial f(g(q))}{\partial q^{\mu}}=f'(g)\frac{\partial g}{\partial q^{\mu}},\quad
\frac{\partial f(g(q))}{\partial q^{\mu*}}=f'(g)\frac{\partial g}{\partial q^{\mu*}}
\end{split}
\end{equation}
where $\mu\in\mathbb{H},\mu \neq 0$ and $f'(g)$ is the real derivatives of real-valued function.
\end{corollary}

\begin{theorem}[(Chain Rule of Right GHR)]\label{thm:nwprightchainrl}
Let $S\subseteq \mathbb{H}$ and suppose $g:S\rightarrow \mathbb{H}$ has the right GHR derivative at an interior point $q$ of the set $S$. Let $T\subseteq \mathbb{H}$ be such that $g(q)\in T$ for all $q \in S$. Assume $f:T\rightarrow \mathbb{H}$ has right GHR derivatives at an inner point $g(q)\in T$, then the right GHR derivatives of the composite function $f(g(q))$ are as follows:
\begin{equation}
\begin{split}
&\frac{\partial_r f(g(q))}{\partial q^{\mu}}=\frac{\partial_r g^{\nu}}{\partial q^{\mu}}\frac{\partial_r f}{\partial g^{\nu}}+\frac{\partial_r g^{\nu i}}{\partial q^{\mu}}\frac{\partial_r f}{\partial g^{\nu i}}+\frac{\partial_r g^{\nu j}}{\partial q^{\mu}}\frac{\partial_r f}{\partial g^{\nu j}}+\frac{\partial_r g^{\nu k}}{\partial q^{\mu}}\frac{\partial_r f}{\partial g^{\nu k}}\\
&\frac{\partial_r f(g(q))}{\partial q^{\mu*}}=\frac{\partial_r g^{\nu}}{\partial q^{\mu*}}\frac{\partial_r f}{\partial g^{\nu}}+\frac{\partial_r g^{\nu i}}{\partial q^{\mu*}}\frac{\partial_r f}{\partial g^{\nu i}}+\frac{\partial_r g^{\nu j}}{\partial q^{\mu*}}\frac{\partial_r f}{\partial g^{\nu j}}+\frac{\partial_r g^{\nu k}}{\partial q^{\mu*}}\frac{\partial_r f}{\partial g^{\nu k}}
\end{split}
\end{equation}
\begin{equation}
\begin{split}
&\frac{\partial_r f(g(q))}{\partial q^{\mu}}=\frac{\partial_r g^{\nu*}}{\partial q^{\mu}}\frac{\partial_r f}{\partial g^{\nu*}}+\frac{\partial_r g^{\nu i*}}{\partial q^{\mu}}\frac{\partial_r f}{\partial g^{\nu i*}}+\frac{\partial_r g^{\nu j*}}{\partial q^{\mu}}\frac{\partial_r f}{\partial g^{\nu j*}}+\frac{\partial_r g^{\nu k*}}{\partial q^{\mu}}\frac{\partial_r f}{\partial g^{\nu k*}}\\
&\frac{\partial_r f(g(q))}{\partial q^{\mu*}}=\frac{\partial_r g^{\nu*}}{\partial q^{\mu*}}\frac{\partial_r f}{\partial g^{\nu*}}+\frac{\partial_r g^{\nu i*}}{\partial q^{\mu*}}\frac{\partial_r f}{\partial g^{\nu i*}}+\frac{\partial_r g^{\nu j*}}{\partial q^{\mu*}}\frac{\partial_r f}{\partial g^{\nu j*}}+\frac{\partial_r g^{\nu k*}}{\partial q^{\mu*}}\frac{\partial_r f}{\partial g^{\nu k*}}
\end{split}
\end{equation}
where $\mu, \nu\in\mathbb{H},\mu\nu \neq 0$.
\end{theorem}

\begin{corollary}[(Chain Rule of Right HR)]\label{cor:rightchainrl}
Let $S\subseteq \mathbb{H}$ and suppose $g:S\rightarrow \mathbb{H}$ has the right HR derivative at an interior point $q$ of the set $S$. Let $T\subseteq \mathbb{H}$ be such that $g(q)\in T$ for all $q \in S$. Assume $f:T\rightarrow \mathbb{H}$ has right HR derivatives at an inner point $g(q)\in T$, then the right HR derivatives of the composite function $f(g(q))$ are as follows:
\begin{equation}
\begin{split}
&\frac{\partial_r f(g(q))}{\partial q^{\mu}}=\frac{\partial_r g}{\partial q^{\mu}}\frac{\partial_r f}{\partial g}+\frac{\partial_r g^{i}}{\partial q^{\mu}}\frac{\partial_r f}{\partial g^{i}}+\frac{\partial_r g^{j}}{\partial q^{\mu}}\frac{\partial_r f}{\partial g^{j}}+\frac{\partial_r g^{k}}{\partial q^{\mu}}\frac{\partial_r f}{\partial g^{k}}\\
&\frac{\partial_r f(g(q))}{\partial q^{\mu*}}=\frac{\partial_r g}{\partial q^{\mu*}}\frac{\partial_r f}{\partial g}+\frac{\partial_r g^{i}}{\partial q^{\mu*}}\frac{\partial_r f}{\partial g^{i}}+\frac{\partial_r g^{j}}{\partial q^{\mu*}}\frac{\partial_r f}{\partial g^{j}}+\frac{\partial_r g^{k}}{\partial q^{\mu*}}\frac{\partial_r f}{\partial g^{k}}
\end{split}
\end{equation}
\begin{equation}
\begin{split}
&\frac{\partial_r f(g(q))}{\partial q^{\mu}}=\frac{\partial_r g}{\partial q^{\mu}}\frac{\partial_r f}{\partial g}+\frac{\partial_r g^{i*}}{\partial q^{\mu}}\frac{\partial_r f}{\partial g^{i*}}+\frac{\partial_r g^{j*}}{\partial q^{\mu}}\frac{\partial_r f}{\partial g^{j*}}+\frac{\partial_r g^{k*}}{\partial q^{\mu}}\frac{\partial_r f}{\partial g^{k*}}\\
&\frac{\partial_r f(g(q))}{\partial q^{\mu*}}=\frac{\partial_r g}{\partial q^{\mu*}}\frac{\partial_r f}{\partial g}+\frac{\partial_r g^{i*}}{\partial q^{\mu*}}\frac{\partial_r f}{\partial g^{i*}}+\frac{\partial_r g^{j*}}{\partial q^{\mu*}}\frac{\partial_r f}{\partial g^{j*}}+\frac{\partial_r g^{k*}}{\partial q^{\mu*}}\frac{\partial_r f}{\partial g^{k*}}
\end{split}
\end{equation}
where $\mu \in \{1,i,j,k\}$.
\end{corollary}

\begin{corollary}\label{cor:rightcompchainrl}
Let $S\subseteq \mathbb{H}$ and suppose $g:S\rightarrow \mathbb{C}$ has the right GHR derivative at an interior point $q$ of the set $S$. Let $T\subseteq \mathbb{C}$ be such that $g(q)\in T$ for all $q \in S$. Assume $f:T\rightarrow \mathbb{C}$ has the CR derivatives at an inner point $g(q)\in T$, then the right GHR derivatives of the composite function $f(g(q))$ are as follows:
\begin{equation}
\begin{split}
&\frac{\partial_r f(g(q))}{\partial q^{\mu}}=\frac{\partial_r g}{\partial q^{\mu}}\frac{\partial f}{\partial g}+\frac{\partial_r g^*}{\partial q^{\mu}}\frac{\partial f}{\partial g^*},\quad
\frac{\partial_r f(g(q))}{\partial q^{\mu*}}=\frac{\partial_r g}{\partial q^{\mu*}}\frac{\partial f}{\partial g}+\frac{\partial_r g^*}{\partial q^{\mu*}}\frac{\partial f}{\partial g^*}
\end{split}
\end{equation}
where $\mu\in\mathbb{H},\mu \neq 0$, $\frac{\partial f}{\partial g}$ and $\frac{\partial f}{\partial g^*}$ are the CR derivatives in CR calculus.
\end{corollary}

\begin{corollary}\label{cor:rightrealchainrl}
Let $S\subseteq \mathbb{H}$ and suppose $g:S\rightarrow \mathbb{R}$ has the right GHR derivative at an interior point $q$ of the set $S$. Let $T\subseteq \mathbb{R}$ be such that $g(q)\in T$ for all $q \in S$. Assume $f:T\rightarrow \mathbb{R}$ has real derivatives at an inner point $f(q)\in T$, then the right GHR derivatives of the composite function $f(g(q))$ are as follows:
\begin{equation}
\begin{split}
&\frac{\partial f(g(q))}{\partial q^{\mu}}=f'(g)\frac{\partial_r g}{\partial q^{\mu}},\quad
\frac{\partial f(g(q))}{\partial q^{\mu*}}=f'(g)\frac{\partial_r g}{\partial q^{\mu*}}
\end{split}
\end{equation}
where $\mu\in\mathbb{H},\mu \neq 0$ and $f'(g)$ is the real derivatives of real-valued function.
\end{corollary}

The proofs of Theorem \ref{thm:nwprightchainrl}, Corollaries \ref{cor:rightchainrl}, \ref{cor:rightcompchainrl} and \ref{cor:rightrealchainrl}  are essentially the same as those of Theorem \ref{thm:nwpleftchainrl}, Corollaries \ref{cor:leftchainrl}, \ref{cor:leftcompchainrl} and \ref{cor:leftrealchainrl}, thus omitted.

\subsection{Mean Value Theorem.}
The mean value theorem is one of the most important theoretical tools in calculus. In this section, we propose a version of mean valued theorem for quaternion-valued functions of quaternion variables.
\begin{theorem}[(Mean Value Theorem of Left Form)]\label{thm:qmvt}
 Let $f:S\subseteq \mathbb{H} \rightarrow\mathbb{H}$ be continuous and its left HR derivatives exist and are continuous in the set $S$. If $q_0,q_1\in S$ such that the segment joining them is also in $S$ then, then
\begin{equation}\label{eq:mvtqvf}
\begin{split}
f(q_1)-f(q_0)&=\int_0^1f(q_0+t\lambda)\left(\frac{\partial }{\partial q}\lambda+\frac{\partial }{\partial q^i}\lambda^i+\frac{\partial }{\partial q^j}\lambda^j+\frac{\partial }{\partial q^k}\lambda^k\right)dt\\
&=\int_0^1f(q_0+t\lambda)\left(\frac{\partial }{\partial q^*}\lambda^*+\frac{\partial }{\partial q^{i*}}\lambda^{i*}+\frac{\partial }{\partial q^{j*}}\lambda^{j*}+\frac{\partial }{\partial q^{k*}}\lambda^{k*}\right)dt
\end{split}
\end{equation}
where $\lambda=q_1-q_0$, $\frac{\partial }{\partial q},\frac{\partial }{\partial q^i},\frac{\partial }{\partial q^j},\frac{\partial }{\partial q^k}$ and so on are the left HR derivatives in Definition \ref{def:lefthr}.
\end{theorem}
\begin{proof}
Put $g(t)=f(q_0+t\lambda)$, $0\leq t\leq 1$. Then $g(t)$ is continuous on $[0,1]$ and has derivatives in $(0,1)$.
By using Corollary \ref{cor:leftchainrl}, the derivative of $g(t)$ can be found as
\begin{equation}\label{eq:dgdtmvt}
g'(t)=\frac{\partial f(q_0+t\lambda)}{\partial q}\lambda+\frac{\partial f(q_0+t\lambda)}{\partial q^i}\lambda^i+\frac{\partial f(q_0+t\lambda)}{\partial q^j}\lambda^j+\frac{\partial f(q_0+t\lambda)}{\partial q^k}\lambda^k
\end{equation}
By substituting the above expression into $g(1)-g(0)=\int_0^1g'(t)dt$ with $g(0)=f(q_0)$ and $g(1)=f(q_1)$, then the theorem follows. The second equality can be proved in similar manner.
\end{proof}

\begin{corollary}\label{cor:realmvncor}
Let $f:S\subseteq \mathbb{H} \rightarrow\mathbb{R}$ be continuous and its left HR derivatives exist and are continuous in the set $S$. If $q_0,q_1\in S$ such that the segment joining them is also in $S$, then
\begin{equation}\label{eq:mvtrealf}
\begin{split}
f(q_1)-f(q_0)&=4\int_0^1\mathfrak{R}\left(\frac{\partial f(q_0+t\lambda)}{\partial q}\lambda\right)dt=4\int_0^1\mathfrak{R}\left(\frac{\partial f(q_0+t\lambda)}{\partial q^*}\lambda^*\right)dt
\end{split}
\end{equation}
where $\lambda=q_1-q_0$, $\frac{\partial f}{\partial q}$ and $\frac{\partial f}{\partial q^*}$ are the left HR derivatives in Definition \ref{def:lefthr}.
\end{corollary}
\begin{proof}
Because $f$ is a real-valued function, then $\frac{\partial f}{\partial q}=\frac{\partial f^{\eta}}{\partial q}$, where $\eta\in\{1,i,j,k\}$. From \eqref{pr:pqmu} and \eqref{pr:lefthrpr2}, it follows that
\begin{equation}
\frac{\partial f}{\partial q^{\eta}}\lambda^{\eta}=\frac{\partial f^{\eta}}{\partial q^{\eta}}\lambda^{\eta}=\left(\frac{\partial f}{\partial q}\lambda\right)^{\eta}
\end{equation}
Hence, the corollary follows from \eqref{eq:z1qlink} and Theorem \ref{thm:qmvt}, the second equality can be derived by using $\mathfrak{R}(pq)=\mathfrak{R}(p^*q^*)$.
\end{proof}

If $\lambda$ is sufficiently small in the modulus, the right-hand side of \eqref{eq:mvtqvf} can be approximated as
\begin{equation}
f(q_1)-f(q_0)\approx f(q_0)\left(\frac{\partial }{\partial q}\lambda+\frac{\partial }{\partial q^i}\lambda^i+\frac{\partial }{\partial q^j}\lambda^j+\frac{\partial }{\partial q^k}\lambda^k\right)
\end{equation}
If the left HR derivatives of $f$ is Lipschitz continuous in the vicinity of $q$ and $q_1$ with the Lipschitz constant $L$, we can
estimate the error in this approximation as follows:
\begin{equation}
\begin{split}
&\left|f(q_1)-f(q_0)-f(q_0)\left(\frac{\partial }{\partial q}\lambda+\frac{\partial }{\partial q^i}\lambda^i+\frac{\partial }{\partial q^j}\lambda^j+\frac{\partial }{\partial q^k}\lambda^k\right)\right|\\
&=\left|\int_0^1[f(q_0+t\lambda)-f(q_0)]\left(\frac{\partial }{\partial q}\lambda+\frac{\partial }{\partial q^i}\lambda^i+\frac{\partial }{\partial q^j}\lambda^j+\frac{\partial }{\partial q^k}\lambda^k\right)dt\right|\\
&\leq4\left|\int_0^1Lt|\lambda|^2dt\right|=2L|\lambda|^2
\end{split}
\end{equation}

\begin{theorem}[(Mean Value Theorem of Right Form)]\label{thm:rtqmvt}
Let $f:S\subseteq \mathbb{H} \rightarrow\mathbb{H}$ be continuous and its right HR derivatives exist and are continuous in the set $S$. If $q_0,q_1\in S$ such that the segment joining them is also in $S$, then
\begin{equation}\label{eq:rtmvtqvf}
\begin{split}
f(q_1)-f(q_0)&=\int_0^1\left(\lambda\frac{\partial_r }{\partial q}+\lambda^i\frac{\partial_r }{\partial q^i}+\lambda^j\frac{\partial_r }{\partial q^j}+\lambda^k\frac{\partial_r }{\partial q^k}\right)f(q_0+t\lambda)dt\\
&=\int_0^1\left(\lambda^*\frac{\partial_r }{\partial q^*}+\lambda^{i*}\frac{\partial_r }{\partial q^{i*}}+\lambda^{j*}\frac{\partial_r }{\partial q^{j*}}+\lambda^{k*}\frac{\partial_r }{\partial q^{k*}}\right)f(q_0+t\lambda)dt
\end{split}
\end{equation}
where $\lambda=q_1-q_0$, $\frac{\partial_r }{\partial q},\frac{\partial_r }{\partial q^i},\frac{\partial_r }{\partial q^j}, \frac{\partial_r }{\partial q^k}$ and so on are the right HR derivatives in Definition \ref{def:righthr}.
\end{theorem}

\begin{corollary}\label{cor:rtrealmvncor}
Let $f:S\subseteq \mathbb{H} \rightarrow\mathbb{R}$ be continuous and its right HR derivatives exist and are continuous in the set $S$. If $q_0,q_1\in S$ such that the segment joining them is also in $S$, then
\begin{equation}\label{eq:rtmvtrealf}
\begin{split}
f(q_1)-f(q_0)&=4\int_0^1\mathfrak{R}\left(\lambda\frac{\partial_r f(q_0+t\lambda)}{\partial q}\right)dt=4\int_0^1\mathfrak{R}\left(\lambda^*\frac{\partial_r f(q_0+t\lambda)}{\partial q^*}\right)dt
\end{split}
\end{equation}
where $\lambda=q_1-q_0$, $\frac{\partial f}{\partial q}$ and $\frac{\partial f}{\partial q^*}$ are the right HR derivatives in Definition \ref{def:righthr}.
\end{corollary}

The proofs of Theorem \ref{thm:rtqmvt} and Corollary \ref{cor:rtrealmvncor}  are essentially the same as those of Theorem \ref{thm:qmvt} and Corollary \ref{cor:realmvncor}, thus omitted.

\subsection{Taylor's Theorem.}
In this section, we derive Taylor's theorem of quaternion-valued functions as a consequence of the univariate Taylor theorem.
\begin{lemma}[(Taylor's Theorem for Univariate Functions \cite{Apostol})]\label{lem:taylorthm1}
Let $f:D\subseteq\mathbb{R} \rightarrow\mathbb{R}$ be $(k+1)$-times continuously differentiable on an open interval D. If $x\in D$, then
\begin{equation}
f(x_0+h)=f(x_0)+f'(x_0)h+f''(x_0)\frac{h^2}{2}+\cdots+f^{(k)}(x_0)\frac{h^k}{k!}+R_k
\end{equation}
where the remainder $R_k$ is given by
\begin{equation}
R_k=\int_{x_0}^{x_0+h}\frac{f^{(k+1)}(t)}{k!}(x_0+h-t)^kdt
\end{equation}
\end{lemma}

\begin{theorem}[(Taylor's Theorem of Left Form)]\label{thm:lefttaylorthm}
 Let $f:S\subseteq \mathbb{H} \rightarrow\mathbb{H}$ be continuous and its 3-times left HR derivatives exist and are continuous in the set $S$. If $q_0,q_0+\lambda\in S$ such that the segment joining them is also in $S$, then
\begin{equation}\label{eq:lefttaylorthm}
\begin{split}
f(q_0+\lambda)-f(q_0)&=\sum_{\mu}\frac{\partial f(q_0)}{\partial q^{\mu}}\lambda^{\mu}+\frac{1}{2}\sum_{\mu,\nu}\frac{\partial^2 f(q_0)}{\partial q^{\nu}\partial q^{\mu}}\lambda^{\nu}\lambda^{\mu}+O(\lambda^3)\\
&=\sum_{\mu}\frac{\partial f(q_0)}{\partial q^{\mu*}}\lambda^{\mu*}+\frac{1}{2}\sum_{\mu,\nu}\frac{\partial^2 f(q_0)}{\partial q^{\nu*}\partial q^{\mu*}}\lambda^{\nu*}\lambda^{\mu*}+O(\lambda^3)\quad
as \;\lambda\rightarrow 0\\
\end{split}
\end{equation}
where $\mu,\nu\in\{1,i,j,k\}$, $\frac{\partial^2 f}{\partial q^{\nu}\partial q^{\mu}}$ and $\frac{\partial^2 f}{\partial q^{\nu*}\partial q^{\mu*}}$ are the second order HR derivatives, given in Section \ref{sec:hgderiv}.
\end{theorem}

\begin{proof}
Define the auxiliary function $g(t)=f(q_0+t\lambda)$ with $0\leq t\leq1$. By using the chain rule in Corollary \ref{cor:leftchainrl}, we obtain:
\begin{equation}
\begin{split}
&g'(t)=\sum_{\mu}\frac{\partial f(q_0+t\lambda)}{\partial q^{\mu}}\lambda^{\mu}\\
&g''(t)=\sum_{\mu,\nu}\frac{\partial^2 f(q_0+t\lambda)}{\partial q^{\nu}\partial q^{\mu}}\lambda^{\nu}\lambda^{\mu}\\
&g'''(t)=\sum_{\mu,\nu,\eta}\frac{\partial^3 f(q_0+t\lambda)}{\partial q^{\eta}\partial q^{\nu}\partial q^{\mu}}\lambda^{\eta}\lambda^{\nu}\lambda^{\mu}
\end{split}
\end{equation}
where $\mu,\nu,\eta\in\{1,i,j,k\}$. The second order Taylor polynomial in Lemma \ref{lem:taylorthm1} gives
\begin{equation}
g(1)=g(0)+g'(0)+\frac{1}{2}g''(0)+R_2
\end{equation}
This is equivalent to
\begin{equation}
f(q_0+\lambda)=f(q_0)+\sum_{\mu}\frac{\partial f(q_0)}{\partial q^{\mu}}\lambda^{\mu}+\frac{1}{2}\sum_{\mu,\nu}\frac{\partial^2 f(q_0)}{\partial q^{\nu}\partial q^{\mu}}\lambda^{\nu}\lambda^{\mu}+R_2
\end{equation}
where
\begin{equation}
R_2=\int_0^1\frac{(1-t)^2}{2}g'''(t)dt=\int_0^1\frac{(1-t)^2}{2}\sum_{\mu,\nu,\eta}\frac{\partial^3 f(q_0+t\lambda)}{\partial q^{\eta}\partial q^{\nu}\partial q^{\mu}}\lambda^{\eta}\lambda^{\nu}\lambda^{\mu}dt
\end{equation}
This integral contains three factors of $\lambda$ in it and the remaining factors are bounded. So, $R_2$ is of the order of $|\lambda|^3$ making the fraction $\frac{|R_2|}{|\lambda|^3}$ bounded, as $\lambda \rightarrow 0$. Hence, the first equality of the theorem follows, and the second equality can be proved in similar manner.
\end{proof}

\begin{corollary}\label{cor:lefttaylorcor}
Let $f:S\subseteq \mathbb{H} \rightarrow\mathbb{R}$ be continuous and its 3-times left HR derivatives exist and are continuous in the set $S$. If $q_0,q_0+\lambda\in S$ such that the segment joining them is also in $S$, then
\begin{equation}
\begin{split}
f(q_0+\lambda)-f(q_0)&=4\mathfrak{R}\left(\frac{\partial f(q_0)}{\partial q}\lambda\right)+2\sum_{\nu}\mathfrak{R}\left(\frac{\partial^2 f(q_0)}{\partial q^{\nu}\partial q}\lambda^{\nu}\lambda\right)+O(|\lambda|^3)\\
&=4\mathfrak{R}\left(\frac{\partial f(q_0)}{\partial q^*}\lambda^*\right)+2\sum_{\nu}\mathfrak{R}\left(\frac{\partial^2 f(q_0)}{\partial q^{\nu*}\partial q^*}\lambda^{\nu*}\lambda^*\right)+O(|\lambda|^3)\;\;
as \;\lambda\rightarrow 0
\end{split}
\end{equation}
where $\nu\in\{1,i,j,k\}$, $\frac{\partial^2 f}{\partial q^{\nu}\partial q}$ and $\frac{\partial^2 f}{\partial q^{\nu*}\partial q^*}$ are the second order HR derivatives in Section \ref{sec:hgderiv}
\end{corollary}
\begin{proof}
By the term $\frac{\partial \mathfrak{R}(q)}{\partial q}$ given in Table \ref{tb:fmlleftderiv} and the chain rule in Corollary \ref{cor:leftchainrl}, the corollary can be proved similar to the proof of Corollary \ref{cor:realmvncor}.
\end{proof}

\begin{theorem}[(Taylor's Theorem of Center Form)]\label{thm:cttaylorthm}
Let $f:S\subseteq \mathbb{H} \rightarrow\mathbb{R}$ be continuous and its 3-times left HR derivatives exist and are continuous in the set $S$. If $q_0,q_0+\lambda\in S$ such that the segment joining them is also in $S$, then
\begin{equation}\label{eq:cttaylorthm}
\begin{split}
f(q_0+\lambda)-f(q_0)&=\sum_{\mu}\frac{\partial f(q_0)}{\partial q^{\mu}}\lambda^{\mu}+\frac{1}{2}\sum_{\mu,\nu}\lambda^{\mu*}\frac{\partial^2 f(q_0)}{\partial q^{\nu}\partial q^{\mu*}}\lambda^{\nu}+O(\lambda^3)\\
&=\sum_{\mu}\frac{\partial f(q_0)}{\partial q^{\mu*}}\lambda^{\mu*}
+\frac{1}{2}\sum_{\mu,\nu}\lambda^{\mu}\frac{\partial^2 f(q_0)}{\partial q^{\nu*}\partial q^{\mu}}\lambda^{\nu*}+O(\lambda^3)\;\;
as \;\lambda\rightarrow 0\\
\end{split}
\end{equation}
where $\mu,\nu\in\{1,i,j,k\}$, $\frac{\partial^2 f}{\partial q^{\nu}\partial q^{\mu*}}$ and $\frac{\partial^2 f}{\partial q^{\nu*}\partial q^{\mu}}$ are the second order HR derivatives, given in Section \ref{sec:hgderiv}.
\end{theorem}

\begin{theorem}[(Taylor's Theorem of Right Form)]\label{thm:righttaylorthm}
 Let $f:S\subseteq \mathbb{H} \rightarrow\mathbb{H}$ be continuous and its 3-times right HR derivatives exist and are continuous in the set $S$. If $q_0,q_0+\lambda\in S$ such that the segment joining them is also in $S$, then
\begin{equation}
\begin{split}
f(q_0+\lambda)-f(q_0)&=\sum_{\mu}\lambda^{\mu}\frac{\partial_r f(q_0)}{\partial q^{\mu}}+\frac{1}{2}\sum_{\mu,\nu}\lambda^{\mu}\lambda^{\nu}\frac{\partial^2_r f(q_0)}{\partial q^{\nu}\partial q^{\mu}}+O(\lambda^3)\\
&=\sum_{\mu}\lambda^{\mu*}\frac{\partial_r f(q_0)}{\partial q^{\mu*}}+\frac{1}{2}\sum_{\mu,\nu}\lambda^{\mu*}\lambda^{\nu*}\frac{\partial^2_r f(q_0)}{\partial q^{\nu*}\partial q^{\mu*}}+O(\lambda^3)\;\;
as \;\lambda\rightarrow 0
\end{split}
\end{equation}
where $\mu,\nu\in\{1,i,j,k\}$, $\frac{\partial^2_r f}{\partial q^{\nu}\partial q^{\mu}}$ and $\frac{\partial^2_r f}{\partial q^{\nu*}\partial q^{\mu*}}$ are the second order HR derivatives in Section \ref{sec:hgderiv}
\end{theorem}

\begin{corollary}\label{cor:righttaylorcor}
Let $f:S\subseteq \mathbb{H} \rightarrow\mathbb{R}$ be continuous and its 3-times right HR derivatives exist and are continuous in the set $S$. If $q_0,q_0+\lambda\in S$ such that the segment joining them is also in $S$, then

\begin{equation}
\begin{split}
f(q_0+\lambda)-f(q_0)&=4\mathfrak{R}\left(\lambda\frac{\partial_r f(q_0)}{\partial q}\right)+2\sum_{\nu}\mathfrak{R}\left(\lambda\lambda^{\nu}\frac{\partial^2_r f(q_0)}{\partial q^{\nu}\partial q}\right)+O(|\lambda|^3)\\
&=4\mathfrak{R}\left(\lambda^*\frac{\partial_r f(q_0)}{\partial q^*}\right)+2\sum_{\nu}\mathfrak{R}\left(\lambda^*\lambda^{\nu*}\frac{\partial^2_r f(q_0)}{\partial q^{\nu*}\partial q^*}\right)+O(|\lambda|^3)\;as \;\lambda\rightarrow 0
\end{split}
\end{equation}\label{eq:rttaylorrealf}
where $\nu\in\{1,i,j,k\}$, $\frac{\partial^2_r f}{\partial q^{\nu}\partial q}$ and $\frac{\partial^2_r f}{\partial q^{\nu*}\partial q^*}$ are the second order HR derivatives in Section \ref{sec:hgderiv}
\end{corollary}

The proofs of Theorem \ref{thm:cttaylorthm} and \ref{thm:righttaylorthm} and Corollary \ref{cor:righttaylorcor}  are essentially the same as those of Theorem \ref{thm:lefttaylorthm} and Corollary \ref{cor:lefttaylorcor}, thus omitted.

\begin{remark}
The Taylor expansion in Theorem \ref{thm:lefttaylorthm} is concisely expressed using the HR derivatives. This is different from the Taylor expansion given by Schwartz \cite{Schwartz}, which decomposes a quaternion $q$ into two mutually perpendicular quaternions in a local coordinate system. In contrast, our idea is to extend the quaternion $q$ as an augmented quaternion based on quaternion involutions \cite{Ell}. Schwartz has also stated that his Taylor expansion would cause trouble when the function has terms $q\omega q$, where $\omega$ is a general quaternion, which limits the admissible class of functions to be real functions, the fixed point to be real and so on. Notice that there are no such restrictions in Theorem \ref{thm:lefttaylorthm}, which only requires the same condition as that in \cite{Schwartz}, that is, the functions $f$ should be  real analytic functions.
\end{remark}

\subsection{Method of Steepest Descent.}\label{sec:steepdesc}
In the CR calculus, the steepest descent direction of the real-valued function $f(z)$ is expressed as $-\frac{\partial f}{\partial z^*}$ in \cite{Brandwood}. For the quaternion case, we now show that the direction of steepest descent of real-valued function $f(q)$ is $-\frac{\partial f}{\partial q^*}$, that is, a generic extension from that in $\mathbb{R}$ and $\mathbb{C}$. Using the first order Taylor expansion in Corollary \ref{cor:lefttaylorcor}, we have
\begin{equation}\label{eq:sd1taylorexp}
\begin{split}
f(q)-f(q_k)= 4 \mathfrak{R}\left(\frac{\partial f(q_k)}{\partial q}(q-q_k)\right)+O\left(|q-q_k|^2\right)
\end{split}
\end{equation}
Set $q-q_k=\alpha d_k, \alpha\in\mathbb{R}^+$, then the second term can be neglected by shrinking $\alpha$.
In this case, $d_k$ is the direction of steepest descent if and only if $\mathfrak{R}\left(\frac{\partial f(q_k)}{\partial q}d_k\right)<0$.
 Obviously, in order to decrease $f(q)-f(q_k)$, the fastest way is to minimize $\mathfrak{R}\left(\frac{\partial f(q_k)}{\partial q}d_k\right)$. It then follows that
\begin{equation}
\left|\mathfrak{R}\left(\frac{\partial f(q_k)}{\partial q}d_k\right)\right|\leq \left|\frac{\partial f(q_k)}{\partial q}d_k\right|=
\left|\frac{\partial f(q_k)}{\partial q}\right|\left|d_k\right|
\end{equation}
Hence, the $\mathfrak{R}\left(\frac{\partial f(q_k)}{\partial q}d_k\right)$ is minimized if and only if $d_k=-\left(\frac{\partial f(q_k)}{\partial q}\right)^*=-\frac{\partial f(q_k)}{\partial q^*}$, showing that the steepest descent direction of $f(q)$ is $-\frac{\partial f(q_k)}{\partial q^*}$. The iterative rule of the steepest descent method can therefore be expressed as
\begin{equation}
q_{k+1}=q_k-\alpha\frac{\partial f(q_k)}{\partial q^*}
\end{equation}
where $\alpha\in\mathbb{R}^+$ is the step size.

\section{Applications of the GHR calculus}
The GHR calculus has an important significance in quaternion analysis, and can be used in optimization, statistics, signal processing, machine learning and other fields.

\subsection{The GHR Derivatives of Elementary Functions.}\label{sec:elemfuns}
We now present some of quaternion-valued derivatives of elementary functions, these functions are often used in nonlinear adaptive filters and quaternion-valued neural networks.

\begin{example}[(Power Function)]
Find the GHR derivative of the power function $f:\mathbb{H}\rightarrow \mathbb{H}$ given by
\begin{equation}
f(q)=q^n
\end{equation}
where $n$ is any positive integer number.
\end{example}
\textbf{Solution}: By using the product rule in Theorem \ref{thm:nwpleftpdrl}, it follows that
\begin{equation}
\begin{split}
\frac{\partial q^n}{\partial q^{\mu}}\mu&=\frac{\partial (qq^{n-1})}{\partial q^{\mu}}\mu=q\frac{\partial q^{n-1}}{\partial q^{\mu}}\mu+\frac{\partial q}{\partial q^{q^{n-1}\mu}}q^{n-1}\mu=q\frac{\partial q^{n-1}}{\partial q^{\mu}}\mu+\mathfrak{R}(q^{n-1}\mu)
\end{split}
\end{equation}
where the term $\frac{\partial q}{\partial q^{\mu}}\mu$, given in Table \ref{tb:fmlleftderiv}, was used in the last equality. Note that the above expression is recurrent about $\frac{\partial q^n}{\partial q^{\mu}}\mu$. Expanding this expression and using the initial condition $\frac{\partial q}{\partial q^{\mu}}\mu=\mathfrak{R}(\mu)$, yields
\begin{equation}\label{eq:res1powfun}
\frac{\partial q^n}{\partial q^{\mu}}\mu=\sum_{m=1}^nq^{n-m}\mathfrak{R}(q^{m-1}\mu)
\end{equation}
In a similar manner, we have
\begin{equation}
\begin{split}
\frac{\partial q^n}{\partial q^{\mu*}}\mu&=\frac{\partial (qq^{n-1})}{\partial q^{\mu*}}\mu=q\frac{\partial q^{n-1}}{\partial q^{\mu*}}\mu+\frac{\partial q}{\partial q^{q^{n-1}\mu*}}q^{n-1}\mu=q\frac{\partial q^{n-1}}{\partial q^{\mu*}}\mu-\frac{1}{2}(q^{n-1}\mu)^*
\end{split}
\end{equation}
which is equivalent to
\begin{equation}
\frac{\partial q^n}{\partial q^{\mu*}}\mu=-\frac{1}{2}\sum_{m=1}^nq^{n-m}(q^{m-1}\mu)^*
\end{equation}

\begin{example}\label{exp:expfun}
(Exponential Function): Find the GHR derivative of the function $f:\mathbb{H}\rightarrow \mathbb{H}$ given by
\begin{equation}
\exp(q)\triangleq \sum_{n=0}^{+\infty}\frac{q^n}{n!}
\end{equation}
\end{example}
\textbf{Solution}:
From \eqref{eq:res1powfun},  it follows that
\begin{equation}\label{eq:expderivative1}
\begin{split}
&\frac{\partial \exp(q)}{\partial q^{\mu}}\mu=\sum_{n=0}^{+\infty}\frac{1}{n!}\sum_{m=1}^nq^{n-m}\mathfrak{R}(q^{m-1}\mu)
\end{split}
\end{equation}
In a similar manner, we have
\begin{equation}\label{eq:expderivative5}
\frac{\partial \exp(q)}{\partial q^{\mu*}}\mu=-\frac{1}{2}\sum_{n=0}^{+\infty}\frac{1}{n!}\sum_{m=1}^nq^{n-m}(q^{m-1}\mu)^*
\end{equation}

\begin{remark}
The exponential function is the most important elementary function,
as both trigonometric functions and hyperbolic functions can be expressed in terms of the exponential function.
The elementary function in Example \ref{exp:expfun} is a power series, and does not change the direction of the vector part of quaternion.
Therefore, such elementary functions can swap positions with a quaternion $q$,
i.e., $f(q)q=qf(q)$, giving an important property, $f^*(q)=f(q^*)$, which can be used in practical applications,
such as quaternion neural networks \cite{Ell} and quaternion nonlinear adaptive filters \cite{Ujang11}.
It is important to note that if the quaternion variable $q$ degenerates into a real variable $x$ in the definitions of elementary functions in this subsection,
then the GHR derivatives simplify into the real derivative, e.g., the GHR derivative of the power function in \eqref{eq:res1powfun} will become $nx^{n-1}$.
Therefore, the GHR derivatives are a generalized form of the real derivatives and the real derivatives are a special case of the GHR derivatives.
\end{remark}

\subsection{Derivation of the QLMS Algorithm}\label{sec:qlms1}
In this section, we rederive the quaternion least mean square (QLMS) algorithm given in \cite{Took09} according to the rules of the GHR calculus. The same real-valued quadratic cost function as in LMS and CLMS is used, that is
\begin{equation}
J(n)=|e(n)|^2=e^*(n)e(n)
\end{equation}
where
\begin{equation}
e(n)=d(n)-y(n),\quad y(n)={\bf w}^T(n){\bf x}(n)
\end{equation}
The weight update of QLMS is then given by
\begin{equation}\label{eq:qlmsllearnrule}
{\bf w}(n+1)={\bf w}(n)-\alpha \nabla_{{\bf w}^*} J(n)={\bf w}(n)-\alpha\left(\frac{\partial J(n)}{\partial {\bf w}^*}\right)^T
\end{equation}
where $\alpha$ is the step size, and $\nabla_{{\bf w}^*} J(n)$ is the conjugate gradient of $J(n)$ with respect to ${\bf w}^*$. In Section \ref{sec:steepdesc}, we have shown the conjugate gradient is the direction of steepest descent. By using the novel product rule in Corollary \ref{cor:nwpleftpdrl}, the gradient is therefore calculated by
\begin{equation}\label{eq:qlmsdJdwconj}
\frac{\partial J(n)}{\partial {\bf w}^*}=e^*(n)\frac{\partial e(n)}{\partial {\bf w}^*}+\frac{\partial e^*(n)}{\partial {\bf w}^{e(n)*}}e(n)
\end{equation}
To find $\nabla_{{\bf w}^*} J(n)$, we need to calculate the following two derivatives
\begin{equation}\label{eq:qlmsdedwconj}
\frac{\partial e(n)}{\partial {\bf w}^*}=\frac{\partial \left(d(n)-{\bf w}^T(n){\bf x}(n)\right)}{\partial {\bf w}^*}=-\frac{\partial \left({\bf w}^T(n){\bf x}(n)\right)}{\partial {\bf w}^*}=\frac{1}{2}{\bf x}^H(n)
\end{equation}
\begin{equation}\label{eq:qlmsdedwconj2}
\begin{split}
\frac{\partial e^*(n)}{\partial {\bf w}^{e(n)*}}e(n)&=\frac{\partial \left(d^*(n)-{\bf x}^H(n){\bf w}^*(n)\right)}{\partial {\bf w}^{e(n)*}}e(n)\\
&=-\frac{\partial \left({\bf x}^H(n){\bf w}^*(n)\right)}{\partial {\bf w}^{e(n)*}}e(n)=-{\bf x}^H(n)\mathfrak{R}(e(n))
\end{split}
\end{equation}
where the terms $\frac{\partial (q\nu)}{\partial q^*}$ and $\frac{\partial (\omega q^*)}{\partial q^{\mu*}}\mu$ are given in Table \ref{tb:fmlleftderiv}, and are used in the
last equalities above. Substituting \eqref{eq:qlmsdedwconj} and \eqref{eq:qlmsdedwconj2} into \eqref{eq:qlmsdJdwconj} yields
\begin{equation}\label{eq:qlmsdJdwconj1b}
\begin{split}
\nabla_{{\bf w}^*} J(n)&=\left(\frac{\partial J(n)}{\partial {\bf w}^*}\right)^T=\frac{1}{2}e^*(n){\bf x}^*(n)-{\bf x}^*(n)\mathfrak{R}(e(n))\\
&=\left(\frac{1}{2}e^*(n)-\mathfrak{R}(e(n))\right){\bf x}^*(n)=-\frac{1}{2}e(n){\bf x}^*(n)
\end{split}
\end{equation}
Finally, the update of the adaptive weight vector of QLMS becomes
\begin{equation}\label{eq:qlmsllearnrule1b}
{\bf w}(n+1)={\bf w}(n)+\alpha\, e(n){\bf x}^*(n)
\end{equation}
where the constant $\frac{1}{2}$ in \eqref{eq:qlmsdJdwconj1b} is absorbed into $\alpha$.
\begin{remark}
Note that if we start from $y(n)={\bf w}^H(n){\bf x}(n)$, the final update rule would become ${\bf w}(n+1)={\bf w}(n)+\alpha\, {\bf x}(n)e^*(n)$.
The QLMS algorithm in \eqref{eq:qlmsllearnrule1b} is different from the QLMS in \cite{Took09}, due to the use of different product rule.
Although, the traditional product rule was used to derive the weight update rule in \cite{Took09}, our counter-examples in Section \ref{sec:limithr} illustrate that the traditional product rule is not applicable for the HR calculus. We therefore use the novel product rule within the GHR calculus to derive the weight update rule in \eqref{eq:qlmsdJdwconj}. Another advantage of QLMS in \eqref{eq:qlmsllearnrule1b}, derived based the GHR calculus, is that it has the same generic form as that of the CMLS \cite{Widrow}.
\end{remark}

\subsection{Derivation of the WL-QLMS Algorithm}
In this section, we rederive the WL-QLMS algorithm based on quaternion widely linear model given in \cite{Took11,Moreno12,Jvia10,Jvia11}. The cost function
to be minimized is a real-valued function of quaternion variables
\begin{equation}
J(n)=|e(n)|^2=e^*(n)e(n)
\end{equation}
\begin{equation}
e(n)=d(n)-y(n),\quad e^*(n)=d^*(n)-y^*(n)
\end{equation}
and
\begin{equation}\label{eq:nwwlqlmsoutput}
y(n)={\bf h}^H(n){\bf x}(n)+{\bf g}^H(n){\bf x}^i(n)+{\bf u}^H(n){\bf x}^j(n)+{\bf v}^H(n){\bf x}^k(n)
\end{equation}
The weight updates are then given by
\begin{equation}\label{eq:nwwlqlmsllearnrule}
\begin{split}
{\bf h}(n+1)={\bf h}(n)-\alpha \nabla_{{\bf h}^*} J(n),\quad {\bf u}(n+1)={\bf u}(n)-\alpha \nabla_{{\bf u}^*} J(n)\\
{\bf g}(n+1)={\bf g}(n)-\alpha \nabla_{{\bf g}^*} J(n), \quad {\bf v}(n+1)={\bf v}(n)-\alpha \nabla_{{\bf v}^*} J(n)
\end{split}
\end{equation}
where $\alpha$ is the step size, $\nabla_{{\bf h}^*} J(n)$, $\nabla_{{\bf g}^*} J(n)$, $\nabla_{{\bf u}^*} J(n)$ and $\nabla_{{\bf v}^*} J(n)$ are respectively the conjugate gradients of $J(n)$ with respect to ${\bf h}^*$, ${\bf g}^*$, ${\bf u}^*$ and ${\bf v}^*$. By using the novel product rule in Corollary \ref{cor:nwpleftpdrl}, $\nabla_{{\bf h}^*} J(n)$ is calculated by
\begin{equation}\label{eq:iqlmsdJdwconj}
\frac{\partial J(n)}{\partial {\bf h}^*}=e^*(n)\frac{\partial e(n)}{\partial {\bf h}^*}+\frac{\partial e^*(n)}{\partial {\bf h}^{e(n)*}}e(n)
\end{equation}
To find $\nabla_{{\bf h}^*} J(n)$, we need to calculate the following two derivatives
\begin{equation}\label{eq:iqlmsdedwconj}
\frac{\partial e(n)}{\partial {\bf h}^*}=-\frac{\partial \left({\bf h}^H(n){\bf x}(n)\right)}{\partial {\bf h}^*}=-\mathfrak{R}({\bf x}^T(n))
\end{equation}
\begin{equation}\label{eq:iqlmsdeconjdwconj}
\frac{\partial e^*(n)}{\partial {\bf h}^{e(n)*}}e(n)=-\frac{\partial \left({\bf x}^H(n){\bf h}(n)\right)}{\partial {\bf h}^{e(n)*}}e(n)=\frac{1}{2}{\bf x}^H(n)e^*(n)
\end{equation}
where the terms $\frac{\partial (q^*\nu)}{\partial q^*}$ and $\frac{\partial (\omega q)}{\partial q^{\mu*}}\mu$ are given in Table \ref{tb:fmlleftderiv},
and are used in the last equalities above. Substituting \eqref{eq:iqlmsdedwconj} and \eqref{eq:iqlmsdeconjdwconj} into \eqref{eq:iqlmsdJdwconj} yields
\begin{equation}\label{eq:iqlmsdJdwconj1b}
\begin{split}
\nabla_{{\bf h}^*} J(n)&=\left(\frac{\partial J(n)}{\partial {\bf h}^*}\right)^T=-e^*(n)\mathfrak{R}({\bf x}(n))+\frac{1}{2}{\bf x}^*(n)e^*(n)\\
&=\left(-\mathfrak{R}({\bf x}(n))+\frac{1}{2}{\bf x}^*(n)\right)e^*(n)=-\frac{1}{2}{\bf x}(n)e^*(n)
\end{split}
\end{equation}
The gradients $\nabla_{{\bf g}^*} J(n)$, $\nabla_{{\bf u}^*} J(n)$ and $\nabla_{{\bf v}^*} J(n)$  can be calculated in a similar way to \eqref{eq:iqlmsdJdwconj1b} and are given by
\begin{equation}\label{eq:nwwlqlmsdJdwconj1c}
\begin{split}
&\nabla_{{\bf g}^*} J(n)=-\frac{1}{2}{\bf x}^i(n)e^*(n),\quad \nabla_{{\bf v}^*} J(n)=-\frac{1}{2}{\bf x}^k(n)e^*(n)\\
&\nabla_{{\bf u}^*} J(n)=-\frac{1}{2}{\bf x}^j(n)e^*(n)
\end{split}
\end{equation}
Finally, the weight update within WL-QLMS can be expressed as
\begin{equation}\label{eq:nwwlqlmsllearnruleww}
\begin{split}
&{\bf h}(n+1)={\bf h}(n)+\alpha\, {\bf x}(n)e^*(n),\quad {\bf u}(n+1)={\bf u}(n)+\alpha\, {\bf x}^j(n)e^*(n)\\
&{\bf g}(n+1)={\bf g}(n)+\alpha\,{\bf x}^i(n)e^*(n),\quad {\bf v}(n+1)={\bf v}(n)+\alpha\,{\bf x}^k(n)e^*(n)
\end{split}
\end{equation}
where the constant $\frac{1}{2}$ in \eqref{eq:nwwlqlmsdJdwconj1c} is absorbed into $\alpha$.
\begin{remark}
There are many variations of WL-QLMS algorithms, such as the WL-QLMS algorithms based on variants $\{q^*,q^{i*},q^{j*},q^{k*}\}$, $\{q^{\mu},q^{\mu i},q^{\mu j},q^{\mu k}\}$ and $\{q^*,q^{\mu i*},q^{\mu j*},q^{\mu k*}\}$, $\mu\in \mathbb{H}$. If the Hermitian operator in \eqref{eq:nwwlqlmsoutput} is replaced with the transpose operator, another variant of WL-QLMS algorithm can be found in \cite{Ujang13,Jahanchahi10}.
\end{remark}

\subsection{Derivation of Quaternion Nonlinear Adaptive Filtering Algorithm.}
In this section, we derive the Quaternion nonlinear gradient descent (QNGD) algorithm given in \cite{Ujang11} according to the rules of GHR calculus. The same real-valued quadratic cost function as in LMS and CLMS is used, that is
\begin{equation}
J(n)=|e(n)|^2
\end{equation}
where
\begin{equation}
e(n)=d(n)-\Phi(s(n)),\quad s(n)={\bf w}^T(n){\bf x}(n)
\end{equation}
and $\Phi$ is the quaternion nonlinearity.  The weight update is given by
\begin{equation}\label{eq:nlllearnrule}
{\bf w}(n+1)={\bf w}(n)-\alpha \nabla_{{\bf w}^*} J(n)={\bf w}(n)-\alpha\left(\frac{\partial J(n)}{\partial {\bf w}^*}\right)^T
\end{equation}
where $\alpha$ is the step size, and $\nabla_{{\bf w}^*} J(n)$ is the conjugate gradient of $J(n)$ with respect to ${\bf w}^*$. By using the chain rule in Corollary \ref{cor:leftchainrl}, the gradient is calculated by
\begin{equation}\label{eq:nldJdwconj}
\frac{\partial J(n)}{\partial {\bf w}^*}
=\sum_{\mu\in\{1,i,j,k\}}\frac{\partial |e(n)|^2}{\partial e^{\mu*}(n)}\frac{\partial e^{\mu*}(n)}{\partial {\bf w}^*}
\end{equation}
where the derivatives of $|e(n)|^2$ can be calculated using the term $\frac{\partial |q|^2}{\partial q^{\mu*}}\mu$ in Table \ref{tb:fmlleftderiv}. It then follows that
\begin{equation}\label{eq:nldme2de}
\frac{\partial |e(n)|^2}{\partial e^{\mu*}(n)}=\frac{1}{2}e^{\mu}(n),\quad \forall \mu\in\{1,i,j,k\}
\end{equation}
To find $\nabla_{{\bf w}^*} J(n)$, we need to calculate the following derivative
\begin{equation}\label{eq:nldedwconj}
\begin{split}
\frac{\partial e^{\mu*}(n)}{\partial {\bf w}^*}&=\frac{\partial \left(d^{\mu*}(n)-y^{\mu*}(n)\right)}{\partial {\bf w}^*}
=-\frac{\partial y^{\mu*}(n)}{\partial {\bf w}^*}\\
&=-\frac{\partial \Phi^{\mu*}(s(n))}{\partial {\bf w}^*},\quad \forall \mu\in\{1,i,j,k\}
\end{split}
\end{equation}
By the chain rule in Corollary \ref{cor:leftchainrl}, it follows that
\begin{equation}\label{eq:nldphisdwconj}
\begin{split}
&\frac{\partial \Phi^{\mu*}(s(n))}{\partial {\bf w}^*}
=\sum_{\nu\in\{1,i,j,k\}}\frac{\partial \Phi^{\mu*}(s(n))}{\partial s^{\nu*}(n)}\frac{\partial s^{\nu*}(n)}{\partial {\bf w}^*}
\end{split}
\end{equation}
where the HR derivatives of $s^*(n)$ can be calculated by using the term $\frac{\partial (\omega q^*)}{\partial q^{*}}$ in Table \ref{tb:fmlleftderiv}, giving
\begin{equation}\label{eq:nldsconjdwconj}
\frac{\partial s^*(n)}{\partial {\bf w}^*}=\frac{\partial \left({\bf x}^H(n){\bf w}^*(n)\right)}{\partial {\bf w}^*}={\bf x}^H(n)
\end{equation}
Using the property \eqref{pr:lefthrpr2} and the term $\frac{\partial (\omega q^*)}{\partial q^{\mu*}}\mu$ in Table \ref{tb:fmlleftderiv}, we have
\begin{equation}\label{eq:nldsiconjdwconj}
\begin{split}
\frac{\partial s^{\nu*}(n)}{\partial {\bf w}^*}&=\left(\frac{\partial s^*(n)}{\partial {\bf w}^{\nu*}}\right)^{\nu}=-\nu\frac{\partial \left({\bf x}^H(n){\bf w}^*(n)\right)}{\partial {\bf w}^{\nu*}}\nu\\
&=-\nu{\bf x}^H(n)\mathfrak{R}(\nu)=0,\quad \forall \nu\in\{i,j,k\}
\end{split}
\end{equation}
Substituting \eqref{eq:nldsconjdwconj} and \eqref{eq:nldsiconjdwconj} into \eqref{eq:nldphisdwconj}, we arrive at
\begin{equation}\label{eq:nldphisdwconjres}
\begin{split}
\frac{\partial \Phi^{\mu*}(s(n))}{\partial {\bf w}^*}&=\frac{\partial \Phi^{\mu*}(s(n))}{\partial s^*(n)}{\bf x}^H(n),\quad \forall \mu\in\{1,i,j,k\}
\end{split}
\end{equation}
By combining \eqref{eq:nldme2de}, \eqref{eq:nldedwconj} and \eqref{eq:nldphisdwconjres} with \eqref{eq:nldJdwconj}, we obtain
\begin{equation}\label{eq:nldJdwconjres}
\nabla_{{\bf w}^*} J(n)=\left(\frac{\partial J(n)}{\partial {\bf w}^*}\right)^T=-\frac{1}{2}\sum_{\mu\in\{1,i,j,k\}}e^{\mu}(n)\frac{\partial \Phi^{\mu*}(s(n))}{\partial s^*(n)}{\bf x}^*
\end{equation}
Finally, the update of the adaptive weight vector of QNGD algorithm can be expressed as
\begin{equation}\label{eq:nonlinearllearnrule1b}
{\bf w}(n+1)={\bf w}(n)+\alpha\sum_{\mu\in\{1,i,j,k\}}e^{\mu}(n)\frac{\partial \Phi^{\mu*}(s(n))}{\partial s^*(n)}{\bf x}^*
\end{equation}
where the constant $\frac{1}{2}$ in \eqref{eq:nldJdwconjres} is absorbed into $\alpha$.
\begin{remark}
If the function $\Phi(q)=q$, then the QNGD algorithm will degenerate into QLMS in Section \ref{sec:qlms1}, that is, the QLMS algorithm is a special case of QNGD. Using the GHR calculus for the QNGD algorithm, the nonlinear function $\Phi$ does not need to satisfy the odd-symmetry condition required in \cite{Ujang11}. We can also derive the augmented QNGD (AQNGD) and widely linear QNGD (WL-QNGD) algorithms in the same way. In order to save space, we leave this to the interested readers.
\end{remark}

\section{Conclusions}
A new framework for the efficient
computation of quaternion derivatives, termed the GHR calculus, has been established.
The proposed methodology has been shown to greatly relax the existence condition for the derivatives of functions of quaternion variables,
and to simplify the calculation of quaternion derivatives through its novel product and chain rule, unlike the existing quaternion derivatives, the GHR calculus is general and can be used for both analytic and non-analytic functions of quaternion variables,
The core of the GHR calculus is the use of quaternion rotation to overcome the non-commutativity of quaternion product,
and the use of quaternion involutions to obtain a quaternion basis, such as $\{q,q^i,q^j,q^k\}$ or their conjugates.
The use of quaternion involutions has been instrumental in establishing two fundamental results: the quaternion mean value theorem and Taylor's theorem.
The proposed framework allows for real- and complex-valued optimization algorithms to be extended to the quaternion field
in a generic, compact and intuitive way. Illustrative examples in adaptive signal processing demonstrate the effectiveness of the proposed framework.

\section*{Appendix.}
In this section, we give the detail of proofs of theorems and combine some important results into a table to make them
more accessible.
\subsection*{Appendix A: The derivation of HR calculus.}
For any quaternion-valued function $f(q)\in \mathbb{H}$, we shall start
from (since the fields $\mathbb{H}$ and $\mathbb{R}^4$ are isomorphic)
\begin{equation}
\begin{split}
f(q)&=f_a(q_a,q_b,q_c,q_d)+if_b(q_a,q_b,q_c,q_d)+jf_c(q_a,q_b,q_c,q_d)+kf_d(q_a,q_b,q_c,q_d)
\end{split}
\end{equation}
where $f_a(\cdot),f_b(\cdot),f_c(\cdot),f_d(\cdot)\in \mathbb{R}$. Then, the function $f$ can be seen as a function of the four independent real-valued variables $q_a,q_b,q_c$ and $q_d$, and the differential of $f$ can be expressed as follows \cite{Sudbery}:
\begin{equation}\label{eq:case1df}
df=\frac{\partial f}{\partial q_a}dq_a+\frac{\partial f}{\partial q_b}dq_b+\frac{\partial f}{\partial q_c}dq_c+\frac{\partial f}{\partial q_d}dq_d
\end{equation}
or
\begin{equation}\label{eq:case2df}
df=dq_a\frac{\partial f}{\partial q_a}+dq_b\frac{\partial f}{\partial q_b}+dq_c\frac{\partial f}{\partial q_c}+dq_d\frac{\partial f}{\partial q_d}
\end{equation}
where $\frac{\partial f}{\partial q_a}$, $\frac{\partial f}{\partial q_b}$, $\frac{\partial f}{\partial q_c}$ and $\frac{\partial f}{\partial q_d}$ are the partial derivatives of $f$ with respect to $q_a$, $q_b$, $q_c$ and $q_d$, respectively. Note that the two equations are identical since $dq_a$, $dq_b$, $dq_c$ and $dq_d$ are real quantities. As a result, both equations are equally valid as a starting point for the derivation of the HR calculus.

\textit{The Left Case \eqref{eq:case1df}.}
There are two ways \eqref{eq:z1qlink} and \eqref{eq:z1conjqlink}  to link the real and quaternion differentials, which correspond to the HR derivatives and conjugate ${\rm HR}$ derivatives, respectively.
\textit{A1. The Left HR Derivatives}: From \eqref{eq:z1qlink}, the differentials of the components of a quaternion can be formulated as
\begin{equation}\label{eq:cas1dzdqlink}
\begin{split}
&dq_a=\frac{1}{4}(dq+dq^i+dq^j+dq^k),\quad  dq_b=-\frac{i}{4}(dq+dq^i-dq^j-dq^k)\\
&dq_c=-\frac{j}{4}(dq-dq^i+dq^j-dq^k),\quad  dq_d=-\frac{k}{4}(dq-dq^i-dq^j+dq^k)
\end{split}
\end{equation}
By inserting \eqref{eq:cas1dzdqlink} into \eqref{eq:case1df}, the differential of the function $f$ becomes
\begin{equation}\label{eq:case1dfdq}
\begin{split}
df&=\frac{1}{4}\frac{\partial f}{\partial q_a}(dq+dq^i+dq^j+dq^k)-\frac{1}{4}\frac{\partial f}{\partial q_b}i(dq+dq^i-dq^j-dq^k)\\
&\quad-\frac{1}{4}\frac{\partial f}{\partial q_b}j(dq-dq^i+dq^j-dq^k)-\frac{1}{4}\frac{\partial f}{\partial q_b}k(dq-dq^i-dq^j+dq^k)\\
&=\frac{1}{4}\left(\frac{\partial f}{\partial q_a}-\frac{\partial f}{\partial q_b}i-\frac{\partial f}{\partial q_c}j-\frac{\partial f}{\partial q_d}k\right)dq+\frac{1}{4}\left(\frac{\partial f}{\partial q_a}-\frac{\partial f}{\partial q_b}i+\frac{\partial f}{\partial q_c}j+\frac{\partial f}{\partial q_d}k\right)dq^i\\
&\quad+\frac{1}{4}\left(\frac{\partial f}{\partial q_a}+\frac{\partial f}{\partial q_b}i-\frac{\partial f}{\partial q_c}j+\frac{\partial f}{\partial q_d}k\right)dq^j+\frac{1}{4}\left(\frac{\partial f}{\partial q_a}+\frac{\partial f}{\partial q_b}i+\frac{\partial f}{\partial q_c}j-\frac{\partial f}{\partial q_d}k\right)dq^k\\
\end{split}
\end{equation}
Now, define the formal left derivatives $\frac{\partial f}{\partial q},\frac{\partial f}{\partial q^i},\frac{\partial f}{\partial q^j}$ and $\frac{\partial f}{\partial q^k}$ so that
\begin{equation}\label{eq:case1dfdqcmp}
df=\frac{\partial f}{\partial q}dq+\frac{\partial f}{\partial q^i}dq^i+\frac{\partial f}{\partial q^j}dq^j+\frac{\partial f}{\partial q^k}dq^k
\end{equation}
holds. Comparing \eqref{eq:case1dfdqcmp} with \eqref{eq:case1dfdq} and applying Lemma \ref{lm:dpindp}, we can obtain the following left HR derivatives
\begin{equation}
\left(\begin{array}{l}
\frac{\partial f}{\partial q}\\
\frac{\partial f}{\partial q^i}\\
\frac{\partial f}{\partial q^j}\\
\frac{\partial f}{\partial q^k}
\end{array}\right)^T=\frac{1}{4}
\left(\begin{array}{l}
\frac{\partial f}{\partial q_a}\\
\frac{\partial f}{\partial q_b}\\
\frac{\partial f}{\partial q_c}\\
\frac{\partial f}{\partial q_d}
\end{array}\right)^T
\left(\begin{array}{cccc}
1 & 1  & 1 & 1 \\
-i & -i & i  & i\\
-j & j  & -j & j\\
-k & k  & k  & -k\\
\end{array}\right)
\end{equation}
\textit{A2. The Left Conjugate $HR$ Derivatives}: From \eqref{eq:z1conjqlink}, the differentials of the components of a quaternion can be formulated as follows
\begin{equation}\label{eq:case1bdzdqlink}
\begin{split}
&dq_a=\frac{1}{4}(dq^*+dq^{i*}+dq^{j*}+dq^{k*}),\quad dq_b=\frac{i}{4}(dq^*+dq^{i*}-dq^{j*}-dq^{k*})\\
&dq_c=\frac{j}{4}(dq^*-dq^{i*}+dq^{j*}-dq^{k*}),\quad dq_d=\frac{k}{4}(dq^*+dq^{i*}+dq^{j*}+dq^{k*})
\end{split}
\end{equation}
By inserting \eqref{eq:case1bdzdqlink} into \eqref{eq:case1df}, the differential of $f$ can be written as
\begin{equation}\label{eq:case1bdfdq}
\begin{split}
df&=\frac{1}{4}\frac{\partial f}{\partial q_a}(dq^*+dq^{i*}+dq^{j*}+dq^{k*})+\frac{1}{4}\frac{\partial f}{\partial q_b}i(dq^*+dq^{i*}-dq^{j*}-dq^{k*})\\
&\quad+\frac{1}{4}\frac{\partial f}{\partial q_b}j(dq^*-dq^{i*}+dq^{j*}-dq^{k*})+\frac{1}{4}\frac{\partial f}{\partial q_b}k(dq^*-dq^{i*}-dq^{j*}+dq^{k*})\\
&=\frac{1}{4}\left(\frac{\partial f}{\partial q_a}+\frac{\partial f}{\partial q_b}i+\frac{\partial f}{\partial q_c}j+\frac{\partial f}{\partial q_d}k\right)dq^*+\frac{1}{4}\left(\frac{\partial f}{\partial q_a}+\frac{\partial f}{\partial q_b}i-\frac{\partial f}{\partial q_c}j-\frac{\partial f}{\partial q_d}k\right)dq^{i*}\\
&\quad+\frac{1}{4}\left(\frac{\partial f}{\partial q_a}-\frac{\partial f}{\partial q_b}i+\frac{\partial f}{\partial q_c}j-\frac{\partial f}{\partial q_d}k\right)dq^{j*}+\frac{1}{4}\left(\frac{\partial f}{\partial q_a}-\frac{\partial f}{\partial q_b}i-\frac{\partial f}{\partial q_c}j+\frac{\partial f}{\partial q_d}k\right)dq^{k*}\\
\end{split}
\end{equation}
Now, we define the formal left derivatives $\frac{\partial f}{\partial q^*},\frac{\partial f}{\partial q^{i*}},\frac{\partial f}{\partial q^{j*}}$ and $\frac{\partial f}{\partial q^{k*}}$ so that
\begin{equation}\label{eq:case1bdfdqcmp}
df=\frac{\partial f}{\partial q^*}dq^*+\frac{\partial f}{\partial q^{i*}}dq^{i*}+\frac{\partial f}{\partial q^{j*}}dq^{j*}+\frac{\partial f}{\partial q^k}dq^{k*}
\end{equation}
holds. Comparing \eqref{eq:case1bdfdqcmp} with \eqref{eq:case1bdfdq} and applying Lemma \ref{lm:dpconjindp}, we can obtain the following left conjugate ${\rm HR}$ derivatives
\begin{equation}
\left(\begin{array}{l}
\frac{\partial f}{\partial q^*}\\
\frac{\partial f}{\partial q^{i*}}\\
\frac{\partial f}{\partial q^{j*}}\\
\frac{\partial f}{\partial q^{k*}}
\end{array}\right)^T=\frac{1}{4}
\left(\begin{array}{l}
\frac{\partial f}{\partial q_a}\\
\frac{\partial f}{\partial q_b}\\
\frac{\partial f}{\partial q_c}\\
\frac{\partial f}{\partial q_d}
\end{array}\right)^T
\left(\begin{array}{cccc}
1 & 1  & 1 & 1 \\
i & i & -i  & -i\\
j & -j  & j & -j\\
k & -k  & -k  & k\\
\end{array}\right)
\end{equation}

\textit{The Right Case \eqref{eq:case2df}.}
There are two ways \eqref{eq:z1qlink} and \eqref{eq:z1conjqlink}  to link the real and quaternion differentials, which correspond the HR derivatives and conjugate ${\rm HR}$ derivatives, respectively.

\textit{A3. The right HR Derivatives}: By applying a rotation transformation on both sides of \eqref{eq:cas1dzdqlink}, we have
\begin{equation}\label{eq:cas2dzdqlink}
\begin{split}
&dq_a=\frac{1}{4}(dq+dq^i+dq^j+dq^k),\quad dq_b=(dq_b)^i=-\frac{1}{4}(dq+dq^i-dq^j-dq^k)i\\
&dq_c=-\frac{1}{4}(dq-dq^i+dq^j-dq^k)j,\quad dq_d=(dq_d)^k=-\frac{1}{4}(dq-dq^i-dq^j+dq^k)k
\end{split}
\end{equation}
Then, by substituting \eqref{eq:cas2dzdqlink} into \eqref{eq:case2df}, the differential of $f$ becomes
\begin{equation}\label{eq:case2dfdq}
\begin{split}
&df=\frac{1}{4}(dq+dq^i+dq^j+dq^k)\frac{\partial f}{\partial q_a}-\frac{1}{4}(dq+dq^i-dq^j-dq^k)i\frac{\partial f}{\partial q_b}\\
&\quad-\frac{1}{4}(dq-dq^i+dq^j-dq^k)j\frac{\partial f}{\partial q_b}-\frac{1}{4}(dq-dq^i-dq^j+dq^k)k\frac{\partial f}{\partial q_b}\\
&=\frac{1}{4}dq\left(\frac{\partial f}{\partial q_a}-i\frac{\partial f}{\partial q_b}-j\frac{\partial f}{\partial q_c}-k\frac{\partial f}{\partial q_d}\right)+\frac{1}{4}dq^i\left(\frac{\partial f}{\partial q_a}-i\frac{\partial f}{\partial q_b}+j\frac{\partial f}{\partial q_c}+k\frac{\partial f}{\partial q_d}\right)\\
&\quad+\frac{1}{4}dq^j\left(\frac{\partial f}{\partial q_a}+i\frac{\partial f}{\partial q_b}-j\frac{\partial f}{\partial q_c}+k\frac{\partial f}{\partial q_d}\right)+\frac{1}{4}dq^k\left(\frac{\partial f}{\partial q_a}+i\frac{\partial f}{\partial q_b}+j\frac{\partial f}{\partial q_c}-k\frac{\partial f}{\partial q_d}\right)\\
\end{split}
\end{equation}
Now, define the formal right derivatives $\frac{\partial_r f}{\partial q},\frac{\partial_r f}{\partial q^i},\frac{\partial_r f}{\partial q^j}$ and $\frac{\partial_r f}{\partial q^k}$ so that
\begin{equation}\label{eq:case2dfdqcmp}
df=dq\frac{\partial_r f}{\partial q}+dq^i\frac{\partial_r f}{\partial q^i}+dq^j\frac{\partial_r f}{\partial q^j}+dq^k\frac{\partial_r f}{\partial q^k}
\end{equation}
holds. Comparing \eqref{eq:case2dfdqcmp} and \eqref{eq:case2dfdq} and using Lemma \ref{lm:dpindp}, we obtain the following right HR derivatives
\begin{equation}
\left(\begin{array}{l}
\frac{\partial_r f}{\partial q}\\
\frac{\partial_r f}{\partial q^i}\\
\frac{\partial_r f}{\partial q^j}\\
\frac{\partial_r f}{\partial q^k}
\end{array}\right)=\frac{1}{4}
\left(\begin{array}{cccc}
1 & -i  & -j & -k \\
1 & -i & j  & k\\
1 & i  & -j & k\\
1 & i  & j  & -k\\
\end{array}\right)
\left(\begin{array}{l}
\frac{\partial f}{\partial q_a}\\
\frac{\partial f}{\partial q_b}\\
\frac{\partial f}{\partial q_c}\\
\frac{\partial f}{\partial q_d}
\end{array}\right)
\end{equation}

\textit{A4. The Right Conjugate $HR$ Derivatives}:
By applying a rotation transformation on both sides of \eqref{eq:case1bdzdqlink}, we have
\begin{equation}\label{eq:case2bdzdqlink}
\begin{split}
&dq_a=\frac{1}{4}(dq^*+dq^{i*}+dq^{j*}+dq^{k*}),\quad dq_b=(dq_b)^i=\frac{1}{4}(dq^*+dq^{i*}-dq^{j*}-dq^{k*})i\\
&dq_c=\frac{1}{4}(dq^*-dq^{i*}+dq^{j*}-dq^{k*})j,\quad dq_d=(dq_d)^k=\frac{1}{4}(dq^*-dq^{i*}-dq^{j*}+dq^{k*})k
\end{split}
\end{equation}
By substituting \eqref{eq:case2bdzdqlink} into \eqref{eq:case2df}, the differential of $f$ can be written as
\begin{equation}\label{eq:case2bdfdq}
\begin{split}
df&=\frac{1}{4}(dq^*+dq^{i*}+dq^{j*}+dq^{k*})\frac{\partial f}{\partial q_a}+\frac{1}{4}(dq^*+dq^{i*}-dq^{j*}-dq^{k*})i\frac{\partial f}{\partial q_b}\\
&\quad+\frac{1}{4}(dq^*-dq^{i*}+dq^{j*}-dq^{k*})j\frac{\partial f}{\partial q_b}+\frac{1}{4}(dq^*-dq^{i*}-dq^{j*}+dq^{k*})k\frac{\partial f}{\partial q_b}\\
&=\frac{1}{4}dq^*\left(\frac{\partial f}{\partial q_a}+i\frac{\partial f}{\partial q_b}+j\frac{\partial f}{\partial q_c}+k\frac{\partial f}{\partial q_d}\right)+\frac{1}{4}dq^{i*}\left(\frac{\partial f}{\partial q_a}+i\frac{\partial f}{\partial q_b}-j\frac{\partial f}{\partial q_c}-k\frac{\partial f}{\partial q_d}\right)\\
&\quad+\frac{1}{4}dq^{j*}\left(\frac{\partial f}{\partial q_a}-i\frac{\partial f}{\partial q_b}+j\frac{\partial f}{\partial q_c}-k\frac{\partial f}{\partial q_d}\right)+\frac{1}{4}dq^{k*}\left(\frac{\partial f}{\partial q_a}-i\frac{\partial f}{\partial q_b}-j\frac{\partial f}{\partial q_c}+k\frac{\partial f}{\partial q_d}\right)\\
\end{split}
\end{equation}
Now, define the formal right derivatives $\frac{\partial_r f}{\partial q^*},\frac{\partial_r f}{\partial q^{i*}},\frac{\partial_r f}{\partial q^{j*}}$ and $\frac{\partial_r f}{\partial q^{k*}}$ so that
\begin{equation}\label{eq:case2bdfdqcmp}
df=dq^*\frac{\partial_r f}{\partial q^*}+dq^{i*}\frac{\partial_r f}{\partial q^{i*}}+dq^{j*}\frac{\partial_r f}{\partial q^{j*}}+dq^{k*}\frac{\partial_r f}{\partial q^{k*}}
\end{equation}
holds. Comparing \eqref{eq:case2bdfdqcmp} with \eqref{eq:case2bdfdq} and using Lemma \ref{lm:dpconjindp}, we can obtain the following right conjugate ${\rm HR}$ derivatives

\begin{equation}
\left(\begin{array}{l}
\frac{\partial_r f}{\partial q^*}\\
\frac{\partial_r f}{\partial q^{i*}}\\
\frac{\partial_r f}{\partial q^{j*}}\\
\frac{\partial_r f}{\partial q^{k*}}
\end{array}\right)=\frac{1}{4}
\left(\begin{array}{cccc}
1 & i  & j & k \\
1 & i & -j  & -k\\
1 & -i  & j & -k\\
1 & -i  & -j  & k\\
\end{array}\right)
\left(\begin{array}{l}
\frac{\partial f}{\partial q_a}\\
\frac{\partial f}{\partial q_b}\\
\frac{\partial f}{\partial q_c}\\
\frac{\partial f}{\partial q_d}
\end{array}\right)
\end{equation}

\subsection*{Appendix B: The Proof of The Product Rule.}
Within the GHR calculus, when a quaternion-valued function is postmultiplied by a real-valued function, then the novel product rule will degenerate into the traditional product rule. This is stated in the next lemma.
\begin{lemma}\label{lem:qvfprodrl}
If the functions $f:\mathbb{H}\rightarrow \mathbb{H}$ and $g:\mathbb{H}\rightarrow \mathbb{R}$ have the left GHR derivatives, then
their product $fg$ satisfies the traditional product rule
\begin{equation*}
\frac{\partial (fg)}{\partial q^{\mu}}=f\frac{\partial g}{\partial q^{\mu}}+\frac{\partial f}{\partial q^{\mu}}g,\quad \frac{\partial (fg)}{\partial q^{\mu*}}=f\frac{\partial g}{\partial q^{\mu*}}+\frac{\partial f}{\partial q^{\mu*}}g
\end{equation*}
where $\frac{\partial f}{\partial q^{\mu}}$ and $\frac{\partial f}{\partial q^{\mu*}}$ are the left GHR derivatives in Definition \ref{def:leftghr}.
\end{lemma}
\begin{proof}
Let $f=f_a+if_b+jf_c+kf_d$, where $f_a,f_b,f_c,f_d\in \mathbb{R}$, then
\begin{equation}
fg=f_ag+if_bg+jf_cg+kf_dg, \quad g\in \mathbb{R}
\end{equation}
Using the property $\frac{\partial (\eta f)}{\partial q^{\mu}}=\eta\frac{\partial f}{\partial q^{\mu}}$ in \eqref{pr:leftghrpr1}, we have
\begin{equation}
\begin{split}
&\frac{\partial (fg)}{\partial q^{\mu}}=\frac{\partial (f_ag)}{\partial q^{\mu}}+i\frac{\partial (f_bg)}{\partial q^{\mu}}+j\frac{\partial (f_cg)}{\partial q^{\mu}}+k\frac{\partial (f_dg)}{\partial q^{\mu}}\\
&=f_a\frac{\partial g}{\partial q^{\mu}}+if_b\frac{\partial g}{\partial q^{\mu}}+jf_c\frac{\partial g}{\partial q^{\mu}}+kf_d\frac{\partial g}{\partial q^{\mu}}+\frac{\partial f_a}{\partial q^{\mu}}g+i\frac{\partial f_b}{\partial q^{\mu}}g+j\frac{\partial f_c}{\partial q^{\mu}}g+k\frac{\partial f_d}{\partial q^{\mu}}g\\
&=f\frac{\partial g}{\partial q^{\mu}}+\left(\frac{\partial f_a}{\partial q^{\mu}}+i\frac{\partial f_b}{\partial q^{\mu}}+j\frac{\partial f_c}{\partial q^{\mu}}+k\frac{\partial f_d}{\partial q^{\mu}}\right)g\\
&=f\frac{\partial g}{\partial q^{\mu}}+\frac{\partial f}{\partial q^{\mu}}g
\end{split}
\end{equation}
The first part of the lemma is proved, and the second part can be shown in a similar way.
\end{proof}

\begin{proof*}[of Theorem~{\rm\ref{thm:nwpleftpdrl}}]
From \eqref{pr:prodijkmurl}, it is seen that $\{1,i^{\mu},j^{\mu},k^{\mu}\}$ is another orthogonal basis of $\mathbb{H}$. Then the quaternion-valued function $g$ can be expressed in the following way
\begin{equation}
 g=g_a+i^{\mu}g_b+j^{\mu}g_c+k^{\mu}g_d,\quad  g_a,g_b,g_c,g_d\in \mathbb{R}
\end{equation}
and $fg=fg_a+fi^{\mu}g_b+fj^{\mu}g_c+fk^{\mu}g_d$, $f\in\mathbb{H}$. Using the sum rule, it follows
\begin{equation}\label{eq:dfgdijkmu}
\frac{\partial (fg)}{\partial q^{\mu}}=\frac{\partial (fg_a)}{\partial q^{\mu}}+\frac{\partial (fi^{\mu}g_b)}{\partial q^{\mu}}+\frac{\partial (fi^{\mu}g_c)}{\partial q^{\mu}}+\frac{\partial (fi^{\mu}g_d)}{\partial q^{\mu}}
\end{equation}
By applying Lemma \ref{lem:qvfprodrl} to the right side of \eqref{eq:dfgdijkmu}, it can be shown that
\begin{equation}
\begin{split}
&\frac{\partial (fg)}{\partial q^{\mu}}=f\frac{\partial g_a}{\partial q^{\mu}}+fi^{\mu}\frac{\partial g_b}{\partial q^{\mu}}+fj^{\mu}\frac{\partial g_c}{\partial q^{\mu}}+fk^{\mu}\frac{\partial g_d}{\partial q^{\mu}}+\frac{\partial f}{\partial q^{\mu}}g_a+\frac{\partial (f{i^{\mu}})}{\partial q^{\mu}}g_b+\frac{\partial (f{j^{\mu}})}{\partial q^{\mu}}g_c+\frac{\partial (f{k^{\mu}})}{\partial q^{\mu}}g_d\\
&=f\left(\frac{\partial g_a}{\partial q^{\mu}}+i^{\mu}\frac{\partial g_b}{\partial q^{\mu}}+j^{\mu}\frac{\partial g_c}{\partial q^{\mu}}+k^{\mu}\frac{\partial g_d}{\partial q^{\mu}}\right)+\frac{\partial f}{\partial q^{\mu}}g_a+\frac{\partial (f{i^{\mu}})}{\partial q^{\mu}}g_b+\frac{\partial (f{j^{\mu}})}{\partial q^{\mu}}g_c+\frac{\partial (f{k^{\mu}})}{\partial q^{\mu}}g_d\\
&=f\frac{\partial g}{\partial q^{\mu}}
+\frac{\partial f}{\partial q^{\mu}}g_a+\frac{\partial (f{i^{\mu}})}{\partial q^{\mu}}g_b+\frac{\partial (f{j^{\mu}})}{\partial q^{\mu}}g_c+\frac{\partial (f{k^{\mu}})}{\partial q^{\mu}}g_d
\end{split}
\end{equation}
Next, by using the result $\frac{\partial (f \eta )}{\partial q^{\mu}}=\frac{\partial f}{\partial q^{\eta\mu}}\eta$ in \eqref{pr:leftghrpr1}, we have
\begin{equation}\label{eq:prfdfgimujmu}
\begin{split}
&\frac{\partial (fg)}{\partial q^{\mu}}=f\frac{\partial g}{\partial q^{\mu}}
+\frac{\partial f}{\partial q^{\mu}}g_a+\frac{\partial f}{\partial q^{\mu i}}i^{\mu}g_b+\frac{\partial f}{\partial q^{\mu j}}j^{\mu}g_c+\frac{\partial f}{\partial q^{\mu k}}k^{\mu}g_d\\
&=f\frac{\partial g}{\partial q^{\mu}}
+\frac{1}{4}\left(\frac{\partial f}{\partial q_a}-\frac{\partial f}{\partial q_b}i^{\mu}-\frac{\partial f}{\partial q_c}j^{\mu}-\frac{\partial f}{\partial q_d}k^{\mu}\right)g_a+\frac{1}{4}\left(\frac{\partial f}{\partial q_a}-\frac{\partial f}{\partial q_b}i^{\mu}+\frac{\partial f}{\partial q_c}j^{\mu}+\frac{\partial f}{\partial q_d}k^{\mu}\right)i^{\mu}g_b\\
&\quad +\frac{1}{4}\left(\frac{\partial f}{\partial q_a}+\frac{\partial f}{\partial q_b}i^{\mu}-\frac{\partial f}{\partial q_c}j^{\mu}+\frac{\partial f}{\partial q_d}k^{\mu}\right)j^{\mu}g_c+\frac{1}{4}\left(\frac{\partial f}{\partial q_a}+\frac{\partial f}{\partial q_b}i^{\mu}+\frac{\partial f}{\partial q_c}j^{\mu}-\frac{\partial f}{\partial q_d}k^{\mu}\right)k^{\mu}g_d
\end{split}
\end{equation}
where the Definition \ref{def:leftghr} and \eqref{pr:def1qmunu} were used in the last equality above. Grouping
$\frac{\partial f}{\partial q_a}$, $\frac{\partial f}{\partial q_b}$, $\frac{\partial f}{\partial q_c}$ and $\frac{\partial f}{\partial q_d}$ in \eqref{eq:prfdfgimujmu} yields
\begin{equation}
\begin{split}
\frac{\partial (fg)}{\partial q^{\mu}}&=f\frac{\partial g}{\partial q^{\mu}}+\frac{1}{4}\frac{\partial f}{\partial q_a}\left(g_a+i^{\mu}g_b+j^{\mu}g_c+k^{\mu}g_d\right)-\frac{1}{4}\frac{\partial f}{\partial q_b}\left(g_a+i^{\mu}g_b+j^{\mu}g_c+k^{\mu}g_d\right)i^{\mu}\\
&\quad-\frac{1}{4}\frac{\partial f}{\partial q_c}\left(g_a+i^{\mu}g_b+j^{\mu}g_c+k^{\mu}g_d\right)j^{\mu}-\frac{1}{4}\frac{\partial f}{\partial q_d}\left(g_a+i^{\mu}g_b+j^{\mu}g_c+k^{\mu}g_d\right)k^{\mu}\\
&=f\frac{\partial g}{\partial q^{\mu}}
+\frac{1}{4}\left(\frac{\partial f}{\partial q_a}g-\frac{\partial f}{\partial q_b}gi^{\mu}-\frac{\partial f}{\partial q_c}gj^{\mu}-\frac{\partial f}{\partial q_d}gk^{\mu}\right)\\
&=f\frac{\partial g}{\partial q^{\mu}}
+\frac{1}{4}\left(\frac{\partial f}{\partial q_a}-\frac{\partial f}{\partial q_b}gi^{\mu}g^{-1}-\frac{\partial f}{\partial q_c}gj^{\mu}g^{-1}-\frac{\partial f}{\partial q_d}gk^{\mu}g^{-1}\right)g\\
&=f\frac{\partial g}{\partial q^{\mu}}
+\frac{1}{4}\left(\frac{\partial f}{\partial q_a}-\frac{\partial f}{\partial q_b}i^{g\mu}-\frac{\partial f}{\partial q_c}j^{g\mu}-\frac{\partial f}{\partial q_d}k^{g\mu}\right)g\\
&=f\frac{\partial g}{\partial q^{\mu}}+\frac{\partial f}{\partial q^{g\mu}}g
\end{split}
\end{equation}
where \eqref{pr:def1qmunu} was used in the second to last equality above, and Definition \ref{def:leftghr} was used in the last equality.
Hence, the first part of the theorem follows, and the second part can be proved in an analogous manner.
\end{proof*}

\subsection*{Appendix C: The Proof of The Chain Rule.}
To prove the chain rule, the following lemma shall be used.
\begin{lemma}\label{lem:qvfchainrl}
Let $q=q_a+iq_b+jq_c+kq_d$, where $q_a,q_b,q_c,q_d\in\mathbb{R}$, then the partial derivatives of the quaternion-valued composite function $f(g(q))$ satisfy the following chain rule
\begin{equation*}
\begin{split}
&\frac{\partial f(g(q))}{\partial \xi}=\frac{\partial f}{\partial g^{\nu}}\frac{\partial g^{\nu}}{\partial \xi}+\frac{\partial f}{\partial g^{ \nu i}}\frac{\partial g^{\nu i}}{\partial \xi}+\frac{\partial f}{\partial g^{\nu j}}\frac{\partial g^{\nu j}}{\partial \xi}+\frac{\partial f}{\partial g^{\nu k}}\frac{\partial g^{\nu k}}{\partial \xi}\\
&\frac{\partial f(g(q))}{\partial \xi}=\frac{\partial f}{\partial g^{\nu*}}\frac{\partial g^{\nu*}}{\partial \xi}+\frac{\partial f}{\partial g^{ \nu i*}}\frac{\partial g^{\nu i*}}{\partial \xi}+\frac{\partial f}{\partial g^{\nu j*}}\frac{\partial g^{\nu j*}}{\partial \xi}+\frac{\partial f}{\partial g^{\nu k*}}\frac{\partial g^{\nu k*}}{\partial \xi}
\end{split}
\end{equation*}
where $\xi\in\{q_a,q_b,q_c,q_d\}$ and $\nu\in\mathbb{H},\nu \neq 0$.
\end{lemma}
\begin{proof}
Let $g(q)=g_a+ig_b+jg_c+kg_d$, where $g_a,g_b,g_c,g_d\in \mathbb{R}$. Then, the function $f(g(q))$ can be seen as a function of the four real-valued variables $g_a,g_b,g_c$ and $g_d$, and the partial derivative of $f(g(q))$ can be expressed as
\begin{equation}\label{eq:profdfdxi}
\frac{\partial f}{\partial \xi}=\frac{\partial f}{\partial g_a}\frac{\partial g_a}{\partial \xi}+\frac{\partial f}{\partial g_b}\frac{\partial g_b}{\partial \xi}+\frac{\partial f}{\partial g_c}\frac{\partial g_c}{\partial \xi}+\frac{\partial f}{\partial g_d}\frac{\partial g_d}{\partial \xi}
\end{equation}
By Definition \ref{def:leftghr}, the partial derivatives $\frac{\partial f}{\partial g_a}$, $\frac{\partial f}{\partial g_b}$, $\frac{\partial f}{\partial g_c}$ and $\frac{\partial f}{\partial g_d}$ are given by
\begin{equation}\label{eq:profdfqabc}
\begin{split}
&\frac{\partial f}{\partial g_a}=\frac{\partial f}{\partial g^{\nu}}+\frac{\partial f}{\partial g^{\nu i}}+\frac{\partial f}{\partial g^{\nu j}}+\frac{\partial f}{\partial g^{\nu k}},\quad \frac{\partial f}{\partial g_b}=\left(\frac{\partial f}{\partial g^{\nu}}+\frac{\partial f}{\partial g^{\nu i}}-\frac{\partial f}{\partial g^{\nu j}}-\frac{\partial f}{\partial g^{\nu k}}\right)i^{\nu}\\
&\frac{\partial f}{\partial g_c}=\left(\frac{\partial f}{\partial g^{\nu}}-\frac{\partial f}{\partial g^{\nu i}}+\frac{\partial f}{\partial g^{\nu j}}-\frac{\partial f}{\partial g^{\nu k}}\right)j^{\nu},\quad \frac{\partial f}{\partial g_d}=\left(\frac{\partial f}{\partial g^{\nu}}-\frac{\partial f}{\partial g^{\nu i}}-\frac{\partial f}{\partial g^{\nu j}}+\frac{\partial f}{\partial g^{\nu k}}\right)k^{\nu}\\
\end{split}
\end{equation}
By applying the quaternion rotation transform $(\cdot)^{\nu}$ to both sides of \eqref{eq:z1qlink} and replacing $q$ with $g$, the real-valued components $g_a,g_b,g_c$ and $g_d$ can be expressed as
\begin{equation}
\begin{split}
&g_a=\frac{1}{4}\left(g^{\nu}+g^{\nu i}+g^{\nu j}+g^{\nu k}\right),\quad g_b=-\frac{i^{\nu}}{4}\left(g^{\nu}+g^{\nu i}-g^{\nu j}-g^{\nu k}\right)\\
&g_c=-\frac{j^{\nu}}{4}\left(g^{\nu}-g^{\nu i}+g^{\nu j}-g^{\nu k}\right),\quad g_d=-\frac{k^{\nu}}{4}\left(g^{\nu}-g^{\nu i}-g^{\nu j}+g^{\nu k}\right)
\end{split}
\end{equation}
Taking the partial derivative  of both sides of the above equations yields
\begin{equation}\label{eq:prfdqabcxi}
\begin{split}
&\frac{\partial g_a}{\partial \xi}=\frac{1}{4}\left(\frac{\partial g^{\nu}}{\partial \xi}+\frac{\partial g^{\nu i}}{\partial \xi}+\frac{\partial g^{\nu j}}{\partial \xi}+\frac{\partial g^{\nu k}}{\partial \xi}\right),\quad \frac{\partial g_b}{\partial \xi}=-\frac{i^{\nu}}{4}\left(\frac{\partial g^{\nu}}{\partial \xi}+\frac{\partial g^{\nu i}}{\partial \xi}-\frac{\partial g^{\nu j}}{\partial \xi}-\frac{\partial g^{\nu k}}{\partial \xi}\right)\\
&\frac{\partial g_c}{\partial \xi}=-\frac{j^{\nu}}{4}\left(\frac{\partial g^{\nu}}{\partial \xi}-\frac{\partial g^{\nu i}}{\partial \xi}+\frac{\partial g^{\nu j}}{\partial \xi}-\frac{\partial g^{\nu k}}{\partial \xi}\right),\quad \frac{\partial g_d}{\partial \xi}=-\frac{k^{\nu}}{4}\left(\frac{\partial g^{\nu}}{\partial \xi}-\frac{\partial g^{\nu i}}{\partial \xi}-\frac{\partial g^{\nu j}}{\partial \xi}+\frac{\partial g^{\nu k}}{\partial \xi}\right)
\end{split}
\end{equation}
By substituting the results from \eqref{eq:profdfqabc} and \eqref{eq:prfdqabcxi} into \eqref{eq:profdfdxi},
and using the distributive law and merging of similar items, a simple result follows
\begin{equation}
\frac{\partial f(g(q))}{\partial \xi}=\frac{\partial f}{\partial g^{\nu}}\frac{\partial g^{\nu}}{\partial \xi}+\frac{\partial f}{\partial g^{ \nu i}}\frac{\partial g^{\nu i}}{\partial \xi}+\frac{\partial f}{\partial g^{\nu j}}\frac{\partial g^{\nu j}}{\partial \xi}+\frac{\partial f}{\partial g^{\nu k}}\frac{\partial g^{\nu k}}{\partial \xi}
\end{equation}
Hence, the first equality of the lemma follows, the second equality can be derived in a similar fashion.
\end{proof}

\begin{proof*}[of Theorem~{\rm\ref{thm:nwpleftchainrl}}]
By using Definition \ref{def:leftghr}, the left HR derivative of the product $fg$ can be expressed as
\begin{equation}\label{eq:dfgqchain}
\begin{split}
&\frac{\partial f(g(q))}{\partial q^{\mu}}=\frac{1}{4}\left(\frac{\partial f(g(q))}{\partial q_a}-\frac{\partial f(g(q))}{\partial q_b}i^{\mu}-\frac{\partial f(g(q))}{\partial q_c}j^{\mu}-\frac{\partial f(g(q))}{\partial q_d}k^{\mu}\right)
\end{split}
\end{equation}
By substituting the first equality of Lemma \ref{lem:qvfchainrl} into \eqref{eq:dfgqchain}, it follows that
\begin{equation}\label{eq:profdfgqqmu}
\begin{split}
\frac{\partial f(g(q))}{\partial q^{\mu}}=&\frac{1}{4}\left(\frac{\partial f}{\partial g^{\nu}}\frac{\partial g^{\nu}}{\partial q_a}+\frac{\partial f}{\partial g^{\nu i}}\frac{\partial g^{\nu i}}{\partial q_a}+\frac{\partial f}{\partial g^{\nu j}}\frac{\partial g^{\nu j}}{\partial q_a}+\frac{\partial f}{\partial g^{\nu k}}\frac{\partial g^{\nu k}}{\partial q_a}\right)\\
&-\frac{1}{4}\left(\frac{\partial f}{\partial g^{\nu}}\frac{\partial g^{\nu}}{\partial q_b}+\frac{\partial f}{\partial g^{\nu i}}\frac{\partial g^{\nu i}}{\partial q_b}+\frac{\partial f}{\partial g^{\nu j}}\frac{\partial g^{\nu j}}{\partial q_b}+\frac{\partial f}{\partial g^{\nu k}}\frac{\partial g^{\nu k}}{\partial q_b}\right)i^{\mu}\\
&-\frac{1}{4}\left(\frac{\partial f}{\partial g^{\nu}}\frac{\partial g^{\nu}}{\partial q_c}+\frac{\partial f}{\partial g^{\nu i}}\frac{\partial g^{\nu i}}{\partial q_c}+\frac{\partial f}{\partial g^{\nu j}}\frac{\partial g^{\nu j}}{\partial q_c}+\frac{\partial f}{\partial g^{\nu k}}\frac{\partial g^{\nu k}}{\partial q_c}\right)j^{\mu}\\
&-\frac{1}{4}\left(\frac{\partial f}{\partial g^{\nu}}\frac{\partial g^{\nu}}{\partial q_d}+\frac{\partial f}{\partial g^{\nu i}}\frac{\partial g^{\nu i}}{\partial q_d}+\frac{\partial f}{\partial g^{\nu j}}\frac{\partial g^{\nu j}}{\partial q_d}+\frac{\partial f}{\partial g^{\nu k}}\frac{\partial g^{\nu k}}{\partial q_d}\right)k^{\mu}
\end{split}
\end{equation}
Grouping around the terms $\frac{\partial f}{\partial g^{\nu}}$, $\frac{\partial f}{\partial g^{\nu i}}$, $\frac{\partial f}{\partial g^{\nu j}}$ and $\frac{\partial f}{\partial g^{\nu k}}$ in \eqref{eq:profdfgqqmu}, we have
\begin{equation}
\begin{split}
\frac{\partial f(g(q))}{\partial q^{\mu}}=&\frac{\partial f}{\partial g^{\nu}}\frac{1}{4}\left(\frac{\partial g^{\nu}}{\partial q_a}-\frac{\partial g^{\nu}}{\partial q_b}i^{\mu}-\frac{\partial g^{\nu}}{\partial q_c}j^{\mu}-\frac{\partial g^{\nu}}{\partial q_d}k^{\mu}\right)\\
&+\frac{\partial f}{\partial g^{\nu i}}\frac{1}{4}\left(\frac{\partial g^{\nu i}}{\partial q_a}-\frac{\partial g^{\nu i}}{\partial q_b}i^{\mu}-\frac{\partial g^{\nu i}}{\partial q_c}j^{\mu}-\frac{\partial g^{\nu i}}{\partial q_d}k^{\mu}\right)\\
&+\frac{\partial f}{\partial g^{\nu j}}\frac{1}{4}\left(\frac{\partial g^{\nu j}}{\partial q_a}-\frac{\partial g^{\nu j}}{\partial q_b}i^{\mu}-\frac{\partial g^{\nu j}}{\partial q_c}j^{\mu}-\frac{\partial g^{\nu j}}{\partial q_d}k^{\mu}\right)\\
&+\frac{\partial f}{\partial g^{\nu k}}\frac{1}{4}\left(\frac{\partial g^{\nu k}}{\partial q_a}-\frac{\partial g^{\nu k}}{\partial q_b}i^{\mu}-\frac{\partial g^{\nu k}}{\partial q_c}j^{\mu}-\frac{\partial g^{\nu k}}{\partial q_d}k^{\mu}\right)\\
=&\frac{\partial f}{\partial g^{\nu}}\frac{\partial g^{\nu}}{\partial q^{\mu}}+\frac{\partial f}{\partial g^{\nu i}}\frac{\partial g^{\nu i}}{\partial q^{\mu}}+\frac{\partial f}{\partial g^{\nu j}}\frac{\partial g^{\nu j}}{\partial q^{\mu}}+\frac{\partial f}{\partial g^{\nu k}}\frac{\partial g^{\nu k}}{\partial q^{\mu}}
\end{split}
\end{equation}
Hence, the first equality of the theorem follows, the other equalities can be derived in a similar fashion.
\end{proof*}

\subsection*{Appendix D: Fundamental Results Based on the GHR Derivatives.}
Several of the most important results of left GHR derivatives are summarized in Table \ref{tb:fmlleftderiv} and are easy for the reader to locate, assuming $\nu$, $\omega$ and $\lambda$ to be constant quaternions, $q$ to be a quaternion-valued variable, and $\mu$ to be any quaternion constants or functions. To show how to use Table \ref{tb:fmlleftderiv}, the following examples are presented.

\begin{table}[htbp]
\caption{Important results of the GHR derivatives}\label{tb:fmlleftderiv}
\renewcommand{\arraystretch}{1.2} 
\arrayrulewidth=1.0pt \tabcolsep=2.5pt
\begin{tabular}{cccc}
 \hline
$f(q)$    & $\frac{\partial f}{\partial q^{\mu}}{\mu}$ & $\frac{\partial f}{\partial q^{\mu*}}{\mu}$ & Note  \\[0pt]
 \hline
$q$  & $\mathfrak{R}(\mu)$  &$-\frac{1}{2}\mu^*$& $--$ \\[5pt]
$\omega q$  & $\omega\mathfrak{R}(\mu)$  &$-\frac{1}{2}\omega\mu^*$& $\forall\omega\in\mathbb{H}$ \\[5pt]
$q\nu$  & $\mathfrak{R}(\nu\mu)$  &$-\frac{1}{2}(\nu\mu)^*$&$\forall\nu\in\mathbb{H}$ \\[5pt]
$\omega q\nu+\lambda$  & $\omega\mathfrak{R}(\nu\mu)$  &$-\frac{1}{2}\omega(\nu\mu)^*$&$\forall\omega,\nu,\lambda\in\mathbb{H}$ \\[5pt]
 \hdashline[1pt/1pt]
$q^*$  & $-\frac{1}{2}\mu^*$  &$\mathfrak{R}(\mu)$ &$--$ \\[5pt]
$\omega q^*$  & $-\frac{1}{2}\omega\mu^*$  &$\omega\mathfrak{R}(\mu)$ &$\forall\omega\in\mathbb{H}$ \\[5pt]
$q^*\nu$  & $-\frac{1}{2}(\nu\mu)^*$  &$\mathfrak{R}(\nu\mu)$& $\forall\nu\in\mathbb{H}$ \\[5pt]
$\omega q^*\nu+\lambda$  & $-\frac{1}{2}\omega(\nu\mu)^*$  &$\omega\mathfrak{R}(\nu\mu)$& $\forall\omega,\nu,\lambda\in\mathbb{H}$ \\[5pt]
 \hdashline[1pt/1pt]
 $\alpha=|V_q|$  & $-\frac{1}{4}\mu \hat{q}$  &$\frac{1}{4}\mu \hat{q}$ &$V_q=\mathfrak{I}(q)$  \\[5pt]
 $\hat{q}=\frac{V_q}{|V_q|}$  & $\frac{2\mu^*+\mu-\mu^{\hat{q}}}{4\alpha}$  &$-\frac{2\mu^*+\mu-\mu^{\hat{q}}}{4\alpha}$ &$\alpha=|V_q|$ \\[5pt]
 $\arctan\left(\frac{|V_q|}{S_q}\right)$  & $-\frac{\mu\hat{q}q^*}{4|q|^2}$  &$\frac{\mu\hat{q}q}{4|q|^2}$ &$S_q=\mathfrak{R}(q)$ \\[5pt]
 \hdashline[1pt/1pt]
 $q^{-1}$  & $-q^{-1}\mathfrak{R}(q^{-1}\mu)$  &$\frac{1}{2}q^{-1}\mu^*(q^*)^{-1}$ &$--$ \\[5pt]
$(q^*)^{-1}$  & $\frac{1}{2}(q^*)^{-1}\mu^*q^{-1}$  &$-(q^*)^{-1}\mathfrak{R}((q^*)^{-1}\mu)$ &$--$ \\[5pt]
$(\omega q \nu+\lambda)^{-1}$  & $-f\omega\mathfrak{R}(\nu f\mu)$  &$\frac{1}{2}f\omega(\nu f\mu)^*$ &$\forall\omega,\nu,\lambda\in\mathbb{H}$ \\[5pt]
$(\omega q^*\nu+\lambda)^{-1}$  & $\frac{1}{2}f\omega(\nu f\mu)^*$  &$-f\omega\mathfrak{R}(\nu f\mu)$ & $\forall\omega,\nu,\lambda\in\mathbb{H}$ \\[5pt]
 \hdashline[1pt/1pt]
$q^2$  & $q\mathfrak{R}(\mu)+\mathfrak{R}( q\mu)$  &$-\frac{1}{2}q\mu^*-\frac{1}{2}( q\mu)^*$& $--$  \\[5pt]
$(q^*)^2$  & $-\frac{1}{2}q^*\mu^*-\frac{1}{2}(q^*\mu)^*$  &$q^*\mathfrak{R}(\mu)+\mathfrak{R}(q^*\mu)$ & $--$\\[5pt]
$(\omega q\nu+\lambda)^2$  & $g\omega\mathfrak{R}(\nu\mu)+\omega\mathfrak{R}(\nu g\mu)$  &$-\frac{1}{2}g\omega(\nu\mu)^*-\frac{1}{2}\omega(\nu g\mu)^*$& $g=\omega q\nu+\lambda$ \\[5pt]
$(\omega q^*\nu+\lambda)^2$  & $-\frac{1}{2}g\omega(\nu\mu)^*-\frac{1}{2}\omega(\nu g\mu)^*$  &$g\omega\mathfrak{R}(\nu\mu)+\omega\mathfrak{R}(\nu g\mu)$ & $g=\omega q^*\nu+\lambda$\\[5pt]
 \hdashline[1pt/1pt]
$\mathfrak{R}(q)$  & $\frac{1}{4}\mu$  &$\frac{1}{4}\mu$ & $--$\\[5pt]
$\mathfrak{R}(\omega q\nu+\lambda)$  & $\frac{1}{4}\mu\nu\omega$  &$\frac{1}{4}\mu\omega^*\nu^*$ & $\forall\omega,\nu,\lambda\in\mathbb{H}$\\[5pt]
$\mathfrak{R}(\omega q^*\nu+\lambda)$  & $\frac{1}{4}\mu\omega^*\nu^*$ & $\frac{1}{4}\mu\nu\omega$ & $\forall\omega,\nu,\lambda\in\mathbb{H}$\\[5pt]
 \hdashline[1pt/1pt]
$\frac{q}{|q|}$  & $\frac{1}{|q|}\mathfrak{R}(\mu)-\frac{1}{4|q|^3}q\mu q^*$  &$-\frac{1}{2|q|}\mu^*-\frac{1}{4|q|^3}q\mu q$ & $--$\\[5pt]
$\frac{q^*}{|q|}$  & $-\frac{1}{2|q|}\mu^*-\frac{1}{4|q|^3}q^*\mu q^*$  &$\frac{1}{|q|}\mathfrak{R}(\mu)-\frac{1}{4|q|^3}q^*\mu q$ & $--$\\[5pt]
$\frac{\omega q\nu+\lambda}{|\omega q\nu+\lambda|}$  & $\frac{\omega}{2|g|}\mathfrak{R}(\nu\mu)+\frac{g}{4|g|^3}\nu^*(\omega^*g\mu)^*$  &$-\frac{\omega}{4|g|}(\nu\mu)^*-\frac{g}{2|g|^3}\nu^*\mathfrak{R}(\omega^*g\mu)$& $g=\omega q\nu+\lambda$ \\[5pt]
$\frac{\omega q^*\nu+\lambda}{|\omega q^*\nu+\lambda|}$  & $-\frac{\omega}{2|g|}(\nu\mu)^*-\frac{f}{|g|}\frac{\partial |g|}{\partial q^{\mu}}\mu$  &$\frac{\omega}{|g|}\mathfrak{R}(\nu\mu)-\frac{f}{|g|}\frac{\partial |g|}{\partial q^{\mu*}}\mu$ & $g=\omega q^*\nu+\lambda$ \\[5pt]
 \hdashline[1pt/1pt]
$|q|$  & $\frac{1}{4|q|}\mu q^*$  &$\frac{1}{4|q|}\mu q$ & $--$\\[5pt]
$|q|^2$  & $\frac{1}{2}\mu q^*$  &$\frac{1}{2}\mu q$ & $--$\\[5pt]
$|\omega q\nu+\lambda|$  & $\frac{g^*}{2|g|}\omega\mathfrak{R}(\nu\mu)-\frac{1}{4|g|}\nu^*(\omega^*g\mu)^*$  &$-\frac{g^*}{4|g|}\omega(\nu\mu)^*+\frac{1}{2|g|}\nu^*\mathfrak{R}(\omega^*g\mu)$& $g=\omega q\nu+\lambda$ \\[5pt]
$|\omega q^*\nu+\lambda|$  & $\frac{g}{2|g|}\nu^*\mathfrak{R}(\omega^*\mu)-\frac{1}{4|g|}\omega(\nu g^*\mu)^*$  &$-\frac{g}{4|g|}\nu^*(\omega^*\mu)^*+\frac{1}{2|g|}\omega\mathfrak{R}(\nu g^*\mu)$& $g=\omega q^*\nu+\lambda$ \\[5pt]
$|\omega q\nu+\lambda|^2$  & $g^*\omega\mathfrak{R}(\nu\mu)-\frac{1}{2}\nu^*(\omega^*g\mu)^*$  &$-\frac{1}{2}g^*\omega(\nu\mu)^*+\nu^*\mathfrak{R}(\omega^*g\mu)$& $g=\omega q\nu+\lambda$ \\[5pt]
$|\omega q^*\nu+\lambda|^2$  & $g\nu^*\mathfrak{R}(\omega^*\mu)-\frac{1}{2}\omega(\nu g^*\mu)^*$  &$-\frac{1}{2}g\nu^*(\omega^*\mu)^*+\omega\mathfrak{R}(\nu g^*\mu)$& $g=\omega q^*\nu+\lambda$ \\[5pt]
\hline
\end{tabular}\\[6pt]
\end{table}

\begin{example}
Find the GHR derivative of the function $f:\mathbb{H}\rightarrow \mathbb{H}$ given by
\begin{equation}
f(q)=q=q_a+iq_b+jq_c+kq_d
\end{equation}
\end{example}
\textbf{Solution}: By using Definition \ref{def:leftghr}, it follows that
\begin{equation}\label{eq:dqdqmu}
\begin{split}
\frac{\partial q}{\partial q^{\mu}}\mu&=\frac{1}{4}\left(\frac{\partial q}{\partial q_a}
-\frac{\partial q}{\partial q_b}i^{\mu}-\frac{\partial q}{\partial q_c}j^{\mu}-\frac{\partial q}{\partial q_d}k^{\mu}\right){\mu}=\frac{1}{4}\left(1-ii^{\mu}-jj^{\mu}-kk^{\mu}\right){\mu}\\
&=\frac{1}{4}\left(\mu-i\mu i-j\mu j-k\mu k\right)=\frac{1}{4}\left(\mu+\mu^i+\mu^j+\mu^k\right)=\mathfrak{R}(\mu)
\end{split}
\end{equation}
In a similar manner, it can be shown that
\begin{equation}\label{eq:dqdqmuconj}
\begin{split}
\frac{\partial q}{\partial q^{\mu*}}\mu&=\frac{1}{4}\left(\frac{\partial q}{\partial q_a}
+\frac{\partial q}{\partial q_b}i^{\mu}+\frac{\partial q}{\partial q_c}j^{\mu}+\frac{\partial q}{\partial q_d}k^{\mu}\right){\mu}=\frac{1}{4}\left(1+ii^{\mu}+jj^{\mu}+kk^{\mu}\right){\mu}\\
&=\frac{1}{4}\left(\mu+i\mu i+j\mu j+k\mu k\right)=\frac{1}{4}\left(\mu-\mu^i-\mu^j-\mu^k\right)=-\frac{1}{2}\mu^*
\end{split}
\end{equation}

\begin{example}
Find the GHR derivative of the function $f:\mathbb{H}\rightarrow \mathbb{H}$ given by
\begin{equation}
f(q)=\omega q\nu+\lambda
\end{equation}
where $\nu,\omega,\lambda\in\mathbb{H}$ are constants.
\end{example}
\textbf{Solution}: Using \eqref{pr:leftghrpr1}, we have
\begin{equation}\label{eq:dwqvydqu}
\frac{\partial (\omega q\nu+\lambda)}{\partial q^{\mu}}{\mu}=\frac{\partial (\omega q\nu)}{\partial q^{\mu}}{\mu}+\frac{\partial \lambda}{\partial q^{\mu}}{\mu}=\omega\frac{\partial ( q\nu)}{\partial q^{\mu}}{\mu}=\omega\frac{\partial q}{\partial q^{\nu\mu}}\nu\mu=\omega\mathfrak{R}(\nu\mu)
\end{equation}
where \eqref{eq:dqdqmu} was used in the last equality. In a similar manner, it follows that
\begin{equation}\label{eq:dwqvdqnp2}
\frac{\partial (\omega q\nu+\lambda)}{\partial q^{\mu*}}{\mu}=\frac{\partial (\omega q\nu)}{\partial q^{\mu*}}{\mu}=\omega\frac{\partial ( q\nu)}{\partial q^{\mu*}}{\mu}=\omega\frac{\partial q}{\partial q^{\nu\mu*}}\nu\mu=-\frac{1}{2}\omega(\nu\mu)*
\end{equation}
where \eqref{eq:dqdqmuconj} was used in the last equality above.

\begin{example}
Find the GHR derivative of the function $f:\mathbb{H}\rightarrow \mathbb{H}$ given by
\begin{equation}
f(q)=(\omega q\nu+\lambda)^2
\end{equation}
where $\nu,\omega,\lambda\in\mathbb{H}$ are constants.
\end{example}
\textbf{Solution}: Set $g(q)=\omega q\nu+\lambda$, then $f(q)=g(q)g(q)$. By using the product rule in Theorem \ref{thm:nwpleftpdrl}, we have
\begin{equation}\label{eq:dwqvydqu3}
\frac{\partial f}{\partial q^{\mu}}=\frac{\partial (gg)}{\partial q^{\mu}}=g\frac{\partial g}{\partial q^{\mu}}+\frac{\partial g}{\partial q^{g\mu}}g=g\frac{\partial (\omega q\nu+\lambda)}{\partial q^{\mu}}+\frac{\partial (\omega q\nu+\lambda)}{\partial q^{g\mu}}g
\end{equation}
From \eqref{eq:dwqvydqu}, it follows that
\begin{equation}
\frac{\partial (\omega q\nu+\lambda)}{\partial q^{\mu}}\mu=\omega\mathfrak{R}(\nu\mu),\quad
\frac{\partial (\omega q\nu+\lambda)}{\partial q^{f\mu}}f\mu=\omega\mathfrak{R}(\nu f\mu)
\end{equation}
By substituting the above expressions into \eqref{eq:dwqvydqu3}, it is found that
\begin{equation}
\frac{\partial f}{\partial q^{\mu}}\mu=g\omega\mathfrak{R}(\nu\mu)+\omega\mathfrak{R}(\nu g\mu)
\end{equation}
The next result can be derived in an analogous manner
\begin{equation}
\begin{split}
\frac{\partial f}{\partial q^{\mu*}}&=\frac{\partial (gg)}{\partial q^{\mu*}}=g\frac{\partial g}{\partial q^{\mu*}}+\frac{\partial g}{\partial q^{g\mu*}}g=g\frac{\partial (\omega q\nu+\lambda)}{\partial q^{\mu*}}+\frac{\partial (\omega q\nu+\lambda)}{\partial q^{g\mu*}}g\\
&=g\frac{\partial (\omega q\nu+\lambda)}{\partial q^{\mu*}}\mu\mu^{-1}+\frac{\partial (\omega q\nu+\lambda)}{\partial q^{g\mu*}}g\mu\mu^{-1}=-\frac{1}{2}g\omega(\nu\mu)^*\mu^{-1}-\frac{1}{2}\omega(\nu g\mu)^*\mu^{-1}
\end{split}
\end{equation}
where \eqref{eq:dwqvdqnp2} was used in the last quality above. This is equivalent to
\begin{equation}
\frac{\partial f}{\partial q^{\mu*}}\mu=-\frac{1}{2}g\omega(\nu\mu)^*-\frac{1}{2}\omega(\nu g\mu)^*
\end{equation}

\begin{example}
Find the GHR derivative of the function $f:\mathbb{H}\rightarrow \mathbb{H}$ given by
\begin{equation}
f(q)=|\omega q\nu+\lambda|^2
\end{equation}
where $\nu,\omega,\lambda\in\mathbb{H}$ are constants.
\end{example}
\textbf{Solution}:  Set $g(q)=\omega q\nu+\lambda$, then $f(q)=g^*g$. By using the product rule in Theorem \ref{thm:nwpleftpdrl},
 the results $\frac{\partial (\omega q\nu+\lambda)}{\partial q^{\mu}}$ and $\frac{\partial (\omega q^*\nu+\lambda)}{\partial q^{\mu}}$ in Table \ref{tb:fmlleftderiv}, it follows
\begin{equation}
\begin{split}
\frac{\partial f}{\partial q^{\mu}}\mu&=\frac{\partial (g^*g)}{\partial q^{\mu}}\mu=g^*\frac{\partial g}{\partial q^{\mu}}\mu+\frac{\partial g^*}{\partial q^{g\mu}}g\mu=g^*\frac{\partial (\omega q\nu+\lambda)}{\partial q^{\mu}}\mu+\frac{\partial (\nu^*q^*\omega^*+\lambda^*)}{\partial q^{g\mu}}g\mu\\
&=g^*\omega\mathfrak{R}(\nu\mu)-\frac{1}{2}\nu^*(\omega^*g\mu)^*
\end{split}
\end{equation}
In a similar manner,
\begin{equation}
\begin{split}
\frac{\partial f}{\partial q^{\mu*}}\mu&=\frac{\partial (g^*g)}{\partial q^{\mu*}}\mu=g^*\frac{\partial g}{\partial q^{\mu*}}\mu+\frac{\partial g^*}{\partial q^{g\mu*}}g\mu=g^*\frac{\partial (\omega q\nu+\lambda)}{\partial q^{\mu*}}\mu+\frac{\partial (\nu^*q^*\omega^*+\lambda^*)}{\partial q^{g\mu*}}g\mu\\
&=-\frac{1}{2}g^*\omega(\nu\mu)^*+\nu^*\mathfrak{R}(\omega^*g\mu)
\end{split}
\end{equation}

\begin{example}
Find the GHR derivative of the function $f:\mathbb{H}\rightarrow \mathbb{H}$ given by
\begin{equation}
f(q)=\frac{1}{\omega q\nu +\lambda}
\end{equation}
\end{example}
\textbf{Solution}: By applying the derivative operator to both sides of $1=(\omega q\nu +\lambda)f(q)$, we have
\begin{equation}
\begin{split}
0&=\frac{\partial ((\omega q\nu +\lambda)f)}{\partial q^{\mu}}\mu=(\omega q\nu +\lambda)\frac{\partial f}{\partial q^{\mu}}\mu+\frac{\partial (\omega q\nu +\lambda)}{\partial q^{f\mu}}f\mu\\
&=(\omega q\nu +\lambda)\frac{\partial f}{\partial q^{\mu}}\mu+\omega\mathfrak{R}(\nu f\mu)
\end{split}
\end{equation}
This is equivalent to
\begin{equation}
\frac{\partial f}{\partial q^{\mu}}\mu=-(\omega q\nu +\lambda)^{-1}\omega\mathfrak{R}(\nu f\mu)=-f\omega\mathfrak{R}(\nu f\mu)
\end{equation}
In a similar manner,
\begin{equation}
\begin{split}
0&=\frac{\partial ((\omega q\nu +\lambda)f)}{\partial q^{\mu*}}\mu=(\omega q\nu +\lambda)\frac{\partial f}{\partial q^{\mu*}}\mu+\frac{\partial (\omega q\nu +\lambda)}{\partial q^{f\mu*}}f\mu\\
&=(\omega q\nu +\lambda)\frac{\partial f}{\partial q^{\mu*}}\mu-\frac{1}{2}\omega(\nu f\mu)^*
\end{split}
\end{equation}
which is equivalent to
\begin{equation}
\frac{\partial f}{\partial q^{\mu*}}\mu=\frac{1}{2}(\omega q\nu +\lambda)^{-1}\omega(\nu f\mu)^*=\frac{1}{2}f\omega(\nu f\mu)^*
\end{equation}

\begin{example}
Find the GHR derivative of the function $f:\mathbb{H}\rightarrow \mathbb{H}$ given by
\begin{equation}
f(q)=\frac{\omega q\nu+\lambda}{|\omega q\nu+\lambda|}
\end{equation}
where $\nu,\omega,\lambda\in\mathbb{H}$ are constants.
\end{example}
\textbf{Solution}: Set $g(q)=\omega q\nu+\lambda$, then $g(q)=f(q)|g(q)|$. By using the product rule in Corollary \ref{cor:leftprodrl} and the result $\frac{\partial |\omega q\nu+\lambda|}{\partial q^{\mu}}$ in Table \ref{tb:fmlleftderiv}, we have
\begin{equation}\label{eq:unitdwqvdq1}
\begin{split}
\frac{\partial g}{\partial q^{\mu}}\mu&=\frac{\partial (f|g|)}{\partial q^{\mu}}\mu=f\frac{\partial |g|}{\partial q^{\mu}}\mu+\frac{\partial f}{\partial q^{\mu}}|g|\mu=f\frac{\partial |\omega q\nu+\lambda|}{\partial q^{\mu}}\mu+\frac{\partial f}{\partial q^{\mu}}\mu|g|\\
&=f\frac{g^*}{2|g|}\omega\mathfrak{R}(\nu\mu)-f\frac{1}{4|g|}\nu^*(\omega^*g\mu)^*+\frac{\partial f}{\partial q^{\mu}}\mu|g|\\
\end{split}
\end{equation}
From the result $\frac{\partial (\omega q\nu+\lambda)}{\partial q^{\mu}}$ in Table \ref{tb:fmlleftderiv},
\begin{equation}\label{eq:unitdwqvdq1b}
\frac{\partial g}{\partial q^{\mu}}\mu=\frac{\partial (\omega q\nu+\lambda)}{\partial q^{\mu}}\mu=\omega\mathfrak{R}(\nu\mu)
\end{equation}
By combining \eqref{eq:unitdwqvdq1} and \eqref{eq:unitdwqvdq1b}, we arrive at
\begin{equation}
\begin{split}
\frac{\partial f}{\partial q^{\mu}}\mu&=\frac{1}{|g|}\omega\mathfrak{R}(\nu\mu)-f\frac{g^*}{2|g|^2}\omega\mathfrak{R}(\nu\mu)+f\frac{1}{4|g|^2}\nu^*(\omega^*g\mu)^*\\
&=\frac{\omega}{|g|}\mathfrak{R}(\nu\mu)-\frac{\omega}{2|g|}\mathfrak{R}(\nu\mu)+\frac{f}{4|g|^2}\nu^*(\omega^*g\mu)^*=\frac{\omega}{2|g|}\mathfrak{R}(\nu\mu)+\frac{g}{4|g|^3}\nu^*(\omega^*g\mu)^*
\end{split}
\end{equation}
In a similar manner,
\begin{equation}
\begin{split}
&-\frac{1}{2}\omega(\nu\mu)^*=\frac{\partial (\omega q\nu+\lambda)}{\partial q^{\mu*}}\mu=\frac{\partial g}{\partial q^{\mu*}}\mu=\frac{\partial (f|g|)}{\partial q^{\mu*}}\mu=f\frac{\partial |g|}{\partial q^{\mu*}}\mu+\frac{\partial f}{\partial q^{\mu*}}|g|\mu\\
&=f\frac{\partial |\omega q\nu+\lambda|}{\partial q^{\mu*}}\mu+\frac{\partial f}{\partial q^{\mu*}}\mu|g|=-f\frac{g^*}{4|g|}\omega(\nu\mu)^*+\frac{f}{2|g|}\nu^*\mathfrak{R}(\omega^*g\mu)+\frac{\partial f}{\partial q^{\mu*}}\mu|g|\\
\end{split}
\end{equation}
which is equivalent to
\begin{equation}
\begin{split}
\frac{\partial f}{\partial q^{\mu*}}\mu
&=-\frac{1}{2|g|}\omega(\nu\mu)^*+f\frac{g^*}{4|g|^2}\omega(\nu\mu)^*-\frac{f}{2|g|^2}\nu^*\mathfrak{R}(\omega^*g\mu)\\
&=-\frac{1}{4|g|}\omega(\nu\mu)^*-\frac{g}{2|g|^3}\nu^*\mathfrak{R}(\omega^*g\mu)\\
\end{split}
\end{equation}

\begin{example}
Find the GHR derivative of the function $f:\mathbb{H}\rightarrow \mathbb{H}$ given by
\begin{equation}
f(q)=\mathfrak{R}(\omega q\nu+\lambda)
\end{equation}
where $\nu,\omega,\lambda\in\mathbb{H}$ are constants.
\end{example}
\textbf{Solution}: Set $g(q)=\omega q\nu+\lambda$. Applying the derivative operator to both sides of $f=\frac{1}{2}g+\frac{1}{2}g^*$, and using the results $\frac{\partial (\omega q\nu+\lambda)}{\partial q^{\mu}}$ and $\frac{\partial (\omega q^*\nu+\lambda)}{\partial q^{\mu}}$ in Table \ref{tb:fmlleftderiv}, we have
\begin{equation}
\begin{split}
\frac{\partial f}{\partial q^{\mu}}\mu&=\frac{1}{2}\frac{\partial g}{\partial q^{\mu}}\mu+\frac{1}{2}\frac{\partial g^*}{\partial q^{\mu}}\mu
=\frac{1}{2}\frac{\partial (\omega q\nu+\lambda)}{\partial q^{\mu}}\mu+\frac{1}{2}\frac{\partial (\nu^* q^*\omega^*+\lambda^*)}{\partial q^{\mu}}\mu\\
&=\frac{1}{2}\omega \mathfrak{R}(\nu\mu)-\frac{1}{4}\nu^*(\omega^*\mu)^*=\frac{1}{2}\left(\mathfrak{R}(\nu\mu)-\frac{1}{2}(\mu\nu)^*\right)\omega
=\frac{1}{4}\mu\nu\omega
\end{split}
\end{equation}
In a similar manner,
\begin{equation}
\begin{split}
\frac{\partial f}{\partial q^{\mu*}}\mu&=\frac{1}{2}\frac{\partial g}{\partial q^{\mu*}}\mu+\frac{1}{2}\frac{\partial g^*}{\partial q^{\mu*}}\mu
=\frac{1}{2}\frac{\partial (\omega q\nu+\lambda)}{\partial q^{\mu*}}\mu+\frac{1}{2}\frac{\partial (\nu^* q^*\omega^*+\lambda^*)}{\partial q^{\mu*}}\mu\\
&=-\frac{1}{4}\omega(\nu\mu)^*+\frac{1}{2}\nu^* \mathfrak{R}(\omega^*\mu)=\frac{1}{2}\left(-\frac{1}{2}\omega\mu^*+\mathfrak{R}(\omega^*\mu)\right)\nu^*=\frac{1}{4}\mu\omega^*\nu^*
\end{split}
\end{equation}

\affiliationone{
Dongpo Xu\\
College of Science, Harbin Engineering University, Harbin 150001, China
   \email{dongpoxu@gmail.com}}
\affiliationtwo{
Cyrus Jahanchahi and Danilo P. Mandic\\
   Department of Electrical and Electronic Engineering, Imperial College London, SW7
2AZ London, UK.
   \email{cyrus.jahanchahi05@imperial.ac.uk\\
   d.mandic@imperial.ac.uk}}
\affiliationthree{%
Clive C. Took\\
Department of Computing, University of Surrey, GU2 7XH Surrey, UK.
   \email{c.cheongtook@surrey.ac.uk}}

\begin{thebibliography}{99}
%
%
\bibitem{Hanson} {\bibname A. J. Hanson},
{\em Visualizing Quaternions}
(Morgan Kaufmann, San Francisco, CA, 2005).

\bibitem{Hjorungnes07} {\bibname A. Hj{\o}rungnes and D. Gesbert},
`Complex-Valued Matrix Differentiation: Techniques and Key Results',
{\em IEEE Trans. Signal Process. }55 (2007) 2740--2746.

\bibitem{Hjorungnes11} {\bibname A. Hj{\o}{\o}rungnes},
{\em Complex-Valued Matrix Derivatives: With Applications in Signal Processing and Communications}
(Cambridge Univ. Press, 2011).

\bibitem{Sudbery} {\bibname A. Sudbery},
`Quaternionic analysis',
{\em Math. Proc. Camb. Phil. Soc. }85 (1979) 199--225.

\bibitem{Widrow} {\bibname B. Widrow, J. McCool, and M. Ball},
`The complex LMS algorithm',
{\em Proc. of the IEEE. }63 (1975) 719--720.

\bibitem{Ujang11} {\bibname B. C. Ujang, C. C. Took, and D. P. Mandic},
`Quaternion valued nonlinear adaptive filtering',
{\em IEEE Trans. Neural Netw. }22 (2011) 1193--1206.

\bibitem{Ujang13} {\bibname B. C. Ujang, C. Jahanchahi, C. C. Took, and D. P. Mandic},
`Adaptive convex combination approach for the identification of improper quaternion processes',
{\em IEEE Trans. Neural Netw. Learn. Syst. }25 (2014) 1193--1206.

\bibitem{Deavours} {\bibname C. A. Deavours},
`The quaternion calculus',
{\em Amer. Math. Monthly. }80 (1973) 995--1008.

\bibitem{Took09} {\bibname C. C. Took and D. P. Mandic},
`The quaternion LMS algorithm for adaptive filtering of hypercomplex processes',
{\em IEEE Trans. Signal Process. }57 (2009) 1316--1327.

\bibitem{Took10a} {\bibname C. C. Took and D. P. Mandic},
`A quaternion widely linear adaptive filter',
{\em IEEE Trans. Signal Process. }58 (2010) 4427--4431.

\bibitem{Took10b} {\bibname C. C. Took and D. P. Mandic},
`Quaternion-valued stochastic gradientbased adaptive IIR filtering',
{\em IEEE Trans. Signal Process. }58 (2010) 3895--3901.

\bibitem{Took11} {\bibname C. C. Took and D. P. Mandic},
`Augmented second order statistics of quaternion random signals',
{\em Signal Process. }91 (2011) 214--224.

\bibitem{Jahanchahi12} {\bibname C. Jahanchahi, C. Cheong Took, and D. P. Mandic},
`On gradient calculation in quaternion adaptive filtering',
{\em Proc. IEEE Inter. Conf. on Acoustics, Speech, and Signal Processing. }(2012) 3773--3776.

\bibitem{Jahanchahi13} {\bibname C. Jahanchahi and D. P. Mandic},
`A Class of Quaternion Kalman Filters',
{\em IEEE Trans. Neural Netw. Learn. Syst. }in press, 2013.

\bibitem{Jahanchahi10} {\bibname C. Jahanchahi, C. Cheong Took, and D. P. Mandic},
`On HR calculus, quaternion valued stochastic gradient, and adaptive three dimensional wind
forecasting',
{\em Proc. IEEE Inter. Joint Conf. on Neural Networks. }(2010) 1--5.

\bibitem{Schwartz} {\bibname C. Schwartz},
`Calculus with a quaternionic variable',
{\em J. Math. Phys. }50 (2009) 1--11.

\bibitem{Brandwood} {\bibname D. Brandwood},
`A complex gradient operator and its application in adaptive array theory',
{\em IEEE Commun., Radar Signal Process. }130 (1983) 11--16.

\bibitem{Mandic11} {\bibname D. P. Mandic, C. Jahanchahi, and C. Cheong Took},
`A quaternion gradient operator and its applications',
{\em IEEE Sig. Proc. Letters. }18 (2011) 47--50.

\bibitem{Mandic09} {\bibname D. P. Mandic and S. L. Goh},
{\em Complex valued nonlinear adaptive filters: Noncircularity, widely linear and neural models}
(John Wiley and Sons Ltd, 2009).

\bibitem{Khong13} {\bibname D. Venkatraman, V. V. Reddy and A. W. Khong},
`On the use of the quaternion generalized Gaussian distribution for footstep detection',
{\em Proc. IEEE Inter. Conf. on Acoustics, Speech, and Signal Processing. }(2013) 6521--6525.

\bibitem{Colombo09} {\bibname F. Colombo, I. Sabadini, D. C. Struppa},
`Slice monogenic functions',
{\em Israel J. Math. }171 (2009), 385--403.

\bibitem{Colombo} {\bibname F. Colombo, I. Sabadini, D. C. Struppa},
{\em Noncommutative Functional calculus: Theory and Applications of Slice Hyperholomorphic Functions}
(Progr. Math., 289, Birkh\"{a}user, Basel, 2011).

\bibitem{Neto} {\bibname F. G. A. Neto and V. H. Nascimento},
`A novel reduced-complexity widely linear QLMS algorithm',
{\em Proc. IEEE Stat. Sig. Proc. Workshop. }(2011) 81--84.

\bibitem{Gentili06} {\bibname G. Gentili, D. C. Struppa},
`A new approach to Cullen-regular functions of a quaternionic
variable',
{\em C. R. Math. Acad. Sci. Paris. }342 (2006) 741--744.

\bibitem{Gentili07} {\bibname G. Gentili, D. C. Struppa},
`A new theory of regular functions of a quaternionic variable',
{\em Adv. Math. }216 (2007) 279--301.

\bibitem{Gentili09} {\bibname G. Gentili, C. Stoppato, D.C. Struppa and F. Vlacc},
`Recent developments for regular functions of a hypercomplex variable Hypercomplex Analysis',
{\em Trends Math., Birkh\"{a}user, Basel. }(2009) 165--186.

\bibitem{Gentili} {\bibname G. Gentili, C. Stoppato, D. C. Struppa},
{\em Regular Functions of a Quaternionic Variable}
(Springer-Verlag, Berlin, Heidelberg, 2013).

\bibitem{Ward} {\bibname J. P. Ward},
{\em Quaternions and Cayley Numbers: Algebra and Applications}
(Kluwer Academic. Publishers, London, 1997).

\bibitem{Moreno12} {\bibname J. Navarro-Moreno, R. M. Fern\'{a}ndez-Alcal\'{a} and J. C. Ruiz-Molina},
`A quaternion widely linear series expansion and its applications',
{\em IEEE Sig. Proc. Letters. }19 (2012) 868--871.

\bibitem{Moreno08} {\bibname J. Navarro-Moreno},
`ARMA prediction of widely linear systems by using the innovations algorithm',
{\em IEEE Trans. Signal Process. }56 (2008) 3061--3068.

\bibitem{Jvia10} {\bibname J. V\'{\i}a, D. Ram\'{\i}rez, and I. Santamar\'{\i}a},
`Properness and widely linear processing of quaternion random vectors',
{\em IEEE Trans. Inf. Theory. }56 (2010) 3502--3515.

\bibitem{Jvia11} {\bibname J. V\'{\i}a, and D. P. Palomar and L. Vielva},
`Generalized likelihood ratios for testing the properness of quaternion gaussian vectors',
{\em IEEE Trans. Signal Process. }59 (2011) 1356--1370.

\bibitem{Khalek} {\bibname K. Abdel-Khalek},
{\em Quaternion analysis}
(Dipartimento di Fisica - Universita de Lecce, Technical Report, 1996).

\bibitem{Kreutz} {\bibname K. Kreutz-Delgado},
{\em The complex gradient operator and the CR Calculus}
(Dept. Elect. Comput. Eng., Univ. California, San Diego, Tech.
Rep. ECE275A, 2006).

\bibitem{Shi} {\bibname L. Shi},
{\em  Exploration in quaternion colour M.S. thesis}
(Dept. Comput. Sci., Simon Fraser Univ., Burnaby, BC, Canada, 2005).

\bibitem{Luna} {\bibname M. E. Luna-Elizarrar\'{a}s, M. Shapiro},
`A survey on the (hyper-) derivatives in complex, quaternionic and Clifford analysis',
{\em Milan J. Math. }79 (2011) 521--542.

\bibitem{Bihan} {\bibname N. Le Bihan and J. Mars},
`Singular value decomposition of quaternion
matrices: A new tool for vector-sensor signal processing',
{\em Signal Process. }84 (2004) 1177--1199.

\bibitem{Bihan06} {\bibname N. Le Bihan and S. Buchholz},
{\em Optimal separation of polarized signals by quaternionic neural networks}
 (Proc. XIV EUSIPCO, Florence, Italy, 2006).

\bibitem{Bihan14} {\bibname N. Le Bihan, S. J. Sangwine and T. A. Ell},
`Instantaneous frequency and amplitude of orthocomplex modulated signals based on quaternion Fourier transform',
{\em Signal Process. }94 (2014) 308--318.

\bibitem{Girard} {\bibname P. R. Girard},
{\em Quaternions, Clifford Algebras and Relativistic Physics} (Birkh\"{a}user, Basel, Boston, Berlin, 2007).


\bibitem{Fueter34} {\bibname R. Fueter},
`Die Funktionentheorie der Differentialgleichungen $\Delta u = 0$ und $\Delta\Delta u = 0$ mit vier reellen
Variablen',
{\em Comment. Math. Helv. }7 (1934/5) 307--330.

\bibitem{Fueter35} {\bibname R. Fueter},
`\"{U}ber die analytische Darstellung der regul\"{a}ren Funktionen einer Quaternionenvariablen',
{\em Math. Helv. }8 (1935/6) 371--378.

\bibitem{Buchholz} {\bibname S. Buchholz and N. Le Bihan},
`Polarized signal classification by complex and quaternionic multi-layer perceptrons',
{\em Int. J. Neural Syst. }18 (2008) 75--85.


\bibitem{Leo} {\bibname S. De Leo and P. P. Rotelli},
`Quaternionic analyticity',
{\em App. Math. Lett. }16 (2003) 1077--1081.

\bibitem{Sangwine12} {\bibname S. J. Sangwine and T. A. Ell},
`Complex and Hypercomplex Discrete Fourier Transforms Based on Matrix Exponential Form of Euler's Formula',
{\em App. Math, Comput. }219 (2012) 644¨C655.

\bibitem{Sangwine11} {\bibname S. J. Sangwine, T. A. Ell and N. Le Bihan},
`Fundamental representations and algebraic properties of biquaternions or complexified quaternions',
{\em Adv, App. Clif. Alg. }21 (2011) 607¨C636.

\bibitem{Miron} {\bibname S. Miron, N. Le Bihan, and J. Mars},
`Quaternion-music for vector-sensor array processing',
{\em IEEE Trans. Signal Process. }54 (2006) 1218--1229.

\bibitem{Ell} {\bibname T. A. Ell and S. J. Sangwine},
`Quaternion involutions and anti-involutions',
{\em Comput, Math. Applicat. }53 (2007) 137--143.

\bibitem{Isokawa} {\bibname T. Isokawa, H. Nishimura, and N. Matsui},
`Quaternionic Multilayer Perceptron with Local Analyticity',
{\em Information. }3 (2012) 756--770.

\bibitem{Apostol} {\bibname T. M. Apostol},
{\em Calculus, Vol.I, 2nd ed } (John Wiley and Sons, Inc, 1967).

\bibitem{Wirtinger} {\bibname W. Wirtinger},
`Zur formalen theorie der funktionen von mehr komplexen ver\"{a}nderlichen',
{\em Mathematische Annalen. }97 (1927) 357--375.

\bibitem{Hamilton} {\bibname W. R. Hamilton},
`On quaternions',
{\em Proc. Royal Irish Acad. }11 (1844).

\end{thebibliography}
\end{document}